\pdfoutput=1 
\documentclass[11pt, a4paper, english]{amsart}
\usepackage{amsmath} 
\usepackage{amsthm} 
\usepackage{amssymb} 

\usepackage{amscd} 
\usepackage{mathtools}
\usepackage{subcaption}
\usepackage{array} 
\usepackage[backref=page,linktocpage]{hyperref} 
\usepackage{caption} %
\usepackage{graphics,graphicx} 
\usepackage{tikz,tikz-cd} 

\usetikzlibrary{graphs,graphs.standard,calc}

\usetikzlibrary{decorations.pathreplacing,}

\usepackage{enumerate} 
\usepackage{newtxtext,newtxmath} 
\usepackage[cal=boondoxo,scr=euler]{mathalfa}

\DeclareMathAlphabet{\mathsf}{OT1}{\sfdefault}{m}{n}

\newcommand{\nocontentsline}[3]{}
\newcommand{\tocless}[2]{\bgroup\let\addcontentsline=\nocontentsline#1{#2}\egroup}

\usepackage[margin=1.1in]{geometry}
\usepackage{verbatim}
\usepackage{xcolor}

\usepackage{scalerel}

\makeatletter
\@namedef{subjclassname@2020}{\textup{2020} Mathematics Subject Classification}
\makeatother

\DeclareMathAlphabet{\amathbb}{U}{bbold}{m}{n}

\newcommand{\mfrak}[1]{\scaleobj{0.93}{\mathfrak{#1}}}

\hypersetup{
    colorlinks = true,
    linkbordercolor = {white},
    linkcolor = {red},
    anchorcolor = {black},
    citecolor = {blue},
    filecolor = {cyan},
    menucolor = {red},
    runcolor = {cyan},
    urlcolor = {black}
}

\usetikzlibrary{automata}

\newtheoremstyle{teoremas}
{12pt}
{11pt}
{\slshape}
{}
{\bfseries}
{}
{.5em}
{}

\theoremstyle{teoremas}
\newtheorem{teo}{Theorem}[section]
\newtheorem{coro}[teo]{Corollary}
\newtheorem{cor}[teo]{Corollary}

\newtheorem{lemma}[teo]{Lemma}

\newtheorem{prop}[teo]{Proposition}

\newtheoremstyle{definition}
{11pt}
{11pt}
{}
{}
{\bfseries}
{}
{.5em}
{}

\theoremstyle{definition}
\newtheorem{defi}[teo]{Definition}
\newtheorem{conj}[teo]{Conjecture}

\newtheorem{problem}[teo]{Problem}

\newtheorem{example}[teo]{Example}
\newtheorem{remark}[teo]{Remark}

\DeclareMathOperator{\ehr}{ehr}

\DeclareMathOperator{\rk}{rk}
\DeclareMathOperator{\vol}{vol}
\DeclareMathOperator{\cover}{cover}

\DeclareMathOperator{\Rel}{Rel}

\DeclareMathOperator{\ann}{ann}

\newcommand{\M}{\mathsf{M}}
\newcommand{\Mat}{\operatorname{Mat}}
\newcommand{\Val}{\operatorname{Val}}

\newcommand{\N}{\mathsf{N}}
\newcommand{\U}{\mathsf{U}}

\newcommand{\A}{\mathrm{A}}
\newcommand{\uCH}{\underline{\mathrm{CH}}}
\newcommand{\uH}{\underline{\mathrm{H}}}
\newcommand{\LL}{\mathsf{\Lambda}}

\newcommand{\rank}{\operatorname{rk}}
\newcommand{\cl}{\operatorname{cl}}
\newcommand{\IH}{\mathrm{IH}}
\newcommand\size[1]{|#1|}

\AtBeginDocument{%
   \def\MR#1{}
}

\title{Valuative invariants for large classes of matroids}

\author[L.~Ferroni]{Luis Ferroni}
\author[B.~Schr\"oter]{Benjamin Schr\"oter}

\address{
  Department of Mathematics, KTH Royal Institute of Technology, Stockholm, Sweden
}
\email{\{ferroni,schrot\}@kth.se}

\thanks{The first author is supported by the Swedish
Research Council grant 2018-03968. 
The second author is supported by the Swedish
Research Council grant 2022-04224 and  thanks the Knut and Alice Wallenberg Foundation for the support.}

\keywords{Matroid polytopes, geometric lattices, valuations, matroid subdivisions.}

\subjclass[2020]{52B40, 05B35, 52B45, 13D40, 05C31, (14T20)}

\allowdisplaybreaks

\begin{document}

\begin{abstract}
    We study an operation in matroid theory that allows one to transition a given matroid into another with more bases via relaxing a \emph{stressed subset}. 
    This framework provides a new combinatorial characterization of the class of (elementary) split matroids. Moreover, it permits to describe an explicit matroid subdivision of a hypersimplex, which in turn can be used to write down concrete formulas for the evaluations of any valuative invariant on these matroids. 
    This shows that evaluations on these matroids depend solely on the behavior of the invariant on a tractable subclass of Schubert matroids. We address systematically the consequences of our approach for several invariants. They include the volume and Ehrhart polynomial of base polytopes, the Tutte polynomial, Kazhdan--Lusztig polynomials, the Whitney numbers of the first and second kind, spectrum polynomials and a generalization of these by Denham, chain polynomials and Speyer's $g$-polynomials, as well as Chow rings of matroids and their Hilbert--Poincar\'e series. The flexibility of this setting allows us to give a unified explanation for several recent results regarding the listed invariants; furthermore, we emphasize it as a powerful computational tool to produce explicit data and concrete examples.
\end{abstract}


\maketitle
\setcounter{tocdepth}{1}
\tableofcontents


\section{Introduction}
\noindent

\subsection{Overview} Matroid theory has received increasing attention in the past few years, due to deep connections with other branches of mathematics unveiled in recent work. The successful introduction of the K\"ahler package for the Chow ring of a matroid by Adiprasito, Huh and Katz \cite{adiprasito-huh-katz}, the study of the conormal Chow ring by Ardila, Denham and Huh \cite{ardila-denham-huh}, and the development of the intersection cohomology module of a matroid by Braden, Huh, Matherne, Proudfoot and Wang \cite{bradenhuh-semismall,bradenhuh} are examples of major recent developments. As a consequence of the existence of structures with such remarkable properties, it was possible to resolve problems in matroid theory that had remained open for half a century. See \cite{ardila-icm22,huh-icm22,eur} for thorough up-to-date expositions about recent progress in this area.

There are many mathematical objects that can be associated to a matroid. Often, they provide a cryptomorphic list of axioms that may in turn be used to define matroids. A problem that arises recurrently in matroid theory is the study of invariants. Due to the variety of approaches to matroid theory in various combinatorial frameworks, the number of settings in which invariants can be defined is vast. For example, in quite diverse contexts, starting with a matroid $\M$ one may consider the following invariants:
    \begin{itemize}
        \item The Tutte polynomial of $\M$ and its specializations.
        \item The Hilbert--Poincar\'e series of the Chow ring of $\M$.
        \item The Poincar\'e polynomial of the intersection cohomology module of $\M$.
        \item The Kazhdan--Lusztig--Stanley polynomial arising from the lattice of flats $\mathcal{L}(\M)$.
        \item The spectrum of the independence complex of $\M$.
        \item The Ehrhart polynomial of the base polytope $\mathscr{P}(\M)$.
        \item The chain polynomial of the lattice of flats $\mathcal{L}(\M)$.
    \end{itemize}
A remarkable fact is that although they are defined in dissimilar combinatorial, algebraic and geometric settings, the invariants of the preceding list share the property of being \emph{valuative}. This means that these invariants behave well under \emph{matroid subdivisions} of the base polytope $\mathscr{P}(\M)$, i.e., the convex hull of the indicator vectors of all bases. A matroid subdivision of $\mathscr{P}(\M)$ is a polyhedral complex all of whose cells are base polytopes of matroids, and whose union is $\mathscr{P}(\M)$. 

Proving that a given invariant is valuative can be a problem of varying difficulty. Some of the available tools to accomplish this task include the $\mathcal{G}$-invariant of Derksen and Fink \cite{derksen-fink}, and a result by Ardila and Sanchez \cite{ardila-sanchez} which shows that ``convolutions'' of valuations are again valuations. However, once we know that an invariant is valuative, it is reasonable to ask how to \emph{use} this property in practice.

In \cite{joswig-schroter} Joswig and Schr\"oter introduced a class of matroids, called ``split matroids'', which arise naturally in the study of the rays of the tropical Grassmannian and  subdivisions of the hypersimplex.
This class of matroids suggests the following paradigm: although the graphic matroid $\mathsf{K}_4$ coming from a complete graph on $4$ vertices has a polytope that is ``indecomposable'', it happens to be a split matroid and hence the valuations on it can be understood by considering an adequate subdivision of a hypersimplex in which $\mathscr{P}(\mathsf{K}_4)$ appears. In other words, instead of subdividing the polytope we are given, we may think of the opposite operation: starting from our polytope, try to ``add'' pieces corresponding to matroids that can be easily described, and in such a way that at the end we obtain a hypersimplex. More precisely, we want to find a reasonable subdivision of a hypersimplex in which one of the cells corresponds to our matroid, and the remaining cells are as easy as possible to describe. We will see this is possible to achieve for all \emph{elementary split matroids}, a subclass of split matroids that comprises all connected split matroids, but excludes some disconnected split matroids.

Yet another reason for which the study of (elementary) split matroids comes natural is because they constitute a well-structured and ``large'' class of matroids.  A long-standing conjecture, posed informally by Crapo and Rota \cite{crapo-rota} half a century ago, and later formalized by Mayhew, Newman, Welsh, and Whittle \cite{mayhew} asserts that the class of paving matroids is ``predominant'' (see Conjecture~\ref{conj:mayhewetal} below). The class of (elementary) split matroids is strictly larger than the class of paving matroids.

In the literature one often finds conjectures regarding invariants of matroids. Some well-known examples are the following:
    \begin{itemize}
        \item The Kazhdan--Lusztig polynomial of a matroid is always real-rooted \cite{gedeonsurvey}.
        \item The Ehrhart polynomial of a matroid always has positive coefficients \cite{deloera}.
        \item The Whitney numbers of the second kind form a unimodal sequence \cite{rota-conjecture}.
        \item The Speyer $g$-polynomial of a matroid always has non-negative coefficients \cite{speyer}.
        \item The Tutte polynomial of a matroid $\M$ without loops and coloops satisfies the inequality \cite{merino-welsh}. 
        \[T_{\M}(2,0)\, T_{\M}(0,2)\geq T_{\M}(1,1)^2\enspace .\]
    \end{itemize}

We note that all these statements follow a common pattern. Often, assertions of this type are made on the basis of computational data coming from matroids of small cardinality or small rank. Therefore, we are motivated to provide a systematic way of producing a significant amount of examples that perhaps can be used to test these and other similar conjectures following this pattern. 
For some invariants, even the class of paving matroids might be too restrictive; for example the Whitney numbers of the second kind or the Tutte polynomials of paving matroids are simple to describe. Hence, it is reasonable to question whether there is a way to describe the evaluation of an arbitrary valuative invariant on an arbitrary (elementary) split matroid. Hopefully, using a tool of this type one can \emph{disprove} a conjecture, \emph{support} it by proving that it holds for all (elementary) split matroids, or perhaps sharpen the view to find a counterexample that does not belong to this class. 
Two previous instances in which this approach has proved successful are the recent counterexamples found to the second of the conjectures above in \cite{ferroni3}, and the proof of the fifth of the above conjectures for all split matroids in \cite{ferroni-schroter}. Let us point out that a counterexample to the fifth conjecture has been found very recently in \cite{merino-welsh-wrong}.

The ultimate purpose of this paper is to show that, indeed, it is possible to describe in a compact way what the evaluation of an arbitrary valuative invariant on an arbitrary (elementary) split matroid looks like, and show how this can be used in practice to deduce new theorems, examples, or properties of well-known invariants. We believe that after a systematic approach to several invariants, it will become more apparent to the reader that the class of (elementary) split matroids is a reasonable choice that stays at a sufficiently good level of generality, that makes computations feasible, and which captures useful data that can be used to either provide evidence or construct potential counterexamples to conjectures and open problems.

\subsection{Outline}

The first of our contributions is establishing a language that allows us to describe the class of split matroids in a new and purely combinatorial way. The gadget we utilize is the notion of ``cover'' of a subset of the ground set of a matroid. Put succinctly, for a subset $A$ of a matroid $\M$ on $E$, we define $\cover(A) = \{S\subseteq E : |S| = \rk(\M)\, \text{ and } |S\cap A|\geq \rk(A) + 1\}$. This allows us to extend the terminology and results of \cite[Section 3]{ferroni-nasr-vecchi}. We combine this with the new and central notion of ``stressed subsets'' of a matroid. A subset $A\subseteq E$ of the ground set of a matroid $\M$ is said to be \emph{stressed} if both the restriction $\M|_A$ and the contraction $\M/A$ are uniform matroids. One of the motivations for our exposition is to highlight that it indeed generalizes and extends the constructions and results of \cite{ferroni-nasr-vecchi} by applying some insights of \cite{joswig-schroter}.

The following theorem, although elementary, lays a robust groundwork for our later constructions.

\begin{teo}
    Let $\M$ be a matroid on $E$ having set of bases $\mathscr{B}$. If $A$ is a stressed subset, then $\widetilde{\mathscr{B}} = \mathscr{B}\sqcup \cover(A)$ is the set of bases of a matroid $\widetilde{\M}$ on $E$.
\end{teo}

We call the resulting matroid $\widetilde{\M}$ a \emph{relaxation} of $\M$, and we denote it by $\Rel(\M,A)$. The above result is stated and proved as Theorem~\ref{thm:generalized-relaxation}. This operation generalizes the classical circuit-hyperplane relaxation, which explains the terminology we use. Moreover, unlike the relaxation of stressed hyperplanes studied in \cite{ferroni-nasr-vecchi}, it behaves well under dualization, and has reasonable counterpart operations in terms of the lattice of cyclic flats and, more importantly to our purposes, the base polytope. In Section \ref{sec:stressed-subsets-relaxations} we explore the interplay of this operation with other features of matroids. In a direct sum of two uniform matroids, the ground sets of each of the summands are stressed and therefore can be relaxed. For fixed integer numbers $r,k,h,n$, we define the matroid:
    \[ \LL_{r,k,h,n} = \Rel(\U_{k-r,n-h}\oplus \U_{r,h}, \U_{r,h}).\]
The indices $k$ and $n$ stand for the rank and the cardinality of $\LL_{r,k,h,n}$, whereas $r$ and $h$ are two additional parameters that correspond to the rank and the size of the subset we have relaxed to obtain $\LL_{r,k,h,n}$. We often refer to these matroids as ``cuspidal matroids''. They can be understood as the prototypical Schubert elementary split matroids (i.e., any matroid that is Schubert and elementary split is isomorphic to a cuspidal matroid). Being informal, these matroids are precisely the ``easy'' pieces that we add to build up a hypersimplex, starting from the base polytope $\mathscr{P}(\M)$ of an elementary split matroid~$\M$.

In Section \ref{sec:large-classes} we use the notions of stressed subsets and relaxations to prove the following characterization.

\begin{teo}
    A matroid $\M$ of rank $k$ and size $n$ is elementary split if and only if after relaxing all the stressed subsets of $\M$ one obtains the uniform matroid $\U_{k,n}$.
\end{teo}

The preceding result, stated and proved later as Theorem~\ref{thm:elem-split-relaxation-yields-uniform}, can be used to show in a conceptual way that the class of elementary split matroids contains all paving matroids and co-paving matroids. We further discuss certain enumerative aspects of these ``large'' classes of matroids.

We prove that the relaxation of stressed subsets on connected matroids induces a subdivision at the level of matroid polytopes that is particularly simple as it comes from a split.

\begin{teo}
    Let $\M$ be a connected matroid of rank $k$ and cardinality $n$, and assume that $A$ is a stressed subset with non-empty cover such that $\rk(A) = r$ and $|A|=h$. Consider the relaxed matroid $\widetilde{\M}=\Rel(\M, A)$. Then, there is a subdivision of $\mathscr{P}(\widetilde{\M})$ consisting of three internal faces. These faces correspond to matroids isomorphic to $\M$, $\LL_{r,k,h,n}$ and $\U_{k-r,n-h}\oplus \U_{r,h}$, respectively.
\end{teo}

This statement is reformulated in a more precise way in Theorem~\ref{thm:relaxation-induces-subdivision}. Continuing with the informal discussion that we initiated above, this implies that the relaxation of stressed subsets indeed provides a way of ``adding pieces'' to the base polytope in a controlled way. In particular, for a connected split matroid (or, more generally, an elementary split matroid), we can construct a subdivision of the hypersimplex having the property that all the remaining maximal pieces correspond to Schubert elementary split matroids. We prove the following general result, which is a direct consequence of the previous statements. This, along with the matroidal subdivision that witnesses its validity, constitutes arguably the main result from which all the other particularizations are derived.

\begin{teo}\label{thm:main-intro}
    Let $\M$ be an elementary split matroid of rank $k$ and cardinality $n$, and let $f$ be a valuative invariant. Then,
    \[ f(\M) = f(\U_{k,n}) - \sum_{r,h} \uplambda_{r,h} \left(f(\LL_{r,k,h,n}) - f(\U_{k-r,n-h}\oplus \U_{r,h})\right),  \]
    where each $\uplambda_{r,h}$ denotes the number of stressed subsets with non-empty cover of size $h$ and rank~$r$.
\end{teo}

This theorem is stated and proved as Theorem~\ref{thm:main}. Its power can be fully appreciated when one deals with a particular invariant $f$. In Section \ref{sec:polytopes} we present several results regarding matroid valuative functions, and tools to prove that an invariant is valuative. These are used in subsequent sections of the paper. 
Some subsections are a thorough overview of notions and tools that already appeared in the literature under various names. Then, in Sections~\ref{sec:seven} -- \ref{sec:nine} we discuss applications of Theorem~\ref{thm:main-intro} for various invariants. In some cases, we perform the full calculation for cuspidal matroids, hence obtaining a counterpart formula for all elementary split matroids. Whereas, in some other cases we just indicate the consequence of the main result for paving or sparse paving matroids --- although we believe that computations for cuspidal matroids are feasible, we have given priority to illustrate the method, instead of undergoing long calculations here.

In Section~\ref{sec:seven} we discuss volumes and Ehrhart polynomials; although we do not include new results, we put the emphasis on this representing a more unified framework. This allows us to recover, for example, the construction of matroids attaining negative Ehrhart coefficients in \cite{ferroni3}. We also address explicit formulas for the Tutte polynomials of elementary split matroids. We believe this formula will be relevant for mathematicians interested in Tutte polynomials. As a quick application of our framework, we settle negatively a conjecture by Matt Larson.

In Section~\ref{sec:eight} we prove that the flag $f$-vector of the lattice of flats is a valuative invariant. This allows us to deduce that the Whitney numbers of the second kind, the chain polynomials of the lattice of flats, and the Hilbert series of the Chow ring (among many other matroid functions) are valuative too. We discuss computational aspects for these invariants, motivated by conjectures by Rota \cite{rota-conjecture} on Whitney numbers of the second kind, Athanasiadis and Kalampogia-Evangelinou \cite{athanasiadis-kalampogia} on chain polynomials, and our own Conjecture~\ref{conj:hilbert-real-rooted} on Hilbert series of Chow rings, respectively. 

In Section~\ref{sec:nine} we discuss three families of invariants: first, the Kazhdan--Lusztig (KL) and $Z$-polynomials; second, the $g$-polynomial of Speyer; and third, the Denham polynomial and the spectrum polynomial. We prove a formula for the KL and $Z$-polynomials of arbitrary corank $2$ matroids and settle a conjecture of Gedeon for these matroids. Furthermore, we prove a formula for $g$-polynomials of paving matroids that extends a prior formula by Speyer for rank $2$ matroids, and we use it to show that the $g$-polynomial of a sparse paving matroid has non-negative coefficients, which supports the well known $f$-vector conjecture by Speyer. Finally, we use our framework to show the valuativeness of the spectrum polynomial and construct a pair of matroids that answer a question raised by Kook, Reiner, and Stanton.

\section{Preliminaries}
\noindent Throughout this article we will assume that the reader is familiar with the main terminology in matroid theory. For undefined concepts we refer to the books \cite{oxley} and \cite{welsh}. We will review some basic notions related to cyclic flats, Schubert matroids, lattice path matroids, and matroid polytope subdivisions.

\subsection{Cyclic flats}

The family of all cyclic flats of a matroid will play an important role in several parts of this paper. Let us start by recalling the definition of a cyclic set.

\begin{defi}
    Let $\M$ be a matroid on the ground set $E$.
    A set $A\subseteq E$ is called \emph{cyclic} if $A$ can be written as a union of circuits of $\M$ or, equivalently, the restriction of $\M$ to $A$, denoted $\M|_A$, has no coloops.
\end{defi}

Note that $F$ is a flat of the matroid $\M$ if and only if the contraction $\M/F$ has no loops. Dualizing, this is equivalent to $\M^*|_{E\smallsetminus F}$ not having any coloop, i.e., the set $E\smallsetminus F$ is cyclic in the dual of $\M$. In other words, the flats of $\M$ and the cyclic sets of $\M^*$ are in correspondence via taking complements. Therefore, the family of all the cyclic sets when ordered by inclusion is a lattice. If $F$ is a cyclic set in $\M$, then $\M|_F$ does not have coloops, so that its dual $\M^*/(E\smallsetminus F)$ is loopless, and hence $E\smallsetminus F$ is a flat of $\M^*$. In summary, this shows the following.

\begin{lemma}\label{lemma:cyclic-flat-dual}
    Let $\M$ be a matroid on $E$. The subset $F\subseteq E$ is a cyclic flat of the matroid $\M$ if and only if its complement $E\smallsetminus F$ is a cyclic flat of the dual matroid $\M^*$.
\end{lemma}

The data given by the ground set plus all the cyclic flats with their ranks suffice to determine the whole matroid. Sims \cite{sims:1980} and, independently, Bonin and De Mier \cite{bonin-demier-cyclic}, proved that this information provides indeed another cryptomorphic characterization of matroids. Let us recall the defining properties of the \emph{lattice of cyclic flats}.

\begin{teo}[{\cite[Theorem~3.2]{bonin-demier-cyclic}}]\label{thm:lattice-cyclic-flats}
    Let $\mathscr{Z}$ be a collection of subsets of a set $E$ and let $\rk:\mathscr{Z}\to \mathbb{Z}_{\geq 0}$. Then $\mathscr{Z}$ is the collection of cyclic flats of a matroid $\M$ on the ground set $E$ and the function $\rk$ is the restriction of the rank function of $\M$ to $\mathscr{Z}$ if and only if
    \begin{enumerate}
        \item[\normalfont(Z0)] $\mathscr{Z}$ is a lattice under inclusion.
        \item[\normalfont(Z1)] $\rk(0_{\mathscr{Z}}) = 0$ where $0_{\mathscr{Z}}$ is the bottom element of the lattice $\mathscr{Z}$.
        \item[\normalfont(Z2)] $0 < \rk(Y) - \rk(X) < |Y\smallsetminus X|$ for all sets $X,Y\in\mathscr{Z}$ such that $X\subsetneq Y$.
        \item[\normalfont(Z3)] For all sets $X,Y$ in $\mathscr{Z}$,
        \[ \rk(X) + \rk(Y) \geq \rk(X\vee_{\mathscr{Z}} Y) + \rk(X\wedge_{\mathscr{Z}} Y) + |(X\cap Y) \smallsetminus (X\wedge_{\mathscr{Z}} Y)|.\]
    \end{enumerate}
\end{teo}

Given the set $E$, the collection $\mathscr{Z}$ and the map $\rk:\mathscr{Z}\to\mathbb{Z}_{\geq 0}$ as above, the collection of bases of the matroid $\M$ is:
    \begin{equation} \label{eq:bases-from-cyclic-flats}
    \mathscr{B}(\M) = \left\{ B\in \binom{E}{k} : |X\cap B| \leq \rk(X) \text{ for all $X\in \mathscr{Z}$}\right\},
    \end{equation}
where $k = \rk(1_{\mathscr{Z}})$ is precisely the rank of the matroid, and $1_{\mathscr{Z}}$ denotes the maximum element of $\mathscr{Z}$. For a geometric interpretation of \eqref{eq:bases-from-cyclic-flats}, we refer to Remark~\ref{rem:polytope-cyclic-flats} below.

We use repeatedly that a matroid $\M$ on $E$ is loopless if and only if $\mathscr{Z}$ has the empty set as its bottom element, i.e., $0_{\mathscr{Z}} = \varnothing$. Analogously, $\M$ is coloopless if and only if $1_\mathscr{Z}=E$. The cyclic flats other than the top and bottom elements of $\mathscr{Z}(\M)$ are called \emph{proper}. Notice that Lemma~\ref{lemma:cyclic-flat-dual} implies that $\mathscr{Z}(\M^*)$ is the dual poset of $\mathscr{Z}(\M)$. 

\begin{lemma}\label{lemma:cyclic-flats-restriction}
    Let $\M$ be a loopless and coloopless matroid, and let $F$ be a cyclic flat of $\M$. The cyclic flats of the restriction $\M|_F$ are precisely the cyclic flats of $\M$ contained in $F$.
\end{lemma}

We refer to \cite[Corollary~1 \& Theorem~6]{cyclic-flats-binary} for details about this statement and more general versions of it. 
Also the next useful lemma can be found in the same article. 

\begin{lemma}[{\cite[Proposition~4]{cyclic-flats-binary}}]\label{lemma:cyclic-flats-uniform}
    Let $\M$ be a loopless and coloopless matroid on $E$. Then $\M$ is a uniform matroid if and only if $\mathscr{Z}(\M)=\{\varnothing, E\}$.
\end{lemma}

\subsection{Schubert matroids and lattice path matroids}

Schubert matroids are an important class of matroids that have appeared under various names in the literature.

\begin{defi}
    A matroid $\M$ is said to be a \emph{Schubert matroid} if its lattice of cyclic flats is a chain.
\end{defi}

See \cite[Section~4]{bonin-demier-structural} for a detailed discussion about appearances of this family of matroids in the literature. Let us emphasize our convention that the family of Schubert matroids is closed under isomorphisms.

Every Schubert matroid can be relabelled so that it becomes a lattice path matroid \cite{bonin-demier-structural}. Let us recall the basic properties of these matroids. Fix two non-negative integers $k \leq n$. We consider lattice paths in the plane $\mathbb{R}^2$, starting at the point $(0,0)$, ending at $(n-k,k)$, and consisting of precisely $n$ segments derived from steps of the form $+(1,0)$ or $+(0,1)$. For example, Figure~\ref{fig:lattice-path} illustrates two lattice paths for the parameters $k = 5$ and $n = 12$.

    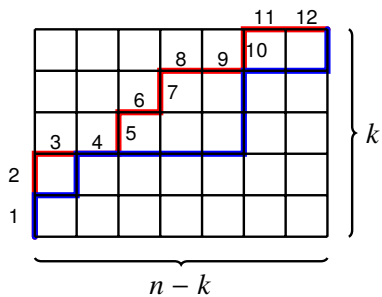
\begin{figure}[ht]
        \centering
        \begin{tikzpicture}[scale=0.55, line width=.9pt]
        \draw[line width=2.3pt,red,line cap=round] (0,0)--(0,2)--(2,2)--(2,2)--(2,3)--(3,3)--(3,4)--(5,4)--(5,5)--(6,5)--(7,5);
        \foreach \x/\y/\l in {-0.5/0.5/1, -0.5/1.5/2, 0.5/2.3/3, 1.5/2.3/4, 2.3/2.5/5, 2.5/3.3/6, 3.3/3.5/7, 3.5/4.3/8, 4.5/4.3/9, 5.3/4.5/10, 5.5/5.3/11, 6.5/5.3/12}
        	\node at (\x,\y) {\scriptsize\sf\l};
        \draw[line width=2.3pt,blue,line cap=round] (0,0)--(0,1)--(1,1)--(1,2)--(5,2)--(5,4)--(7,4)--(7,5);
        \draw (0,0) grid (7,5);
        \draw[decoration={brace,mirror, raise=8pt},decorate]
         (0,0) -- node[below=10pt] {$n-k$} (7,0);
        \draw[decoration={brace,mirror, raise=8pt},decorate]
         (7,0) -- node[right=10pt] {$k$} (7,5);
        \end{tikzpicture}
    \caption{Two lattice paths for $k=5$ and $n=12$. The blue path stays below the red path.}\label{fig:lattice-path}
    
    \end{figure}

It is common to denote each lattice path as a sequence of $n$ symbols, either N (``north'') or E (``east''), putting one of these two possibilities as the $i$-th symbol according to whether the $i$-th step is either $+(0,1)$ or $+(1,0)$ respectively. In Figure \ref{fig:lattice-path} the red path would be denoted as the sequence $P=\text{NNEENENEENEE}$. Another way of encoding a lattice path $P$ is by considering the set $s(P)$ of all the positions in which there is a N. For instance, for the path $P$ in Figure \ref{fig:lattice-path} is the set $s(P)=\{1,2,5,7,10\}$. 

For fixed $k$ and $n$ and consider two lattice paths $L$ and $U$. We say that $L$ \emph{is below} $U$ if for each $m=1,\ldots,n$ we have
    \begin{equation}\label{eq:ineq-lattice-paths}
        \left|s(L) \cap [m]\right| \leq \left| s(U) \cap [m]\right|\,.
    \end{equation}
This has a straightforward interpretation in the grid; see again Figure \ref{fig:lattice-path} where the path $L$ (colored blue) stays below the path $U$ (colored red). In this example $s(L) = \{1,3,8,9,12\}$ and $s(U) = \{1,2,5,7,10\}$.

Whenever the path $L$ is below $U$ we also say that $U$ is \emph{above} $L$. If furthermore the inequality in equation \eqref{eq:ineq-lattice-paths} is strict for all $m=1,\ldots,n-1$ then we say that $L$ is \emph{strictly below} $U$ or, equivalently, that $U$ is \emph{strictly above} $L$.

\begin{teo}[\cite{bonin-demier}]\label{thm:lattice_path_matroid}
    Let us fix $0\leq k\leq n$ and two lattice paths $L$ and $U$ such that $L$ is below $U$. Then the set
    \[ \mathscr{B}[L,U] := \{ s(P) : P\text{ is a lattice-path that is above $L$ and below $U$}\}\]
    is the set of bases of a matroid on $[n]$ of rank $k$. 
\end{teo}

The matroid of Theorem~\ref{thm:lattice_path_matroid} is denoted by 
$\M[L,U]$. We call such a matroid a lattice path matroid with \emph{lower path} $L$ and \emph{upper path} $U$. The fundamentals of the theory of lattice path matroids are discussed in \cite{bonin-demier} and \cite{bonin-demier-structural}.

As mentioned before, Schubert matroids are (up to isomorphism) a subclass of lattice path matroids. The next theorem characterizes Schubert matroids as lattice path matroids. 
For this purpose we use the notation $\M[U]:=\M[L,U]$ for lattice path matroids of rank $k$ on $n$ elements whose lower path $L=\text{\normalfont E\dots EN\dots N}$ is the lower-right border of the rectangular grid with vertices $(0,0)$, $(0,k)$, $(n-k,0)$ and $(n-k,k)$.

\begin{teo}\label{teo:schubert-is-lpm}
    A matroid $\M$ is a Schubert matroid if and only if it is isomorphic to a lattice path matroid $\M[U]$. 
\end{teo}

\begin{proof}
    This follows from \cite[Corollary 4.4]{bonin-demier-structural}.
\end{proof}

\subsection{Matroid polytopes and subdivisions}
Throughout the rest of this paper, the main polyhedral object that we associate to a matroid is its base polytope.

\begin{defi}
    Let $\M$ be a matroid on $E$. For each $i\in E$, denote by $e_i$ the $i$-th canonical vector in $\mathbb{R}^E$. The \emph{(base) polytope} of $\M$ is defined as
        \[ \mathscr{P}(\M):= \text{convex hull} \{e_B : B\in\mathscr{B}\} \subseteq \mathbb{R}^E,\]
    where $\mathscr{B}$ is the collection of bases of $\M$ and $e_B := \sum_{i\in B} e_i$ for each $B\in\mathscr{B}$.
\end{defi}

The dimension of $\mathscr{P}(\M)$ is always $n-c(\M)$ where $n$ is the size of the ground set of $\M$, and $c(\M)$ is the number of connected components of $\M$.

Consider the uniform matroid $\U_{k,n}$. Its base polytope is given by the convex hull of all $0/1$-vectors in $\mathbb{R}^n$ consisting of exactly $k$ ones and $n-k$ zeros. The resulting polytope is known as the \emph{hypersimplex}~$\Delta_{k,n}$. An inequality description of this polytope is given by the following slice of a cube:
        \[ \Delta_{k,n} = \mathscr{P}(\U_{k,n})=\left\{ x\in [0,1]^n : \sum_{i=1}^n x_i = k\right\}.\]

Matroid polytopes can be characterized as the polytopes whose vertices have coordinates $0$ and $1$, and whose edge directions are of the form $e_i-e_j$; see \cite{gelfand-goresky-macpherson-serganova}.

\begin{remark}\label{rem:polytope-cyclic-flats}
    An exterior description of $\mathscr{P}(\M)$ using the cyclic flats follows from equation~\eqref{eq:bases-from-cyclic-flats}:
    \[\mathscr{P}(\M) = \left\{x\in \Delta_{k,n} : \sum_{i\in F} x_i \leq \rk(F) \text{ for all $F\in\mathscr{Z}(\M)$}\right\}.\]
    If $\M$ is loopless, it is possible to further reduce the above inequality description by looking only at \emph{indecomposable} cyclic flats, i.e., those cyclic flats $F$ for which $\M|_F$ is a connected matroid. We refer to \cite[Theorem~40.5]{SchrijverB} for a more detailed discussion on the inequality description of both the independence polytope and base polytope of a matroid.
\end{remark}

Our next goal is to recall the fundamental notions on polyhedral subdivisions, and to establish the notation that we use subsequently throughout the paper. 

\begin{defi}
    Let $\mathscr{P}(\M)\subseteq \mathbb{R}^E$ be a matroid polytope. A \emph{matroid subdivision} of $\mathscr{P}$ is a finite collection $\mathcal{S}=\{\mathscr{P}_1,\ldots,\mathscr{P}_s\}$ of matroid polytopes $\mathscr{P}_i = \mathscr{P}(\M_i)\subseteq \mathbb{R}^E$ such that
    \begin{itemize}
        \item $\mathscr{P} = \bigcup_{i=1}^s \mathscr{P}_i$.
        \item Any face of a polytope $\mathscr{P}_i\in\mathcal{S}$ belongs to $\mathcal{S}$.
        \item For each $i\neq j$, the intersection $\mathscr{P}_i \cap \mathscr{P}_j$ is 
        a common (possibly empty) face of both $\mathscr{P}_i$ and $\mathscr{P}_j$.
    \end{itemize}
\end{defi}

Sometimes when we deal with a subdivision $\mathcal{S}$ of a matroid polytope $\mathscr{P}$ it is convenient to focus on the interior faces of the subdivision. More precisely, if we denote the boundary of $\mathscr{P}$ by $\partial\mathscr{P}$, we define
    \[ \mathcal{S}^{\operatorname{int}} := \left\{\mathscr{P}_i\in\mathcal{S}: \mathscr{P}_i\not\subseteq \partial\mathscr{P}\right\}.\]
Often we consider the set of all inclusion-wise maximal polytopes in $\mathcal{S}$. Therefore, we define
    \[ \mathcal{S}^{\max} := \left\{\mathscr{P}_i\in\mathcal{S}: \mathscr{P}_i\not\subseteq \mathscr{P}_j \text{ for $j\neq i$}\right\}.\]
We observe that the union of all the inclusion-wise maximal  faces of a subdivision is precisely $\mathscr{P}$, and moreover that all the interior faces are obtained as intersections of maximal faces. 

\begin{example}\label{ex:matroidal_subdivision}
    Consider the matroid $\U_{2,4}$. Its base polytope, i.e., the hypersimplex $\Delta_{2,4}$, is an octahedron in $\mathbb{R}^4$, lying in the hyperplane $x_1+x_2+x_3+x_4=2$. Consider also the two matroids $\M_1$ and $\M_2$ on $\{1,2,3,4\}$ with bases
        \begin{align*} 
            \mathscr{B}(\M_1) &= \{\{1,2\},\{1,3\},\{1,4\},\{2,3\},\{2,4\}\},\\
            \mathscr{B}(\M_2) &= \{\{1,3\},\{1,4\},\{2,3\},\{2,4\},\{3,4\}\},
        \end{align*}
    respectively. 
    The collection $\mathcal{S}$ consisting of all the faces of $\mathscr{P}(\M_1)$ and $\mathscr{P}(\M_2)$ is a matroid subdivision of the hypersimplex $\Delta_{2,4}$. Notice that the collection of interior faces $\mathcal{S}^{\operatorname{int}}$ consists of three polytopes. They are $\mathscr{P}(\M_1)$, $\mathscr{P}(\M_2)$ and their only common facet, $\mathscr{P}(\M_1)\cap \mathscr{P}(\M_2)$, as depicted in Figure \ref{fig:oct_subdivision}. Furthermore, the face $\mathscr{P}(\M_1)\cap \mathscr{P}(\M_2)$ is the base polytope of the matroid $\M_3$ whose bases are $\mathscr{B}(\M_3) = \{\{1,3\},\{1,4\},\{2,3\},\{2,4\}\}$.
    The matroid $\M_3$ is isomorphic to the direct sum $\U_{1,2}\oplus \U_{1,2}$.

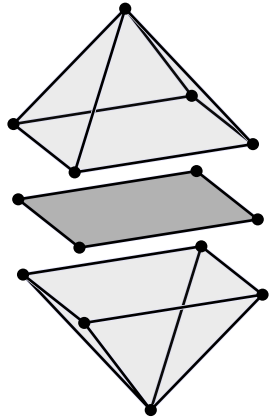
\begin{figure}[ht]
\centering
\begin{tikzpicture}[x  = {(0.9cm,-0.076cm)},
                    y  = {(-0.06cm,0.95cm)},
                    z  = {(-0.44cm,-0.29cm)},
                    scale = 1.75,
                    color = {black}]

  \coordinate (v3) at (1, 0, 0);
  \coordinate (v4) at (-1, 0, 0);
  \coordinate (v1) at (0, 1, 0);
  \coordinate (v2) at (0,-1, 0);
  \coordinate (v5) at (0, 0, 1);
  \coordinate (v6) at (0, 0,-1);
 
 \coordinate (t) at (0, -0.6, 0);
 \coordinate (p2) at ($(v2)+(t)$);
 \coordinate (p3) at ($(v3)+(t)$);
 \coordinate (p4) at ($(v4)+(t)$);
 \coordinate (p5) at ($(v5)+(t)$);
 \coordinate (p6) at ($(v6)+(t)$);

 \coordinate (q1) at ($(v1)-(t)$);
 \coordinate (q3) at ($(v3)-(t)$);
 \coordinate (q4) at ($(v4)-(t)$);
 \coordinate (q5) at ($(v5)-(t)$);
 \coordinate (q6) at ($(v6)-(t)$);

 \tikzstyle{linestyle} = [preaction={draw=blue!8, line cap=round, line join=round, line width=1.5 pt}, draw=black, line cap=round, line join=round, line width=1 pt, fill=gray, fill opacity=0.08];

  \draw[line width=1 pt, linestyle] (p2) -- (p6) -- (p3) -- cycle ;
  \draw[line width=1 pt, linestyle] (p2) -- (p6) -- (p4) -- cycle ;
  \draw[line width=1 pt, linestyle] (p2) -- (p5) -- (p3) -- cycle ;
  \draw[line width=1 pt, linestyle] (p2) -- (p5) -- (p4) -- cycle ;
  \draw[line width=1 pt, linestyle] (p3) -- (p5) -- (p4) -- (p6) -- cycle ;

  \draw[line width=1 pt, linestyle, fill opacity=0.6] ($.5*(p3)+0.5*(q3)$) -- ($.5*(p5)+0.5*(q5)$) --  ($.5*(p4)+0.5*(q4)$) --  ($.5*(p6)+0.5*(q6)$) -- cycle;

  \draw[line width=1 pt, linestyle] (q3) -- (q5) -- (q4) -- (q6) -- cycle ;
  \draw[line width=1 pt, linestyle] (q1) -- (q6) -- (q3) -- cycle ;
  \draw[line width=1 pt, linestyle] (q1) -- (q6) -- (q4) -- cycle ;
  \draw[line width=1 pt, linestyle] (q1) -- (q5) -- (q3) -- cycle ;
  \draw[line width=1 pt, linestyle] (q1) -- (q5) -- (q4) -- cycle ;

   \fill[fill=black] (p2) circle (1.3 pt);
   \fill[fill=black] (p3) circle (1.3 pt);
   \fill[fill=black] (p4) circle (1.3 pt);
   \fill[fill=black] (p5) circle (1.3 pt);
   \fill[fill=black] (p6) circle (1.3 pt);

   \fill[fill=black] ($.5*(p3)+0.5*(q3)$) circle (1.3 pt);
   \fill[fill=black] ($.5*(p4)+0.5*(q4)$) circle (1.3 pt);
   \fill[fill=black] ($.5*(p5)+0.5*(q5)$) circle (1.3 pt);
   \fill[fill=black] ($.5*(p6)+0.5*(q6)$) circle (1.3 pt);

   \fill[fill=black] (q1) circle (1.3 pt);
   \fill[fill=black] (q3) circle (1.3 pt);
   \fill[fill=black] (q4) circle (1.3 pt);
   \fill[fill=black] (q5) circle (1.3 pt);
   \fill[fill=black] (q6) circle (1.3 pt);
\end{tikzpicture}
        \caption{The (interior) faces in a matroid subdivision of $\Delta_{2,4}$.}\label{fig:oct_subdivision}
\end{figure}
\end{example}

Note that the subdivision $\mathcal{S}$ of $\U_{2,4}$ in Example \ref{ex:matroidal_subdivision} has the property that $|\mathcal{S}^{\max}| = 2$, i.e., there are exactly two maximal cells whose union yield our original matroid. Subdivisions of this type are called \emph{hyperplane splits} or just \emph{splits}. These subdivisions are obtained by cutting the polytope with two halfspaces that are separated by a single hyperplane.
The intersections of the two (closed) halfspaces with the polytope yield the two maximal cells of the split. It is automatic that these cells intersect each other along a common face. A combinatorial description of the cases in which splits do exist can be found in the work of Kim \cite{Kim}.

\subsection{Matroid valuations}

The notion of valuations on convex polytopes is central in discrete geometry; see for example \cite{volland,McMullen:1977,BetkeKneser:1985} for the foundations and \cite{JochemkoSanyal:2018} for newer developments. Many invariants of polytopes, such as the euclidean volume (cf. Example~\ref{ex-volume}), or the number of interior lattice points, are examples of valuations. Roughly speaking, a valuation is a map that respects  ``inclusion-exclusion'' under subdivisions, i.e., it can be computed from the value of the smaller pieces of the subdivision. 

From now on, we will consider that the ground set of any matroid is a subset of the positive integers. We introduce the following notations:
\[\begin{array}{lll}
    \mfrak{M}_{E,k} &:=& \,\left\{\mathscr{P}(\M) : \M \text{ is a matroid on $E$ of rank $k$}\right\},\\
    \mfrak{M}_E \;\; &:= &\,\left\{\mathscr{P}(\M) : \M \text{ is a matroid on $E$}\right\},\\
    \mathfrak{M} \;\; &:= &\,\displaystyle\bigsqcup_{\substack{E\subseteq \{1,2,\ldots\}\\ |E| < \infty}} \mathfrak{M}_E.
\end{array}    \]

\begin{defi}\label{def:weak-valuation}
    Let $\A$ be an abelian group. A \emph{weak valuation} is a map $f:\mfrak{M}\to \A$ such that for every matroid polytope $\mathscr{P}$ and every matroid subdivision $\mathcal{S}$ of $\mathscr{P}$,
    \begin{equation} \label{eq:weak-valuation}
    f(\mathscr{P}) = \sum_{\mathscr{Q}\in\mathcal{S}^{\operatorname{int}}} (-1)^{\dim(\mathscr{P}) - \dim(\mathscr{Q})} f(\mathscr{Q}).
    \end{equation}
\end{defi}

\begin{example}\label{ex-volume}
    Consider the map $f:\mfrak{M} \to \mathbb{R}$ defined by
        $f(\mathscr{P}) = \operatorname{vol}(\mathscr{P})$,
    where by $\operatorname{vol}$ we mean the euclidean volume. Observe that for any subdivision $\mathcal{S}=\{\mathscr{P}_1,\ldots,\mathscr{P}_s\}$ of a polytope $\mathscr{P}$, we will only need to consider the internal faces that are full-dimensional, because the volume of a lower dimensional polytope vanishes. Now, since the union of all the full dimensional internal faces covers the polytope~$\mathscr{P}$ up to some measure zero set exactly once, equation \eqref{eq:weak-valuation} holds.
\end{example}

Let us fix a matroid polytope $\mathscr{P}\in\mathfrak{M}$, a matroid subdivision $\mathcal{S}$ and a weak valuation $f:\mathfrak{M}\to \A$. For each non-empty subset $I\subseteq\mathcal{S}^{\max}$ denote by $\mathscr{P}_I$ the intersection of all the polytopes in $I$. Each polytope $\mathscr{P}_I$ belongs to $\mathcal{S}^{\operatorname{int}}$. In \cite[Theorem~3.5]{ardila-fink-rincon}, Ardila, Fink, and Rinc\'on proved that the map $f$ being a weak valuation is equivalent to the following condition on $f$:
    \begin{equation}\label{eq:weak-valuation-inclusion-exclusion}
    f(\mathscr{P}) = \sum_{\varnothing\neq I\subseteq \mathcal{S}^{\max}}(-1)^{|I|-1} f(\mathscr{P}_I).
    \end{equation}
In other words, equation \eqref{eq:weak-valuation-inclusion-exclusion} characterizes weak valuations. This explains more rigorously why we think of them as maps ``respecting'' the inclusion-exclusion principle. The proof of the equivalence with Definition~\ref{def:weak-valuation} follows from  a topological argument for computing Euler characteristics. For the details we refer the reader Ardila, Fink, and Rinc\'on's paper.

\begin{remark}\label{rem:constants-are-valuations}
    Constant functions are valuations. More precisely, the map $f:\mathfrak{M}\to \mathbb{Z}$ given by $f(\mathscr{P}(\M)) = 1$ for every matroid $\M$ is a valuation. This is a non-obvious fact, which was proved by Ardila, Fink, and Rinc\'on in \cite[Section~4]{ardila-fink-rincon}. The proof relies on the computation of the Euler characteristic of a suitable cell complex.
\end{remark}

The name ``weak'' valuation suggests that there is a notion of ``strong'' valuation, which should of course be more restrictive. Before providing the definition, let us set some notation. The \emph{indicator function} of a polytope $\mathscr{P}\subseteq\mathbb{R}^E$ is the map $\amathbb{1}_{\mathscr{P}}:\mathbb{R}^E \to \mathbb{Z}$ with
    \[\amathbb{1}_{\mathscr{P}}(x) := \begin{cases*}
        1 & \text{if $x\in\mathscr{P}$,}\\
        0 & \text{if $x\notin \mathscr{P}$}.
    \end{cases*}\]

\begin{defi}\label{def:strong-valuation}
    Let $\A$ be an abelian group. A \emph{strong valuation} is a map $f:\mfrak{M}\to \A$ satisfying:
        \[ \sum_{i=1}^s a_i f(\mathscr{P}_i)=0, \qquad \text{whenever} \qquad \sum_{i=1}^s a_i \amathbb{1}_{\mathscr{P}_i}(x) = 0, \qquad a_i\in \mathbb{Z}.\]
\end{defi}

Notice that, as opposed to the case of weak valuations, in the preceding definition we did not make any mentions of subdivisions. In \cite{derksen-fink} Derksen and Fink showed that strong and weak valuations on matroids are equivalent notions.

\begin{teo}[{\cite[Theorem~3.5]{derksen-fink}}]
    Let $\A$ be an abelian group. A map $f:\mfrak{M}\to \A$ is a weak valuation if and only if it is a strong valuation.
\end{teo}

It is because of this property that we will drop the words ``weak'' or ``strong'' in front of ``valuation''. For a useful source about other variations of the notion of valuation and a panoramic view of how they relate to one another, with a specific view on matroids, we refer the reader to \cite[Appendix~A]{eur-huh-larson}.

Let $\Mat_{k}(E)$ denote the free abelian group generated by all matroids $\M$ with ground set $E$ of rank~$k$. Consider the subgroup of $\Mat_k(E)$ given by all linear combinations $\sum a_i \M_i\in \Mat_{k}(E)$ satisfying that the function $\sum a_i \amathbb{1}_{\mathscr{P}(\M_i)}$ is identically zero. The quotient of $\Mat_{k}(E)$ by this subgroup is called the \emph{valuative group of matroids on $E$ of rank $k$}, and is denoted by $\Val_k(E)$.  We refer to \cite{eur-huh-larson} for more details about valuative groups.

The valuative group is canonically isomorphic to the integer span of all indicator functions of all the matroid polytopes for matroids with ground set $E$ and rank $k$, which we denote $\mathbb{I}_{\mathfrak{M}_{E,k}}$. The isomorphism $\Val_k(E)\cong \mathbb{I}_{\mathfrak{M}_{E,k}}$ is given by the map $\M \mapsto \amathbb{1}_{\mathscr{P}(\M)}$.

These groups are freely generated. In fact, the following is an appealing result by Derksen and Fink, proved in a different way in \cite[Corollary~7.9]{eur-huh-larson}. 

\begin{teo}[{\cite[Theorem~5.4]{derksen-fink}}]\label{thm:schubert-basis}
    The set $\{ \amathbb{1}_{\mathscr{P}(\M)} : \M \text{ is a Schubert matroid on $E$ of rank $k$}\}$ is a basis of the free abelian group $\mathbb{I}_{\mfrak{M}_{E,k}}$.
\end{teo}

This implies that any valuation $f:\mathfrak{M}\to \A$ is determined by its values on Schubert matroids uniquely. We prove a related result in the Appendix~\ref{appendix-series-parallel}, which to the best of our knowledge has not appeared before in the literature: we show that the direct sums of series-parallel matroids on $E$ of rank $k$ are a spanning set for  the valuative group $\Val_k(E)$. Equivalently, the indicator functions of matroid polytopes of matroids on $E$ that have rank $k$ and are direct sums of series-parallel matroids span the group $\mathbb{I}_{\mathfrak{M}_{E,k}}$.

\subsection{Valuative invariants}\label{subsec:valinv}

A map $f:\mathfrak{M}\to \A$ is an \emph{invariant} if $f(\mathscr{P}(\M))=f(\mathscr{P}(\N))$ whenever $\M\cong \N$, i.e., if they are isomorphic matroids. If in addition $\A$ is an abelian group and the function $f$ is a valuation, we say that $f$ is a \emph{valuative invariant}. To lighten our notation, we will write $f(\M)$ whenever we have a map $f:\mfrak{M}\to \A$, instead of $f(\mathscr{P}(\M))$.

\begin{example}\label{ex:rank-function}
    Consider the maps $f,g:\mathfrak{M}\to \mathbb{R}[t]$ given by $f(\M) = t^{\rk(\M)}$ and $g(\M) = t^{\rk(\M^*)}$. Clearly, they are invariants of matroids. Again, it is a non-obvious fact that they are valuations. Notice that if we fix a matroid $\M$ and a subdivision $\mathcal{S}$ of $\mathscr{P}(\M)$ in $\mathfrak{M}$, then all the matroids corresponding to the polytopes in $\mathcal{S}$ have the same ground set and rank. In particular, $f$ and $g$ are constant on all interior faces in $\mathcal{S}^{\operatorname{int}}$. The valuativeness follows then from the arguments mentioned in Remark~\ref{rem:constants-are-valuations}.
\end{example}

A reasonable question that might arise at this point is what techniques exist to prove that an invariant $f:\mathfrak{M}\to \A$ is valuative. An important tool is the following invariant of matroids, which was addressed first by Derksen in \cite{derksen} and later studied by Derksen and Fink in \cite{derksen-fink}. 

\begin{defi}\label{def:ginvar}
    Let $\M$ be a matroid on $E$, and $|E|=n$. Consider $\mathcal{C}_E$ the set of all the $n!$ complete chains of subsets of $E$ of the form $C=\{\varnothing = S_0 \subsetneq S_1 \subsetneq \cdots \subsetneq S_n = E\}$.  To each chain $C$ as before, let us associate the vector $s(C)\in \{0,1\}^n$ defined by $s(C)=(s_1,\ldots,s_n)$ where $s_i=\rank_\M(S_i)-\rank_\M(S_{i-1})$ for $1 \leq i\leq n$.
    Let $\mathbf{U}$ be the free abelian group spanned by the $2^n$ symbols $U_{s}$ for $s\in \{0,1\}^n$. We define the map $\mathcal{G}:\mfrak{M}_E\to \mathbf{U}$ by
        \[ \mathcal{G}(\M) = \sum_{C\in \mathcal{C}_E} U_{s(C)}.\]
\end{defi}

This map is called the \emph{$\mathcal{G}$-invariant}. A key feature of the $\mathcal{G}$-invariant is that it satisfies a universality property, in the sense that it is valuative and every other valuative invariant factors through $\mathcal{G}$. Let us state this precisely.

\begin{teo}[{\cite[Theorem~1.4]{derksen-fink}}]\label{thm:g-invariant-universal}
    The map $\mathcal{G} : \mathfrak{M}_E \to \mathbf{U}$ is a valuative invariant. Moreover, it is universal. In other words, if $f:\mathfrak{M}_E\to \A$ is a valuative invariant, then there exists a (necessarily unique) homomorphism of abelian groups $\widetilde{f}:\mathbf{U}\to \A$ such that
        \[ f = \widetilde{f}\circ \mathcal{G}.\]
\end{teo}

\begin{example}
    For the matroids of Example \ref{ex:matroidal_subdivision}, we have
    \begin{align*}
        \mathcal{G}(\U_{2,4}) &= 24\, U_{(1,1,0,0)},\\
        \mathcal{G}(\M_1) = \mathcal{G}(\M_2) &= 20\, U_{(1,1,0,0)} + 4\, U_{(1,0,1,0)}, \\
        \mathcal{G}(\M_3) &= 16\, U_{(1,1,0,0)} + 8\, U_{(1,0,1,0)}.
    \end{align*}
    Observe that, indeed, $\mathcal{G}(\U_{2,4}) = \mathcal{G}(\M_1)+\mathcal{G}(\M_2) - \mathcal{G}(\M_3)$, as one should expect.
\end{example}

The $\mathcal{G}$-invariant is often useful to conclude the valuativeness of other invariants, due to the fact that it is defined in a purely combinatorial way, without relying on the geometry of the polytope. As a quick example, notice that the number of bases of the matroid is encoded in $\mathcal{G}$ by looking at the coefficient of $U_{(1,\ldots,1,0,\ldots,0)}$, where in the subindex there are $k$ ones followed by $n-k$ zeros, and then dividing that number by $k!(n-k)!$. In particular, this says that the map $\M\mapsto |\mathscr{B}(\M)|$ is a valuative invariant---though, of course, this can be proved in many other ways.


\section{Stressed subsets and relaxations}\label{sec:stressed-subsets-relaxations}
\noindent The purpose of this section is to introduce a new operation in matroid theory that extends the classical circuit-hyperplane relaxation. We show how this operation interacts with the lattice of cyclic flats and, more importantly, with the base polytope of the original matroid.

\subsection{The cover of a subset}

The first construction associates to each set in a matroid a family of non-bases. Under the special circumstances explained in Theorem~\ref{thm:generalized-relaxation}, these non-bases will admit to be added to the set of bases of $\M$ giving place to a new matroid. 

\begin{defi}
    Let $\M$ be a matroid of rank $k$ on the ground set $E$ and let $A$ be a subset of $E$. We define the \emph{cover}\footnote{In a previous version of this manuscript we called the \emph{cover} the \emph{cusp of a set}.} of $A$ in $\M$ to be the set
        \[ \cover(A) := \left\{S\in \binom{E}{k} : |S\cap A| \geq \rk(A) + 1\right\}.\]
\end{defi}

When dealing with several matroids at the same time, we will use the more precise notation $\cover_{\M}(A)$. Notice that, indeed, for every set $A$, the family $\cover(A)$ and the collection of bases $\mathscr{B}$ of $\M$ are disjoint. This is because $\rk(A) = \max_{B\in\mathscr{B}} |B\cap A|$, and hence $|B\cap A| \leq \rk(A)$ whenever the set $B$ is a basis. Our terminology is borrowed from \cite{pendavingh-vanderpol-flatcover}, where a non-basis $S$ is said to be \emph{covered} by the flat $F$ in a matroid $\M$ if $|S\cap F|\geq \rk_{\M}(F) + 1$, i.e., if $S\in \cover(F)$. In other words, the cover of $F$ is exactly the set of non-bases covered by $F$. 

For a matroid $\M$ of rank $k$ and $n$ elements, we can describe $\cover(A)$ geometrically as follows. Consider the hypersimplex $\Delta_{k,n}$. The elements of $\cover(A)$ correspond to the vertices of $\Delta_{k,n}$ that lie ``beyond'' the hyperplane defined by the inequality $\sum_{i\in A} x_i \leq \rk(A)$. In other words, $\cover(A)$ is always obtained via splitting a hypersimplex with a suitable hyperplane. We will often rely on this geometric description as in some cases it leads to shorter and more conceptual proofs.

There is a useful expression that allows us to express the size of the cover of a set $A$ in a matroid~$\M$ in terms of the following numerical data: the size and rank of the matroid~$\M$ and the size and the rank of the subset $A$.

\begin{prop}\label{prop:size-of-cusp}
    Let $\M$ be a matroid of rank $k$ and cardinality $n$, and let $A$ be a subset of the ground set of $\M$ which has size $h$ and rank $r$. Then
        \[ \left|\cover(A)\right| = \sum_{i=r+1}^k \binom{h}{i}\binom{n-h}{k-i}.\]
\end{prop}

\begin{proof}
    By definition, the members in the cover of $A$ have size $k$. Moreover, they are precisely those sets which  contain $i\geq \rk(A) +1 = r + 1$ elements lying in $A$ and $k - i$ elements in the complement of $A$. In other words, we have
    \[
        \left|\cover(A)\right| = \sum_{i = r + 1}^k \binom{|A|}{i} \binom{|E\smallsetminus A|}{k-i}
        = \sum_{i=r+1}^k \binom{h}{i}\binom{n-h}{k-i}. \qedhere
    \]
\end{proof}

There are some instances in which the cover of a set is empty. For example, if $A$ is independent, we have $|A|=\rk(A)$, so the condition $|S\cap A|\geq \rk(A)+1$ becomes impossible to achieve, as $|S\cap A| \leq |A| = \rk(A)$. The following result characterizes the only scenarios in which the cover of a subset can be empty.

\begin{prop}\label{prop:cusp-empty}
    Let $\M$ be a matroid on $E$ with bases $\mathscr{B}$ and $A\subseteq E$. Then the following are equivalent:
    \begin{enumerate}[\normalfont(i)]
        \item $\cover(A) = \varnothing$.
        \item Either $A$ is independent in $\M$ or $E\smallsetminus A$ is independent in $\M^*$.
    \end{enumerate}
\end{prop}

\begin{proof}
    Assume that $\cover(A)=\varnothing$. Then the sum in Proposition~\ref{prop:size-of-cusp} is zero.
    Following the notation of that proposition, this sum vanishes if and only if either $\rk(A)+1 = r+1 > k = \rk(\M)$ or the first binomial factor is always zero, i.e., $\rk(A)+1=r+1 > h = \size{A}$.
    Hence, either $\rk(A) \geq \rk(\M)$ which is equivalent to $\rk(A) = \rk(\M)$; or $\rk(A) \geq |A|$ which is equivalent to $\rk(A) = |A|$. We conclude that either $A$ has full rank (thus $E\smallsetminus A$ is independent in $\M^*$), or it is an independent set.
    
    Conversely, if $A$ is independent, we have $\rk(A) = |A|$ and hence there is no set $S$ satisfying the inequality $|S\cap A| \geq \rk(A)+1 = |A|+1$ as this would contradict that $S\cap A\subseteq A$; it follows that the cover of $A$ is empty. Otherwise, if $E\smallsetminus A$ is an independent set in the dual matroid $\M^*$, we have that $\rk(A) = \rk(\M)$. Similar to the case above the condition $|S\cap A| \geq \rk(A) + 1 = \rk(\M) + 1$ implies that $S$ is at least of size $\rk(\M)+1 = k+1$ excluding that $S$ is an element in $\cover(A)$. Therefore it follows again that the cover must be empty.
\end{proof}

Recall that $E\smallsetminus A$ is independent in $\M^*$ if and only if $A$ is a spanning subset in $\M$, i.e., $\cl_{\M}(A) = E$. We often rely on the preceding statement. In such cases, we indicate explicitly that a cover is non-empty---which means that the set $A$ is neither independent nor spanning.

\begin{remark}\label{rem:cusp-corank1}
    If $A$ is a subset such that $\rk(A) = k - 1$ in a matroid $\M$ of rank $k$, then $\cover(A) = \binom{A}{k}$.
    By definition, we have $\cover(A) = \left\{S\in \binom{E}{k}: |S\cap A| \geq k\right\}$. Any set $S$ in this family satisfies that $k = |S| \geq |S\cap A| \geq k$. In particular, $S\cap A = S$, which implies that $S\subseteq A$ and hence we find $\cover(A) = \binom{A}{k}$, as  claimed.
    In particular, if $A$ is a hyperplane, its cover is exactly the family of all $k$-subsets of $A$.
\end{remark}

A convenient feature is that covers behave well when dualizing a matroid. Whenever $\M$ is a matroid on $E$, we will write $\cover^*(A)$ to denote the cover of $A\subseteq E$ in the dual matroid $\M^*$.

\begin{prop}\label{prop:cusp-dual}
    If $\M$ is a matroid on $E$, then, for each $A\subseteq E$, the following holds:
    \[ S\in \cover^*(A) \iff E\smallsetminus S \in \cover(E\smallsetminus A).\]
\end{prop}

\begin{proof}
    Assume that $\M$ is a matroid on $n$ elements and rank $k$. From the geometric description, we see that $\cover_{\M^*}(A)$ corresponds to the vertices of $\Delta_{n-k,n}$ which violate the inequality $\sum_{i\in A} x_i \leq \rk_{\M^*}(A)$. By considering the involution $x_i\mapsto 1 - x_i$ for all coordinates in $\mathbb{R}^n$, we have that $\Delta_{n-k,n}$ maps onto $\Delta_{k,n}$, whereas $\mathscr{P}(\M^*)$ maps onto $\mathscr{P}(\M)$. The vertices of $\Delta_{n-k,n}$ satisfying $\sum_{i\in A} x_i > \rk_{\M^*}(A)$ map onto the vertices of $\Delta_{k,n}$ satisfying $\sum_{i\notin A} x_i > \rk(E\smallsetminus A)$, i.e., precisely the elements of $\cover_{\M}(E\smallsetminus A)$.
\end{proof}

\subsection{Relaxation of stressed subsets}

Our next goal is to introduce a new operation that extends the established ``circuit-hyperplane relaxation''. The following statement can be found for example in \cite[Proposition~1.5.14]{oxley}.

\begin{prop}\label{prop:ch-relaxation}
    Let $\M$ be a matroid on $E$ with set of bases $\mathscr{B}$. If $\M$ has a circuit-hyperplane $H$, then the collection $\widetilde{\mathscr{B}} = \mathscr{B} \sqcup \{H\}$ is the set of bases of a matroid $\widetilde{\M}$ on $E$.
\end{prop}

Our aim is to develop a framework in which the collection of bases of a matroid $\M$ can be enlarged by including non-bases of $\M$ in a controllable fashion. We begin with the introduction of the terminology.

\begin{defi}\label{def:stressed-subset}
    Let $\M$ be a matroid with ground set $E$. A subset $A\subseteq E$ is called \emph{stressed} if both the restriction $\M|_A$ and the contraction $\M/A$ are uniform matroids.
\end{defi}

Choosing the name ``stressed subset'' stems from the notion of ``stressed hyperplane'' introduced by Ferroni, Nasr and Vecchi in \cite{ferroni-nasr-vecchi}. In that paper, a hyperplane $H$ in a matroid $\M$ is said to be stressed if all the subsets of $H$ having size $\rk(\M)$ are circuits; in particular, a circuit-hyperplane is always a stressed hyperplane.
The definition that we gave above does in fact extend this concept, as we shall explain.
Observe that for every matroid $\M$ and every hyperplane $H$, the contraction $\M/H$ is isomorphic to a uniform matroid of the form $\U_{1,n-h}$ where $n$ is the cardinality of $\M$ and $h=|H|$. Therefore, according to Definition~\ref{def:stressed-subset} for a hyperplane to be stressed one only has to verify that $\M|_H$ is uniform; that is equivalent to requiring that $\M|_H\cong \U_{k-1,h}$ where $k=\rk(\M)$, i.e., that all the subsets of size $k-1$ of $H$ are independent in $\M$ or, more explicitly, that all the subsets of $H$ of size $k$ are circuits.

Clearly, $A$ is a stressed subset of the matroid $\M$ if and only if its complement is a stressed subset of the dual $\M^*$, because $\M^*|_{E\smallsetminus A} = (\M/A)^*$, $\M^*/(E\smallsetminus A) = (\M|_A)^*$ and the dual of a uniform matroid is a uniform matroid.

When we deal with a matroid $\M$ we usually denote by $k$ its rank and by $n$ the cardinality of its ground set. When $A$ is a stressed subset, we use the letters $r$ and $h$ to denote its rank $\rk(A)$ and size $|A|$, respectively. Notice that in this case we have
    \begin{equation}\label{eq:which-uniforms}
    \M|_A \cong \U_{r,h} \quad \text{and} \quad \M/A \cong \U_{k-r,n-h}\enspace.\end{equation}

\begin{example}\label{example:O7}
    Consider the matroid $\mathsf{O}_7$ depicted in Figure \ref{fig:matroidO7}; see \cite[p.~644]{oxley} for a list of properties this matroid possesses. It has rank $3$ and size $7$. If we consider the set $X = \{3,6,7\}$, we have that $\M|_X \cong \U_{2,3}$, whereas $\M/X \cong \U_{1,4}$. This says that $X$ is a stressed subset. In fact, $X$ is a circuit-hyperplane and therefore $\cover(X) = \{X\}$. On the other hand, if we consider $Y = \{1,2,3,4\}$ we obtain that $\M|_Y \cong \U_{2,4}$ and $\M/Y \cong \U_{1,3}$; in this case $\cover(Y) = \{\{1,2,3\}, \{1,3,4\},\{1,2,4\},\{2,3,4\}\}$, since $Y$ is a hyperplane (cf. Remark~\ref{rem:cusp-corank1}).
\end{example}

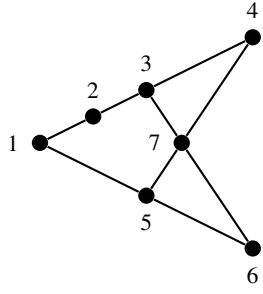
\begin{figure}[ht]
    \centering
	\begin{tikzpicture}  
	[scale=0.7,auto=center,every node/.style={circle,scale=0.8, fill=black, inner sep=2.7pt}] 
	\tikzstyle{edges} = [thick];
	
	\node[label=left:$1$] (a1) at (0,0) {};  
	\node[label=above:$2$] (a2) at (2/2,1/2)  {};
	\node[label=above:$3$] (a3) at (4/2,2/2)  {};  
	\node[label=above:$4$] (a4) at (8/2,4/2) {};
	\node[label=below:$5$] (a5) at (4/2,-2/2)  {};  
	\node[label=below:$6$] (a6) at (8/2,-4/2)  {};    
	\node[label=left:$7$] (a7) at (5.33/2,0) {};
	
	\draw[edges] (a1) -- (a2); 
	\draw[edges] (a2) -- (a3);  
	\draw[edges] (a3) -- (a4);  
	\draw[edges] (a1) -- (a5);  
	\draw[edges] (a5) -- (a6);
	\draw[edges] (a7) -- (a3);
	\draw[edges] (a7) -- (a4);
	\draw[edges] (a7) -- (a5);
	\draw[edges] (a7) -- (a6);
	\end{tikzpicture} \caption{A visualization of the matroid $\mathsf{O}_7$.}\label{fig:matroidO7}
\end{figure}

The next proposition establishes a useful characterization of the stressed subsets with non-empty cover, at least in the case in which our matroid $\M$ does not contain any loops and coloops. In particular, these stressed sets are cyclic flats that necessarily satisfy an additional property.

\begin{prop}\label{prop:cyclic-covering-iff-stressed}
    Let $\M$ be a matroid on $E$ without loops and coloops and let $\mathscr{Z}$ be its lattice of cyclic flats. For every subset $A$ of the matroid, the following are equivalent.
    \begin{enumerate}[\normalfont(a)]
        \item\label{it:cyclic-covering_a} $A$ is a stressed subset with non-empty cover.
        \item\label{it:cyclic-covering_b} The set $\{\varnothing, A, E\}$ forms a saturated chain in $\mathscr{Z}(\M)$.
    \end{enumerate}
\end{prop}

\begin{proof}
    Let us prove first that \eqref{it:cyclic-covering_a} $\Rightarrow$ \eqref{it:cyclic-covering_b}. Assuming that $A$ is a stressed set implies $\M|_A\cong \U_{r,h}$ and $\M/A\cong \U_{k-r,n-h}$ where $r$, $h$, $k$ and $n$ are chosen as in \eqref{eq:which-uniforms}. Additionally we assume that $\cover(A)\neq\varnothing$, and thus $A$ is neither an independent nor a spanning set. In other words, this means that $r < h$ and $r < k$ respectively. 
    Since we have that $\M/A\cong \U_{k-r,n-h}$ and $r<k$, it follows that $\M/A$ does not have loops, i.e., $A$ is a flat. The set $A$ is a union of circuits since $\M|_A \cong \U_{r,h}$ does not contain coloops as $r < h$. Since $\M$ is loopless and coloopless we have that $0_{\mathscr{Z}}=\varnothing$ and that $1_{\mathscr{Z}}=E$. 
    Now suppose that additionally there is a cyclic flat $F$ such that $\varnothing \subsetneq F\subsetneq A$. 
    Then $0 <\rk(F) < \rk(A) = r$, and since $\M|_A\cong \U_{r,h}$ it follows that $\rk(F) = |F| > 0$, i.e., $F$ is an independent set, contradicting that $F$ is a non-empty cyclic flat.
    This proves that $A$ indeed covers $0_{\mathscr{Z}}$. The proof that $1_{\mathscr{Z}}$ covers $A$ is entirely analogous. This can be shown by applying  the same arguments to the dual matroid $\M^*$.
    
    Now let us prove that \eqref{it:cyclic-covering_b} $\Rightarrow$ \eqref{it:cyclic-covering_a}. 
    Assume that $A$ is a proper cyclic flat, i.e., $\varnothing \subsetneq A\subsetneq E$. Then $\cover(A)\neq \varnothing$, because $A$ is a (non-empty) union of circuits and hence dependent, and it is non-spanning as it is a closed set which is strictly contained in the ground set $E$. It remains to show that $A$ is stressed. 
    By Lemma~\ref{lemma:cyclic-flats-restriction}, the lattice of cyclic flats of the restriction $\M|_A$, denoted by $\mathscr{Z}(\M|_A)$, is exactly the restriction of $\mathscr{Z}$ to the interval $[\varnothing, A]$, which by assumption is equal to $\{\varnothing, A\}$. The matroid $\M|_A$ with ground set $A$ satisfies $\mathscr{Z}(\M|_A) = \{\varnothing, A\}$; it follows that it is uniform. An entirely analogous reasoning shows that $\M/A$ is uniform as well and thus $A$ is stressed.
\end{proof}

In the first part of the proof of \eqref{it:cyclic-covering_a} $\Rightarrow$ \eqref{it:cyclic-covering_b}, we did not use the assumption that $\M$ was loopless and coloopless. Therefore, we have the following result for arbitrary matroids.  

\begin{coro}\label{cor:stressed-nonemptycusp-implies-cyclic}
    Let $\M$ be a matroid on the ground set $E$ and let $A\subseteq E$ be a stressed set with non-empty cover. Then $A$ is a indecomposable cyclic flat.
\end{coro}

One of the results in \cite[Theorem~3.4]{ferroni-nasr-vecchi} states that a stressed hyperplane $H$ is a witness that the set of bases of a matroid can be enlarged by a so-called ``stressed hyperplane relaxation''. This operation extends the set of bases $\mathscr{B}$ of the matroid $\M$ by declaring all the circuits contained in $H$ to be bases. In particular, it generalizes the classical circuit-hyperplane relaxation of Proposition~\ref{prop:ch-relaxation}.
In light of Remark~\ref{rem:cusp-corank1}, when $H$ is a hyperplane, we have that
    \[ \cover(H) = \binom{H}{k}.\]
In other words, whenever one relaxes a stressed hyperplane (or in particular a circuit-hyperplane), one includes precisely the cover of the stressed hyperplane to the set of bases of the matroid. The goal now is to prove that this extends to a general stressed subset. Before stating this appropriately, let us prove the following handy tool.

\begin{lemma}\label{lem:basis}
    Let $\M$ be a matroid on $E$ of rank $k$, and let $A$ be a stressed subset of $\M$. Every $S\in \binom{E}{k}$ satisfying that $\size{S\cap A}=\rk(A)$ is a basis of $\M$.
\end{lemma}

\begin{proof}
    Let us pick a set $S$ as in the statement. Then the set $S\cap A$ is a basis of the matroid $\M|_A$ as this matroid is uniform and the set $S\cap A$ is of size $\rk \M|_A$. Furthermore, the set $S\cap (E\smallsetminus A)$ is a basis of the contraction $\M/A$ for an analogous reason. By \cite[Corollary~3.1.8]{oxley}, we obtain that $S=(S\cap A)\cup (S\cap (E\smallsetminus A))$ is a basis of $\M$.
\end{proof}

A stressed subset is not necessarily a flat. If its cover is non-empty, by Corollary~\ref{cor:stressed-nonemptycusp-implies-cyclic}, it indeed is automatically a cyclic flat. In the last paragraph of \cite[Section~3]{bonin-demier-cyclic} Bonin and De Mier speculated that there should be a generalization of circuit-hyperplane relaxations. The main result in this section is a proof of this: the presence of a stressed subset allows us to enlarge the set of bases in a controlled way. 

\begin{teo}\label{thm:generalized-relaxation}
    Let $\M$ be a matroid on $E$ with bases $\mathscr{B}$ and let $A$ be a stressed subset of $\M$. Then the set $\widetilde{\mathscr{B}}=\mathscr{B}\sqcup \cover(A)$ is the family of bases of a matroid on~$E$.
\end{teo}

\begin{defi}
    We denote the matroid with bases $\widetilde{\mathscr{B}}$ of Theorem~\ref{thm:generalized-relaxation} by $\operatorname{Rel}(\M,A)$ and call it the \emph{relaxation} of~$\M$ by $A$.
\end{defi}

\begin{proof}[Proof of Theorem~\ref{thm:generalized-relaxation}]
    If the cover of $A$ is empty, then there is nothing to prove, because $\widetilde{\mathscr{B}} = \mathscr{B}$. Let us assume therefore that $A$ is a stressed subset with non-empty cover (hence, a cyclic flat). We are going to show that the basis exchange property is satisfied for any pair of bases $B_1$, $B_2$ in $\widetilde{\mathscr{B}}$ and any $x\in B_1\smallsetminus B_2$. We consider all four cases separately according to whether $B_1$ and $B_2$ lie in $\mathscr{B}$ or $\cover(A)$.
    
    There is nothing to do if both bases $B_1$ and $B_2$ are in $\mathscr{B}$, as $\mathscr{B}$ is the collection of bases of the matroid~$\M$.
    
    If $B_1\in\cover(A)$ and $B_2\in\mathscr{B}$ then $\size{B_1\cap A} > \rk(A) \geq \size{B_2\cap A}$ or equivalently 
    \[
        \size{B_2\smallsetminus A} \geq k-\rk(A) > \size{B_1\smallsetminus A}\enspace .
    \]
    Thus, the set $B_2\smallsetminus (B_1\cup A)$ is non-empty and we may pick an element $y$ from it.
    Clearly, $B_3 := (B_1\smallsetminus\{x\}) \cup \{y\}$ intersects $A$ in at least $\rk(A)$ many elements as $\size{B_1\cap A}\geq \rk(A)+1$. There are two cases, either $\size{B_3\cap A} > \rk(A)$ and thus $B_3\in\cover(A)$; or $\size{B_3\cap A} = \rk(A)$ and since $A$ is stressed, we may apply Lemma~\ref{lem:basis} to conclude that $B_3$ is a basis of $\M$.
    
    The same argument is valid whenever $B_1\in\cover(A)$ and $B_2\in\cover(A)$. In this case, the intersection of $B_3 := (B_1\smallsetminus\{x\}) \cup \{y\}$ with $A$ is at least of size $\rk(A)$ for every $y\in B_2$.
    
    If $B_1\in\mathscr{B}$,  $\size{B_1\cap A} = \rk(A)$ and $B_2\in\cover(A)$ then
    \[
        \size{B_2\cap A} \geq \rk(A)+1 > \rk(A) = \size{B_1\cap A}\enspace .
    \]
    Thus there is an element $y\in (B_2\smallsetminus B_1)\cap A$. As before, the intersection of the set $B_3 := (B_1\smallsetminus\{x\}) \cup \{y\}$ with $A$ has cardinality at least $\rk(A)$ and therefore $B_3\in\widetilde{\mathscr{B}}$.
    
    We are left with the case $B_1\in\mathscr{B}$, $\size{B_1\cap A} < \rk(A)$, and $B_2\in\cover(A)$. In this case
    \[
        \size{B_2\cap A} \geq \rk(A)+1 \text{ and } \size{B_1\smallsetminus A}> k-\rk(A) \enspace .
    \]
    Let us construct a set $B'$ of size $k$ as follows: pick $\rk(A)$ elements of $B_2\cap A$ and add another $k-\rk(A)$ elements of $B_1\smallsetminus A$. By Lemma~\ref{lem:basis}, the set $B'$ is contained in the collection $\mathscr{B}$ because $A$ is stressed and $\size{B'\cap A} = \rk(A)$. Applying the basis exchange property for the distinct bases $B_1$ and $B'$ of the matroid $\M$ provides an element $y\in B'\smallsetminus B_1$ such that $B_3 := (B_1\smallsetminus\{x\}) \cup \{y\}\in\mathscr{B}\subseteq\widetilde{\mathscr{B}}$. Observe that $y\in B_2$ as $y\not\in B_1$. Hence $B_1$, $B_2$ satisfy the exchange property as well which completes the proof.
\end{proof}

\begin{remark}
    We offer the following geometric argument that leads to an alternative proof of Theorem~\ref{thm:generalized-relaxation}. For a stressed subset $A$, the inequality
    \[
    \sum_{i\in A} x_i \leq \rank(A)
    \]
    either defines a facet of the matroid polytope $\mathscr{P}(\M)\subseteq\mathscr{P}(\U_{k,n})=\Delta_{k,n}$ or it contains the entire polytope. 
    In both cases we obtain a supporting hyperplane of $\mathscr{P}(\M)$.
    Furthermore, by definition this hyperplane intersects $\mathscr{P}(\M)$ in a matroid polytope whose matroid is isomorphic to $\M|_A\oplus \M/A$, a direct sum of two uniform matroids. This implies that the hyperplane does not intersect any other facet of $\mathscr{P}(\M)$ lying in the interior of the hypersimplex $\Delta_{k,n}$. Hence the inequality defines the unique half-space that weakly separates the points $e_B$ for $B\in\cover(A)$ from $\mathscr{P}(\M)$ and has maximal support on $\mathscr{P}(\M)$.
    It follows that the relaxation of that inequality, i.e., dropping it, leads to a new polytope whose edges are contained in $\Delta_{k,n}$ and whose vertices are the indicator vectors of $\widetilde{\mathscr{B}}$. 
    In particular, this new polytope is a matroid polytope and $\widetilde{\mathscr{B}}$ is the collection of bases of a matroid. Below, in Theorem~\ref{thm:relaxation-induces-subdivision}, we explore a further geometric feature of relaxations: they yield a nice subdivision of the relaxed matroid.
\end{remark}

\begin{example}\label{ex:relaxation}
    Consider the lattice path matroid $\M$ depicted on the left hand side of Figure \ref{fig:example-of-relaxation}. This matroid has rank $4$ and cardinality $8$. The set $F=\{6,7,8\}$ has rank $2$ and cardinality $3$ in $\M$. The set $F$ is clearly dependent and non-spanning, so $\cover(F)\neq\varnothing$. Furthermore, it is stressed, as the restriction is $\M|_F\cong \U_{2,3}$ and the contraction on $F$ is $\M/F\cong \U_{2,5}$. 
    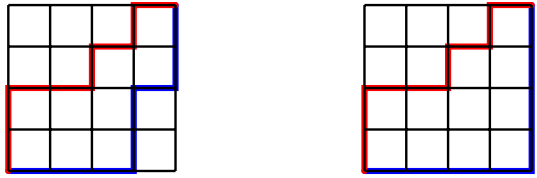
\begin{figure}[ht]
        \centering
        \begin{tikzpicture}[scale=0.55, line width=.9pt]

        \draw[line width=2.3pt,blue,line cap=round] (0,0)--(3,0) -- (3,2) -- (4,2) -- (4,4);
        \draw[line width=2.3pt,red,line cap=round] (0,0)--(0,2)--(2,2)--(2,2)--(2,3)--(3,3)--(3,4)--(4,4);
        \draw (0,0) grid (4,4);
        \end{tikzpicture}  \qquad\qquad\qquad
        \begin{tikzpicture}[scale=0.55, line width=.9pt]

        \draw[line width=2.3pt,blue,line cap=round] (0,0)--(4,0) -- (4,4);
        \draw[line width=2.3pt,red,line cap=round] (0,0)--(0,2)--(2,2)--(2,2)--(2,3)--(3,3)--(3,4)--(4,4);
        \draw (0,0) grid (4,4);
        \end{tikzpicture}
    \caption{On the left the matroid of Example~\ref{ex:relaxation} and on the right its relaxation.}\label{fig:example-of-relaxation}
    \end{figure}
    It can be proved that $\cover(F) = \left\{ \{1,6,7,8\}, \{2,6,7,8\}, \{3,6,7,8\}, \{4,6,7,8\}, \{5,6,7,8\} \right\}$, and $\widetilde{\M} := \Rel(\M, F)$ is precisely the Schubert matroid depicted on the right hand side of Figure~\ref{fig:example-of-relaxation}. Notice that the relaxation in this example is neither a circuit-hyperplane relaxation nor a stressed hyperplane relaxation. The set $F$ is not a hyperplane of $\M$ and its complement $E\smallsetminus F$ is not a hyperplane of $\M^*$.
\end{example}

An important feature of this new notion of relaxation is that, unlike the stressed hyperplane relaxation of \cite{ferroni-nasr-vecchi}, it is now compatible with dualizations. 

\begin{prop}\label{prop:relax-dual}
    The stressed subset relaxation commutes with taking duals of matroids. In other words, if $A$ is a stressed subset of $\M$, then $E\setminus A$ is a stressed subset of $\M^*$ and
        \[ \Rel(\M, A)^* = \Rel(\M^*, E\smallsetminus A).\]
\end{prop}

\begin{proof}
    This follows directly from Proposition~\ref{prop:cusp-dual}: The elements in $\cover^*(E\smallsetminus A)$ are exactly the complements of the elements in $\cover(A)$.
\end{proof}

\subsection{Structural properties of stressed subset relaxations}

As we have mentioned repeatedly, if a stressed set has non-empty cover, then it is a cyclic flat. In order to emphasize this property, we use the letter $F$ for a stressed subset with non-empty cover. However, there are a few exceptions. For instance, the next result is such a case in which we use the letter $F$ although the only requirement we impose is that our set be stressed. 

Throughout this subsection we collect some basic features of the matroid $\Rel(\M,F)$ in terms of $\M$ and $F$. We start by characterizing the rank function of this new matroid.

\begin{prop}\label{prop:rank-relaxation}
    Let $\M$ be a matroid of rank $k$, and let $F$ be a stressed subset. Consider $\widetilde{\M}=\Rel(\M,F)$, the relaxation of $\M$ by $F$. For every subset $A\subseteq E$, the rank of $A$ in $\widetilde{\M}$ is given by
    \[ \rk_{\widetilde{\M}}(A) = \begin{cases} \min(|A|, k) & \text{if $|A\cap F|\geq \rk(F) + 1$,}\\ \rk_{\M}(A) & \text{otherwise.}\end{cases}\]
\end{prop}

\begin{proof}
    Let us denote by $\widetilde{\rk}$ the rank function of $\widetilde{\M}$. Since every basis of the matroid $\M$ is a basis in $\widetilde{\M}$ we have $\rk(A) \leq \widetilde{\rk}(A)$ for every $A\subseteq E$. 
    We will begin by dealing with the first case, that is $|A\cap F| \geq \rk(F) + 1$.
    In this case we consider two subcases according to whether $|A|\leq k$ or $|A|>k$. In the first case, we can add elements to $A$ until we obtain a set $S$ of cardinality $k$. This set will satisfy $|S\cap F| \geq |A\cap F| \geq \rk(F) + 1$ and thus $S\in\cover(F)$. In particular, we have that $S$ is a basis of $\widetilde{\M}$ and given that $A$ is contained in it, we have that $\widetilde{\rk}(A) = |A|$. 
    If $|A|>k$, since $|A\cap F|\geq \rk(F)+1$, we may select $\rk(F)+1$ elements of $A\cap F$ and discard elements of $A$ that are not in our previous selection until we obtain a set $S$ of size~$k$. The set $S$ satisfies $|S\cap F| \geq \rk(F)+1$, and hence again, $S$ is a basis of $\widetilde{\M}$ and the set $A$ contains the set $S$. We conclude that $\widetilde{\rk}(A) = k$ whenever $|A|>k$.
    
    It remains to show that the inequality $\rk(A) < \widetilde{\rk}(A)$ implies $|A\cap F|\geq \rk(F) + 1$. 
    Assume $\rk(A) < \widetilde{\rk}(A)$. By definition of the rank function in $\widetilde{\M}$, there is a set $S\in\cover(F)$ such that 
        \[ |A\cap S| > |A\cap I| \;\; \text{for all independent sets $I$ of $\M$}.\]
    Consider the family $\mathscr{X}$ of all subsets of $S\cap F$ that are of size $|S\cap F| - \rk(F)$. Observe that by construction
    $|(S\smallsetminus X)\cap F| = \rk(F)$
    for every $X\in \mathscr{X}$. In particular, by Lemma~\ref{lem:basis}, $S\smallsetminus X$ is an independent set in $\M$. Using the inequality above, we obtain that
        \[ |A\cap S| > |A\cap (S\smallsetminus X)| = |(A\smallsetminus X)\cap S|\]
    for every $X\in\mathscr{X}$. Observe that this implies that $A\cap X\neq \varnothing$ so it follows that $A$ intersects all the members of $\mathscr{X}$. Thus $A\cap F$ intersects all the elements in $\mathscr{X}$, too, as all elements of $\mathscr{X}$ are subsets of $F$. 
    Notice that if a set  intersects all the members of the collection $\binom{[a]}{b}$, then its size is at least $a-b+1$. It is clear that a set of size $a-b$ is not sufficient, because the complement of any $b$-subset of $[a]$ does not intersect this $b$-set. This gives us
        \[ |A\cap F| \geq |S\cap F| - (|S\cap F| - \rk(F)) + 1 = \rk(F) + 1,\]
    as desired.
\end{proof}

\begin{remark}\label{rem:rank}
   Using the notation of the last proposition, and denoting $r=\rk(F)$, we have that if $|A\cap F|\geq r + 1$, then $\widetilde{\rk}(A) = \min(|A|, k)$. We claim that in this case $\rk(A) = \min(k, r + |A\smallsetminus F|)$. To convince ourselves, let us call $X = A\cap F$ and $Y = A\smallsetminus F$. The fact that $|X|\geq \rk(F) + 1$ and that $\M|_F \cong \U_{r,h}$ implies that $\rk(X) = \rk(F)$. 
   Furthermore, since $\M/F \cong \U_{k-r,n-h}$, it follows that the rank of $Y$ in $\M/F$ is $\min(k-r, |Y|)$.  By \cite[Proposition 3.1.6]{oxley}:
   \[ \rk_{\M/F}(Y) = \rk(F\cup Y) - \rk(F),\]
   which implies
   \[ \rk(F\cup Y) = r + \min(k-r, |Y|).\]
  Now, since $X\subseteq F$ and $\rk(X) = \rk(F)$, it follows that $\rk(X\cup Y) = \rk(F\cup Y)$. Therefore,
  \[ \rk(A) = \rk(X\cup Y) = r + \min(k-r,|Y|) = \min(k, r+|A\smallsetminus F|),\]
  as claimed. This formula will be of use in some later proofs.
\end{remark}

\subsection{The lattice of cyclic flats}

As mentioned earlier, in \cite[Section~3]{bonin-demier-cyclic} it was hinted that there should be a generalization of circuit-hyperplane relaxations. They made the observation that whenever the lattice of cyclic flats $\mathscr{Z}$ of a matroid $\M$ has an element $F$ that covers the bottom element $0_{\mathscr{Z}}$ and is covered by the top element $1_{\mathscr{Z}}$, then it is possible to remove it and obtain a new matroid. As Proposition~\ref{prop:cyclic-covering-iff-stressed} suggests, the generalization they foreshadowed is exactly the notion of relaxation of stressed subsets that we introduced, at least for loopless and coloopless matroids.

\begin{prop}\label{prop:relaxation-equals-edelete-cyclic-flat}
    Let $\M$ be a matroid on $E$ without loops and coloops. Assume that $\M$ has a stressed subset $F$ with non-empty cover. Then the following two matroids coincide.
    \begin{enumerate}[\normalfont(a)]
        \item The relaxation $\Rel(\M,F)$.
        \item The matroid on $E$ whose lattice of cyclic flats is obtained from $\mathscr{Z}(\M)$ by deleting the element~$F$.
    \end{enumerate}
\end{prop}

\begin{proof}
    Notice that by Theorem~\ref{thm:lattice-cyclic-flats}, deleting an element that covers $\varnothing$ and is covered by $E$ yields the lattice of cyclic flats of a matroid $\N$ on $E$. Furthermore, by \eqref{eq:bases-from-cyclic-flats} it follows that 
    \[ \mathscr{B}(\N) = \left\{B \in \binom{E}{k}: |X\cap B| \leq \rk_{\M}(X) \text{ for all } X\in \mathscr{Z}(\M)\smallsetminus\{F\}\right\}.\]
    From this, we see that $\mathscr{B}(\M)\subseteq \mathscr{B}(\N)$. 
    Notice that $B\in \mathscr{B}(\N)\smallsetminus\mathscr{B}(\M)$ implies 
    \begin{center}
            $B\in\binom{E}{k}$ and $|F\cap B|\geq \rk_{\M}(F)+1$ or, equivalently, $B\in \cover(F)$.
    \end{center}
    In other words $\mathscr{B}(\N)\smallsetminus\mathscr{B}(\M)\subseteq\cover(F)$. We claim that this inclusion of sets is in fact an equality, and thus $\Rel(\M,F) = \N$.
    To see that the converse holds take a $k$-subset $B\subseteq E$ such that $B\in \cover(F)$. Consider any other proper cyclic flat $X\neq F$ of $\M$. Since $F$ is stressed, by Proposition~\ref{prop:cyclic-covering-iff-stressed} it follows that $X\vee F = E$ and that $X\wedge F = \varnothing$. Thus, by property (Z3) in Theorem~\ref{thm:lattice-cyclic-flats} we have
    \[ \rk_\M(X) + \rk_\M(F) \geq \rk_\M(E) + \rk_\M(\varnothing) + |X\cap F| = k + |X\cap F|.\]
    Now, since we assumed $B\in \cover(F)$ we have that $|B\cap F|\geq r + 1$ or, equivalently, $|B\cap (E\smallsetminus F)|\leq k - r - 1$ where $r=\rk_\M(F)$. In particular, 
    \begin{align*}
        |X\cap B| &= |X\cap B\cap F| + |X\cap B \cap (E\smallsetminus F)|\\
        &\leq |X\cap F| + |B\cap (E\smallsetminus F)|\\
        &\leq |X\cap F| + k - r - 1\\
        &\leq \rk_\M(X) + \rk_\M(F) - \rk_\M(F) - 1\\
        &= \rk_\M(X) - 1.
    \end{align*}
    This shows that $B\in\cover(F)$ implies $B\in\mathscr{B}(\N)$ and thus completes the proof as the cover of $F$ is clearly disjoint from $\mathscr{B}(\M)$.
\end{proof}

\subsection{A special kind of Schubert matroids}

Let us analyze a family of examples that will play a central role in the sequel. We start with the following preliminary result.

\begin{lemma}
    Let $\M = \U_{k-r,n-h}\oplus\U_{r,h}$ where $0 \leq r < h$ and $0< k - r \leq n - h$. Then, the ground set of the direct summand $\U_{r,h}$ is a stressed subset with non-empty cover and thus a cyclic flat.
\end{lemma}

\begin{proof}
    Let us call $F$ the ground set of the summand $\U_{r,h}$. Clearly, both the restriction $\M|_F \cong \U_{r,n}$ and contraction $\M/F \cong \U_{k-r,n-h}$ are uniform matroids, hence $F$ is stressed. Since $r < k$, it follows that $F$ is not spanning in $\M$, because the rank of $\M$ is $k = k-r + r > \rk(F)$. Since $r<h$, also $F$ is not independent in $\M$. It follows by Proposition~\ref{prop:cusp-empty} that $\cover(F)\neq\varnothing$.
\end{proof}

The disconnected matroid $\M$ of the preceding statement has rank $k$ and cardinality $n$ and admits a stressed cyclic flat with non-empty cover having rank $r$ and size $h$. Next we will take a closer look at the relaxation of that flat. 

\begin{defi}\label{def:cuspidal}
    For $0 \leq r \leq h$ and $0\leq k - r \leq n - h$, we define the matroid \[\LL_{r,k,h,n} := \Rel(\U_{k-r,n-h}\oplus\U_{r,h}, \U_{r,h}),\]
    i.e., we relax the subset corresponding to the ground set of the direct summand $\U_{r,h}$.
\end{defi}

\begin{remark}
    It is worth mentioning that relaxing the first summand of $\U_{k-r,n-h}\oplus\U_{r,h}$ leads to a matroid that is isomorphic to the cuspidal matroid $\LL_{k-r,k,n-h,n}$.
\end{remark}

The matroids $\LL_{r,k,h,n}$ and matroids isomorphic to them will be subsequently called \emph{cuspidal matroids}. After we have properly introduced the notion of elementary split matroid, it will be realized that cuspidal matroids are precisely the Schubert elementary split matroids.


Notice that our definition includes the cases $r=h$ and $k-r=0$. In those cases, $F$ is in fact stressed, but we have $\cover(F)=\varnothing$, and thus the relaxation is exactly the same matroid we started with. More precisely, we are defining
    \begin{equation}\label{eq:cuspidal-border-case}
        \LL_{h,k,h,n} := \U_{k-h,n-h}\oplus \U_{h,h}, \quad\text{and}\quad\LL_{k,k,h,n} := \U_{0,n-h}\oplus \U_{k,h}\enspace . 
    \end{equation}

The matroid $\LL_{r,k,h,n}$ always has rank $k$ and cardinality $n$. It follows from Theorem~\ref{thm:generalized-relaxation} that the set of bases of $\LL_{r,k,h,n}$ is given by
    \begin{equation}\label{eq:bases-of-cuspidal}
    \mathscr{B}(\LL_{r,k,h,n}) = \left\{B\in \binom{E}{k} : |F\cap B|\geq r\right\} = \left\{B\in \binom{E}{k} : |(E\smallsetminus F)\cap B|\leq k-r \right\},
    \end{equation}
where $F$ is the ground set of the direct summand $\U_{r,h}$ and thus of size $h$.
By Proposition~\ref{prop:size-of-cusp} the number of elements in $\mathscr{B}(\LL_{r,k,h,n})$ is  
    \begin{equation}\label{eq:num-bases-cuspidal}
    |\mathscr{B}(\LL_{r,k,h,n})| = \sum_{i=r}^k \binom{h}{i}\binom{n-h}{k-i}.
    \end{equation}
    
Some particular instances of these matroids have made appearances in the literature before. For example, the matroid $\LL_{k-1,k,h,n}$ appeared in \cite{ferroni-nasr-vecchi}, where it is denoted by $\widetilde{\mathsf{V}}_{k,h,n}$ and in \cite{hanely}, and called \emph{panhandle matroid}. Another relevant case is the further specialization $\LL_{k-1,k,k,n}$ which is known as the \emph{minimal matroid} $\mathsf{T}_{k,n}$ (denoted this way in \cite{ferroni2}). This matroid receives its name from being the unique connected matroid of rank $k$ and size $n$ up to isomorphism achieving the minimal number of bases; see \cite{dinolt, murty}.

\begin{remark}
    We have chosen the ordering of the subscripts $r, k, h, n$, because we always have the inequalities $r\leq k \leq n$ and $r\leq h \leq n$. However, it can either be $k < h$ or $k \geq h$. 
\end{remark}

\begin{remark}\label{remark:cases-cuspidal-is-uniform}
   If $r=0$, we have $\LL_{0,k,h,n} \cong \U_{k,n}$ for every $0\leq h\leq n-k$. In other words, when $r=0$, the value of $h$ is irrelevant. Conversely, when $k-r=n-h$, the matroid $\LL_{r,k,h,n} = \U_{k,n}$ is independent of the parameters $r$ and $h$. 
\end{remark}

\begin{prop}
    The matroid $\LL_{r,k,h,n}$ is the Schubert matroid presented as a lattice path matroid using the upper path $U = \text{\normalfont N}^{k-r}\text{\normalfont E}^{n-k-h+r}\text{\normalfont N}^r\text{\normalfont E}^{h-r}$, as depicted in Figure~\ref{fig:cuspidal}. 
    \begin{figure}[ht]
        \centering
        \begin{tikzpicture}[scale=0.50, line width=.9pt]

        \draw[line width=2.3pt,blue,line cap=round] (0,0)--(9,0) -- (9,7);
        \draw[line width=2.3pt,red,line cap=round] (0,0)--(0,3) -- (5,3)--(5,7)--(9,7);
        \draw (0,0) grid (9,7);
        \draw[decoration={brace,raise=7pt},decorate]
            (0,0) -- node[left=7pt] {$k-r$} (0,3);
            
        \draw (5,0) grid (9,7);
        \draw[decoration={brace,mirror, raise=4pt},decorate]
        (0,0) -- node[below=7pt] {$n-k$} (9,0);
        
        \draw[decoration={brace, raise=5pt},decorate]
        (5,7) -- node[above=7pt] {$h-r$} (9,7); 
        
        \draw[decoration={brace, raise=5pt},decorate]
        (9,7) -- node[right=7pt] {$k$} (9,0); 
        \end{tikzpicture}
        \caption{A presentation of $\LL_{4,7,8,16}$ as a lattice path matroid.}
        \label{fig:cuspidal}
    \end{figure}
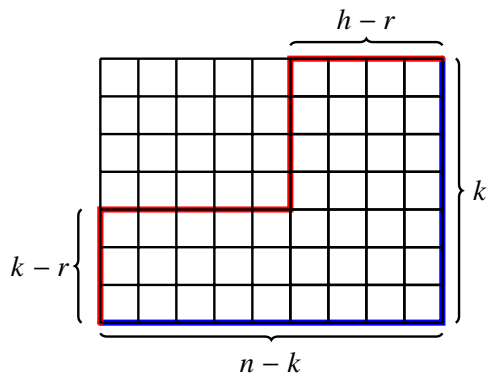
\end{prop}

\begin{proof}
    By definition, the lattice path matroid in Figure \ref{fig:cuspidal} has the following set of bases:
        \[ \mathscr{B} = \left\{ B\in \binom{[n]}{k} : \left|B\cap [n-h]\right| \leq k - r\right\}.\]
    Comparing this with equation~\eqref{eq:bases-of-cuspidal} yields the result.
\end{proof}

\begin{lemma}
    The dual of the matroid $\LL_{r,k,h,n}$ is isomorphic to another cuspidal matroid, namely
        \[ \LL_{r,k,h,n}^* \cong \LL_{n-h-k+r,\,n-k,\,n-h,\,n}.\]
\end{lemma}

\begin{proof}
    Using Proposition~\ref{prop:relax-dual}, we obtain that
    \begin{align*}
        \LL_{r,k,h,n}^* &= \Rel(\U_{k-r,n-h}\oplus\U_{r,h} ,\U_{r,h})^*\\
        &= \Rel(\U_{n-h-k+r,n-h}\oplus\U_{h-r, h},  \U_{n-h-k+r,n-h})\\
        &\cong \LL_{n-h-k+r, n-k, n-h, n}.
    \end{align*}
    Notice that we abused the notation $\Rel(-, \U_{a,b})$ again to say that we are relaxing the ground set of the direct summand $\U_{a,b}$.
\end{proof}

Proposition~\ref{prop:rank-relaxation} tells us directly the rank function of the matroids $\LL_{r,k,h,n}$.

\begin{coro}\label{cor:rank_cusp}
    Let $X=\{n-h+1,\ldots,n\}$. The rank function of the matroid $\LL_{r,k,h,n}$ is given by
    \[
    \rank(A) =  
    \begin{cases} \min(|A|, k) & \text{if $|A\cap X|\geq r+1$,}\\ \min(|A|, |X\cap A|+k-r) & \text{otherwise.}\end{cases}
    \]
\end{coro}

\subsection{Subdivisions from relaxations}

By combining the outer description of Remark~\ref{rem:polytope-cyclic-flats} and Proposition~\ref{prop:relaxation-equals-edelete-cyclic-flat}, we can describe the base polytope of $\Rel(\M,F)$ as follows.

\begin{prop}\label{prop:polytope_ineq}
    If $\M$ is loopless and coloopless and $F$ is a stressed subset with non-empty cover, then $\mathscr{P}(\Rel(\M,F))$ is obtained from $\mathscr{P}(\M)$ by relaxing the inequality $\sum_{i\in F}x_i\leq \rank(F)$.
\end{prop}

In particular, when specializing the preceding statement for the matroids $\LL_{r,k,h,n}$ we obtain the following.

\begin{coro}\label{coro:cuspidal-polytope}
    The base polytope of the cuspidal matroid $\LL_{r,k,h,n}$ is given by
    \begin{align*}
     \mathscr{P}(\LL_{r,k,h,n}) &=
    \left\{x\in [0,1]^n : \sum_{i=1}^n x_i = k, \text{ and } \sum_{i=1}^{n-h} x_i \leq k-r\right\}.
    \end{align*}
\end{coro}

The following result explains why the relaxation of a stressed subset can be used to describe a subdivision of the relaxed matroid. 

\begin{teo}\label{thm:relaxation-induces-subdivision}
    Let $\M$ be a matroid of rank $k$ and cardinality $n$ with a stressed subset $F$ whose cover is non-empty. Further, let $\widetilde{\M}$ be the relaxation $\Rel(\M,F)$.
    Then, there exist matroids $\N_1$
    and $\N_2$ such that
    \begin{align*}
        \mathscr{P}(\M) \cup \mathscr{P}(\N_1) = \mathscr{P}(\widetilde{\M})\quad \text{ and }\quad
        \mathscr{P}(\M) \cap \mathscr{P}(\N_1) = \mathscr{P}(\N_2).
    \end{align*}
    Moreover, $\N_1\cong \LL_{r,k,h,n}$ and $\N_2\cong \U_{k-r,n-h}\oplus\U_{r,h}$ where $r =\rk(F)$ and $h=|F|$. 
\end{teo}

\begin{proof}
    The statement is a geometric reformulation of previous results and definitions.
    We may assume without loss of generality that $E=\{1,\ldots,n\}$ and $F=\{n-h+1,\ldots,n\}$.
    With these assumptions, let us set $\N_1=\LL_{r,k,h,n}$ and $\N_2=\U_{k-r,n-h} \oplus \U_{r,h}$.
    We show first that $\mathscr{P}(\M) \cap \mathscr{P}(\N_1) = \mathscr{P}(\N_2)$.
    Proposition~\ref{prop:polytope_ineq} implies that the points $x\in\mathscr{P}(\M)$ satisfy
    \[
        \sum_{i=n-h+1}^n x_i = \sum_{i\in F} x_i \leq \rank(F) = r,
    \]
    and Corollary~\ref{coro:cuspidal-polytope} shows that all $x\in\mathscr{P}(\N_1)=\mathscr{P}(\LL_{r,k,h,n})$ fulfill
    \[
    \sum_{i=1}^{n-h} x_i \leq k-r \text{ or equivalently }
    \sum_{i\in F} x_i \geq \rank(F).
    \]
    Thus 
    \[
    \mathscr{P}(\M) \cap \mathscr{P}(\LL_{r,k,h,n}) \subseteq
    \left\{x\in [0,1]^n : \sum_{i=1}^n x_i = k, \text{ and } \sum_{i\in F} x_i = \rank(F)\right\} = 
    \mathscr{P}(\U_{k-r,n-h}\oplus\U_{r,h}).
    \]
    The equality follows from the fact that all bases of $\U_{k-r,n-h}\oplus\U_{r,h}$ are bases of both $\M$ and $\LL_{r,k,h,n}$ see Lemma~\ref{lem:basis} and the discussion below Definition~\ref{def:cuspidal}.
    
    Now, we discuss the equality $\mathscr{P}(\M) \cup \mathscr{P}(\N_1) = \mathscr{P}(\widetilde{\M})$. 
    Every basis of $\M$ and $\N_1$ is a basis of $\widetilde{\M}$, thus $\mathscr{P}(\M)\cup \mathscr{P}(\N_1) \subseteq \mathscr{P}(\widetilde{\M})$.
    To prove the reversed inclusion $\mathscr{P}(\M) \cup \mathscr{P}(\N_1) \supseteq \mathscr{P}(\widetilde{\M})$,  take $x\in \mathscr{P}(\widetilde{\M})$. Now, either $\sum_{i \in F} x_i \leq r$ and thus $x \in \mathscr{P}(\M)$ by combining Proposition~\ref{prop:polytope_ineq} and Proposition~\ref{prop:relaxation-equals-edelete-cyclic-flat}, or $\sum_{i \in F} x_i \geq r$.
    In this case $\sum_{i \not\in F} x_i \leq k-r$. 
    Observe that our point $x$ lies in  $\mathscr{P}(\widetilde{\M})$ and hence satisfies the conditions $0\leq x_i\leq 1$ for each $i\in E$, $\sum_{i=1}^n x_i = k$ and $\sum_{i=1}^{n-h} x_i \leq k-r$. Hence, by Corollary~\ref{coro:cuspidal-polytope}, it follows that $x\in \mathscr{P}(\LL_{r,k,h,n})$. This completes the proof.
\end{proof}

By substituting $r+1 = k = h$ in the preceding statement one retrieves a main result of Ferroni \cite[Theorem~5.4]{ferroni2} namely, relaxing a circuit-hyperplane corresponds to gluing a copy of the base polytope of a minimal matroid $\mathsf{T}_{k,n}$ on a facet of the polytope.

\section{Large classes of matroids}\label{sec:large-classes}
\noindent
The aim of this section is to introduce the class of split matroids and its subclass of elementary split matroids. Split matroids were first studied by Joswig and Schr\"oter \cite{joswig-schroter} from both a polyhedral and tropical point of view. In this section we give a description of them using the notion of relaxation that we developed in Section \ref{sec:stressed-subsets-relaxations}. Also, we will emphasize some enumerative aspects in order to explain what we mean by a ``large'' class of matroids. We will also discuss some bounds on the number of stressed subsets in matroids belonging to these large classes.

\subsection{An overview of paving matroids} \label{subsect:paving-matroids}

\begin{defi}
    A matroid $\M$ is said to be \emph{paving} if all the dependent sets of $\M$ have size at least $\rk(\M)$. If both $\M$ and its dual $\M^*$ are paving, then $\M$ is said to be \emph{sparse paving}.
\end{defi}

It is straightforward to check with the above definition that, e.g., all loopless matroids of rank $2$ and all simple matroids of rank $3$ are paving.  Even though paving matroids are easier to approach than \emph{arbitrary} matroids, they still are far from well-understood. In terms of relaxations of stressed hyperplanes, the class of paving matroids can be characterized as follows.

\begin{prop}[{\cite[Proposition 3.16]{ferroni-nasr-vecchi}}]
    A matroid $\M$ is paving if and only if all of its hyperplanes are stressed.
\end{prop}


Observe that, by Remark~\ref{rem:cusp-corank1}, a hyperplane $H$ of the matroid $\M$ has non-empty cover if and only if $|H|\geq k=\rk(\M)$. This is clear geometrically, as it is impossible for a vertex of the hypersimplex $\Delta_{k,n}$ to violate an inequality of the form $\sum_{i\in H} x_i \leq \rk(H)$ unless $|H| > \rk(H)$. In particular, if we relax all the stressed hyperplanes of cardinality at least $k$, we obtain a paving matroid such that all of its hyperplanes have size $k-1$. In other words, we obtained the uniform matroid $\U_{k,n}$. It follows that all paving matroids can be relaxed until obtaining a uniform matroid.

\begin{coro}\label{cor:charact-paving}
    A matroid $\M$ is paving if and only if after relaxing all of its stressed hyperplanes with non-empty cover one obtains a uniform matroid.
\end{coro}

Similarly, one can characterize sparse paving matroids, by saying that all the non-bases are circuit-hyperplanes (cf.~\cite[Lemma~2.8]{ferroni3}). One of our goals in the next subsections will be to generalize the last statement to a larger, but still well-structured, family of matroids. We do this by characterizing all matroids that lead to a uniform matroid when one relaxes all of their stressed \emph{subsets} with non-empty cover (as opposed to only hyperplanes).

\subsection{Split matroids}

In \cite{joswig-schroter} Joswig and Schr\"oter initiated the study of a class of matroids that contains all paving matroids and their duals. This family of matroids arose in the context of studying the polyhedral structure of the tropical Grassmannian and Dressian. They focused on the splits of a hypersimplex. 

We begin with a definition of split matroids that relies on the terms ``flacet'' introduced in \cite{feichtner-sturmfels} and ``split flacet''. 

\begin{defi}
    Let $\M$ be a connected matroid. A \emph{flacet} of $\M$ is a flat $F$ such that $\M|_F$ and $\M/F$ are connected. A flacet $F$ of $\M$ is said to be a \emph{split flacet} if $0 < \rk(F) < |F|$ and $E\smallsetminus F$ contains at least one element that is not a coloop of $\M$.
\end{defi}

Note that a split flacet of a matroid $\M$ is by definition a cyclic flat of $\M$.
In our next definition we take advantage of the fact that we have introduced the term ``stressed subset'' to denote that a set~$A$ has the property that both its restriction and contraction are uniform. This allows us to paraphrase \cite[Theorem~11]{joswig-schroter}, which characterizes connected split matroids and use it in combination with \cite[Proposition 15]{joswig-schroter}  as our definition. 

\begin{defi}
    A \emph{connected split matroid} $\M$ is a connected matroid for which all the split flacets are stressed. 
    A \emph{split matroid} is a matroid 
    isomorphic to a direct sum of matroids of which at most one of the direct summands is a non-uniform connected split matroid and the remaining summands are uniform matroids.
\end{defi}

Recently, B\'erczi, Kir\'aly, Schwarcz, Yamaguchi, and Yokoi studied in \cite{berczi} a subclass of split matroids, called ``elementary split matroids''. 

\begin{defi}[\cite{berczi}]\label{def:elementary-split}
    A matroid $\M$ is an \emph{elementary split matroid} if it is either a connected split matroid or a direct sum of two uniform matroids.
\end{defi}

Throughout the rest of this paper, we write our statements for elementary split matroids as opposed to split matroids. The reason for this is that most of our statements for connected split matroids carry over verbatim to elementary split matroids, whereas they often require annoying adjustments for disconnected split matroids. Let us mention explicitly that the class of elementary split matroids contains all paving and copaving matroids as well.

The next theorem is a list of equivalent characterizations of elementary split matroids that can be found in \cite[Theorem~11]{berczi}.

\begin{teo}[{\cite[Theorem~11]{berczi}}]\label{thm:kristof-equivalences}
    Let $\M$ be a matroid. The following are equivalent.
    \begin{enumerate}[\normalfont(a)]
        \item $\M$ is an elementary split matroid.
        \item $\M$ has no minor isomorphic to $\U_{0,1}\oplus \U_{1,1}\oplus \U_{1,2}$.
        \item\label{it:car_el_split} $\M$ is loopless and coloopless and every two proper cyclic flats $F_1$ and $F_2$ are incomparable in~$\mathscr{Z}$; or $\M$ is isomorphic to a direct sum of the form $\U_{a,b}\oplus\U_{c,c}$ or $\U_{a,b}\oplus\U_{0,c}$, with $0 \leq a \leq b$ and $c\geq 1$.
    \end{enumerate}
\end{teo}

If we combine Proposition~\ref{prop:cyclic-covering-iff-stressed} with the characterization of elementary split matroids in Theorem~\ref{thm:kristof-equivalences}\eqref{it:car_el_split}, we obtain the following extension of Corollary~\ref{cor:charact-paving}, as we were aiming for.

\begin{teo}\label{thm:elem-split-relaxation-yields-uniform}
    A matroid $\M$ is elementary split if and only if after relaxing all of its stressed subsets with non-empty cover (which are thus cyclic flats) one obtains a uniform matroid.
\end{teo}

\begin{proof}
    Assume that $\M$ is an elementary split matroid. If $\M$ is loopless and coloopless, by the preceding list of equivalences we obtain that every pair of proper cyclic flats are incomparable, i.e., if $F_1$ and $F_2$ are cyclic flats and $F_1\subsetneq F_2$ then either $F_1=\varnothing$ or $F_2=E$; in particular, it follows that every proper cyclic flat of $\M$ covers $0_{\mathscr{Z}}$ and is covered by $1_{\mathscr{Z}}$. By Proposition~\ref{prop:cyclic-covering-iff-stressed}, this implies that all these cyclic flats are stressed and have non-empty cover. By Proposition~\ref{prop:relaxation-equals-edelete-cyclic-flat}, it follows that relaxing all of these cyclic flats amounts to erasing them from the lattice of cyclic flats. After all these relaxations, the only remaining elements are the bottom and the top elements, namely $\varnothing$ and $E$. Hence the resulting matroid is uniform by Lemma~\ref{lemma:cyclic-flats-uniform}. If $\M$ has loops or coloops, it is a direct sum of two uniform matroids of the form $\U_{a,b}\oplus \U_{c,c}$ (if it has coloops) or $\U_{a,b}\oplus \U_{0,c}$ (if it has loops) with $0 \leq a \leq b$ and $c\geq 1$. In the first case, the only stressed subset with non-empty cover is the ground set of $\U_{a,b}$, whereas in the second case the only stressed subset with non-empty cover is the ground set of $\U_{0,c}$. It is straightforward to verify that in both cases the cover of the stressed subset under consideration consists of all the non-bases. In particular, the relaxation yields a uniform matroid.
    
    Conversely, assume that $\M$ is a matroid such that relaxing all of its stressed subsets with non-empty cover yields a uniform matroid. If $\M$ is loopless and coloopless, by Proposition~\ref{prop:cyclic-covering-iff-stressed} and Proposition~\ref{prop:relaxation-equals-edelete-cyclic-flat}, the only possibility is that all the proper elements in the lattice of cyclic flats $\mathscr{Z}$ of $\M$ cover the bottom and are covered by the top element. This implies that all the proper cyclic flats are incomparable. The preceding theorem implies that $\M$ is elementary split. Now let us analyze the case in which $\M$ has loops or coloops. It suffices to look at the case in which $\M$ has loops, otherwise we apply 
    Proposition~\ref{prop:relax-dual} together with the fact that elementary split matroids are closed under dualization.
    Now, consider any stressed subset $A$ with non-empty cover. It automatically is a cyclic flat by Corollary~\ref{cor:stressed-nonemptycusp-implies-cyclic}, and thus contains the set $L\neq\varnothing$ of all loops of $\M$. 
    Furthermore, since $A$ is stressed the restriction $\M|_A$ is isomorphic to a uniform matroid $\U_{r,h}$.
    The elements in $L$ are loops of this matroid, thus it follows that $r = 0$, $A=L$, and $h=|L|$.
    Using once more that $A$ is a stressed subset shows that $\M/A\cong \U_{k,n-h}$. It follows that $\M \cong \U_{k,n-h}\oplus \U_{0,h}$ and hence is an elementary split matroid.
\end{proof}

\begin{example}\label{example:cuspidal-are-split}
    All cuspidal matroids $\LL_{r,k,h,n}$ are indeed elementary split matroids. The only matroids of this form with (co)loops are of the form $\U_{a,b}\oplus \U_{0,c}$ (respectively $\U_{a,b}\oplus \U_{c,c}$), thus by Theorem~\ref{thm:kristof-equivalences} it follows that these matroids are elementary split. 
    Furthermore, when $\LL_{r,k,h,n}$ matroid neither has loops nor coloops, then it has only one proper cyclic flat and again Theorem~\ref{thm:kristof-equivalences} implies the claim. 
    An alternative way of showing this is to look at the definition of the matroid  $\LL_{r,k,h,n}$. This matroid is obtained by relaxing the ground set of the second direct summand in the matroid $\U_{k-r,n-h}\oplus \U_{r,h}$; if one additionally relaxes the \emph{first} direct summand, the resulting matroid is precisely $\U_{k,n}$ which shows by Theorem~\ref{thm:elem-split-relaxation-yields-uniform} that $\LL_{r,k,h,n}$ is elementary split. Notice that in all cases the condition on the cyclic flats also guarantees that $\LL_{r,k,h,n}$ is Schubert.
\end{example}

\subsection{Large classes and predominance}

We include this short subsection to justify the use of the word ``large'' in the title of the present article, and to pose Conjecture~\ref{conj:split-predominates}. A major open problem in matroid theory consists of finding the right asymptotics for the number of labelled or unlabelled matroids on $n$ elements. More precisely, it is desirable to obtain asymptotic estimations for the quantities
    \begin{align*}
        \operatorname{mat}(n) &:= |\{\text{matroids on $[n]$}\}|,\\
        \operatorname{mat}'(n) &:= |\{\text{matroids on $[n]$ up to isomorphism}\}|.
    \end{align*}

More than half a century ago, Crapo and Rota postulated in \cite{crapo-rota} that ``paving matroids may predominate in any asymptotic enumeration of matroids''. More recently, Mayhew, Newman, Welsh, and Whittle \cite{mayhew} formalized this and posed the following conjecture.

\begin{conj}[\cite{mayhew}]\label{conj:mayhewetal}
    Let $\operatorname{sp}(n)$ denote the number of sparse paving matroids on $[n]$ and $\operatorname{sp}'(n)$ the number of sparse paving matroids on $[n]$ up to isomorphism. Then
    \[ \lim_{n\to \infty} \frac{\operatorname{sp}(n)}{\operatorname{mat}(n)} = 1 \;\;\;\;\; \text{and} \;\;\;\;\; \lim_{n\to \infty} \frac{\operatorname{sp}'(n)}{\operatorname{mat}'(n)} = 1.\]
\end{conj}

Some progress towards this conjecture has been made in the past decade by Pendavingh and van der Pol \cite{pendavingh-vanderpol}. This is one of the reasons for which paving matroids have received particular attention throughout the last decades.

Besides the question of whether asymptotically all matroids are sparse paving, paving, elementary split or split, it would be useful to understand better how the sizes of these classes of matroids compare to one another. For example, a question that arose in our study is whether the class of elementary split matroids captures asymptotically all the \emph{non-sparse paving} matroids. With the help of a computer we found that there are $156704$ isomorphism classes of non-sparse paving matroids on $9$ elements. Among these, $83628$ correspond to elementary split matroids. In other words, for $n=9$, more than $53\%$ of the matroids that are not sparse paving are elementary split. We propose the following conjecture.

\begin{conj}\label{conj:split-predominates}
    Let $\operatorname{sp}(n)$ and $\operatorname{es}(n)$ denote the number of sparse paving matroids and elementary split matroids on $[n]$, respectively. Let $\operatorname{sp}'(n)$ and $\operatorname{es}'(n)$ be the number of sparse paving matroids and elementary split matroids on $[n]$ up to isomorphism, respectively. Then
    \[ \lim_{n\to \infty} \frac{\operatorname{es}(n)-\operatorname{sp}(n)}{\operatorname{mat}(n)-\operatorname{sp}(n)} = 1 \;\;\;\;\; \text{and} \;\;\;\;\; \lim_{n\to \infty} \frac{\operatorname{es}'(n)-\operatorname{sp}'(n)}{\operatorname{mat}'(n)-\operatorname{sp}'(n)} = 1 .\]
\end{conj}

Essentially, our conjecture is asserting that the class of elementary split matroids is \emph{genuinely} larger than the class of sparse paving matroids. We expect elementary split matroids to predominate even if we restrict the enumeration to the class of non sparse paving matroids. Of course, this does not preclude the fact that there exist many other matroids featuring other properties. We include in Table~\ref{table:number-matroids} the number of isomorphism classes of matroids and split matroids on $n$ elements and fixed rank, for $n\leq 9$.

{\footnotesize
\begin{table}\label{table:}
    \begin{subtable}[t]{.5\textwidth}
        \caption{Number of isomorphism classes of matroids\\ on $[n]$ of rank $k$}
        \raggedright
        \begin{tabular}{l r r r r r r r r r r}\hline
        $k\backslash n$   & 1   & 2 & 3 & 4 & 5 & 6 & 7 & 8 & 9   \\ \hline
        0 &   1 & 1 & 1 & 1 & 1 & 1 & 1 & 1 & 1\\
        1 &   1 & 2 & 3 & 4 & 5 & 6 & 7 & 8 & 9\\
        2 &    & 1 & 3 & 7 & 13 & 23 & 37 & 58 & 87\\
        3 &    &  & 1 & 4 & 13 & 38 & 108 & 325 & 1275 \\
        4 &    &  &  & 1 & 5 & 23 & 108 & 940 & 190214 \\
        5 &    &  &  &  & 1 & 6 & 37 & 325 & 190214\\
        6 &    &  &  &  &  & 1 & 7 & 58 & 1275\\
        7 &    &  &  &  &  &  & 1  & 8 & 87 \\
        8 &    &  &  &  &  &  &  & 1 & 9 \\
        9 &    &  &  &  &  &  &  &  & 1 \\
        \end{tabular}
    \end{subtable}
    \begin{subtable}[t]{.5\textwidth}
        \caption{Number of isomorphism classes of elementary split matroids on $[n]$ of rank $k$}
        \begin{tabular}{l r r r r r r r r r r}\hline
         $k\backslash n$  & 1   & 2 & 3 & 4 & 5 & 6 & 7 & 8 & 9   \\ \hline
         0&   1 & 1 & 1 & 1 & 1 & 1 & 1 & 1 & 1\\
         1&   1 & 2 & 3 & 4 & 5 & 6 & 7 & 8 & 9\\
         2&    & 1 & 3 & 6 & 9 & 14 & 19 & 27 & 36\\
         3&    &  & 1 & 4 & 9 & 18 & 38 & 94 & 434 \\
         4&    &  &  & 1 & 5 & 14 & 38 & 407 & 154568 \\
         5&    &  &  &  & 1 & 6 & 19 & 94 & 154568\\
         6&    &  &  &  &  & 1 & 7 & 27 & 434\\
         7&    &  &  &  &  &  & 1  & 8 & 36 \\
         8&    &  &  &  &  &  &  & 1 & 9 \\
         9&    &  &  &  &  &  &  &  & 1 \\
        \end{tabular}
    \end{subtable}
    \caption{Number of isomorphism classes of matroids and elementary split matroids.}
    \label{table:number-matroids}
\end{table}
}

\subsection{Parameters and bounds}

The goal in this subsection is to establish some bounds on the number of stressed subsets or stressed hyperplanes that a sparse paving, paving, or elementary split matroid can possibly have. The technical results in this section will be needed only in specific parts of the sequel, so the reader can skip this part until necessary.

We denote by $\uplambda_{r,h}$ the number of stressed subsets of size $h$ and rank $r$ in a given (elementary split) matroid. The motivation is twofold. If one is interested in performing an exhaustive computer search, having upper bounds on the parameters $\uplambda_{r,h}$ provides a rough but useful criterion to terminate the search. Also, if one is interested in constructing ``complicated'' elementary split matroids, it is convenient to have larger values $\uplambda_{r,h}$, hence it is desirable to obtain lower bounds or estimations of how large they can actually be.

\begin{prop}\label{prop:bound-hyperplanes-paving}
    Let $\M$ be a paving matroid of rank $k$ and cardinality $n$. Let us denote by $\uplambda_h$ the number of hyperplanes of $\M$ of size $h$ and by $h_{\max}$ the size of the largest of those hyperplanes. Then,
    \[ \sum_{h\geq k} \uplambda_h \binom{h}{k} \leq \frac{h_{\max}-k+1}{n-k+1}\binom{n}{k}.\]
\end{prop}

\begin{proof}
    We apply a double counting argument on the bipartite graph whose nodes are the independent sets of size $k$ and $k-1$, and whose edges connect a $k$-set with a $(k-1)$-set whenever it is contained in the other.
    
    Given that the matroid $\M$ is paving, all the $\binom{n}{k-1}$ subsets of cardinality $k-1$ are independent. 
    Thus our bipartite graph has $\binom{n}{k-1}+|\mathscr{B}(\M)|$ nodes where $\mathscr{B}(\M)$ is the collection of bases of $\M$.
    
    Each independent set $I$ of rank $k-1$ is contained in a unique hyperplane $H$ (the flat spanned by $I$ itself), hence there are $n-|H|$ elements that we can add to the independent set $I$ and obtain a basis. This implies that the degree of the node $I$ in our graph is exactly $n-|H|$, and hence it is at least $n-h_{\max}$.
    
    The other nodes in our graph correspond to bases, each of these nodes has degree $k$. Thus the total number of edges of the bipartite graph is exactly $k|\mathscr{B}(\M)|$. We conclude
        \[ (n-h_{\max})\binom{n}{k-1} \leq k |\mathscr{B}(\M)|\enspace,\]
    which is equivalent to
        \[ \frac{n-h_{\max}}{n-k+1} \binom{n}{k} \leq |\mathscr{B}(\M)|\enspace .\]
    Since relaxing all the hyperplanes of the matroid $\M$ which are of size $h\geq k$ yields the uniform matroid $\U_{k,n}$, we have that
    \[ \binom{n}{k} = |\mathscr{B}(\U_{k,n})| = |\mathscr{B}(\M)| + \sum_{H\text{ hyperplane}} \left|\cover(H)\right| = |\mathscr{B}(\M)|+\sum_{h\geq k} \uplambda_h \binom{h}{k},\]
    where we used Remark~\ref{rem:cusp-corank1}. Hence we obtain the desired inequality
    \[ \sum_{h\geq k} \uplambda_h \binom{h}{k} = \binom{n}{k} - |\mathscr{B}(\M)| \leq \binom{n}{k} - \frac{n-h_{\max}}{n-k+1} \binom{n}{k}=\frac{h_{\max}-k+1}{n-k+1} \binom{n}{k}\enspace. \qedhere\]
\end{proof}

\begin{coro}\label{coro:bound-circuit-hyperplanes-sparse-paving}
    If $\M$ is a sparse paving matroid of rank $k$ and cardinality $n$ and with $\uplambda$ circuit-hyperplanes, then
    \[ \uplambda \leq \min\left\{\frac{1}{k+1}, \frac{1}{n-k+1}\right\} \binom{n}{k}.\]
\end{coro}

\begin{proof}
    Since $\M$ is sparse paving, it is in particular paving and the size of all hyperplanes is the same and equal to $h_{\max}=k$. In particular, the preceding formula simplifies to the following inequality:
        \[ \uplambda = \uplambda_k \leq \frac{1}{n-k+1}\binom{n}{k}.\]
    Moreover, since $\M^*$ is sparse paving as well with the same number $\uplambda$ of circuit-hyperplanes, the same reasoning shows that
        \[ \uplambda\leq \frac{1}{k+1}\binom{n}{n-k}.\]
    The result follows by considering the minimum of these two expressions.
\end{proof}

A natural question that might arise at this point is whether there exists a reasonable estimation or lower bound for the maximum number of circuit-hyperplanes that a sparse paving matroid of fixed rank and cardinality can have. The following result, due to Graham and Sloane, when rephrased in terms of matroids, allows us to get such a bound.

\begin{lemma}[{\cite[Theorem~1]{graham-sloane}}]\label{lem:bound-ch-sparsepaving}
    For every $1\leq k\leq n$ there exists a sparse paving matroid of rank $k$, cardinality $n$ and having at least $\frac{1}{n}\binom{n}{k}$ circuit-hyperplanes.
\end{lemma}

The equivalence of the above statement with the precise formulation of \cite[Theorem~1]{graham-sloane} is explained in detail in \cite[Section~5]{ferroni3}.

Now, let us focus our attention to the case of elementary split matroids. We will state now some inequalities that the numbers of stressed subsets with non-empty cover of each size and rank in an elementary split matroid must satisfy.

\begin{lemma}\label{lemma:bound-lambdarhs}
    Let $\M$ be an elementary split matroid of rank $k$ and cardinality $n$. Let $\uplambda_{r,h}$ denote the number of stressed subsets with non-empty cover of rank $r$ and size $h$. Then:
        \[ \sum_{r,h}\uplambda_{r,h} \sum_{i=r+1}^k \binom{h}{i}\binom{n-h}{k-i} \leq \binom{n}{k}.\]
\end{lemma}

\begin{proof}
    If we relax a stressed subset $F$ with non-empty cover having rank $r$ and cardinality $h$, the number of bases of $\M$ increases by $|\cover(F)| = \sum_{i=r+1}^k \binom{h}{i}\binom{n-h}{k-i}$. Furthermore, if we relax all such subsets, we end up with the uniform matroid $\U_{k,n}$, which has $\binom{n}{k}$ bases. In other words,
        \[ |\mathscr{B}(\M)| + \sum_{r,h}\uplambda_{r,h} \sum_{i=r+1}^k \binom{h}{i}\binom{n-h}{k-i} = \binom{n}{k},\]
    from which the inequality in the lemma follows.
\end{proof}

We observe that if one particularizes the preceding bound to a paving matroid $\M$, one obtains a much weaker version of the one in Proposition~\ref{prop:bound-hyperplanes-paving}. 

The statement in Lemma~\ref{lemma:bound-lambdarhs} is particularly useful in the context of performing a computer search that considers all elementary split matroids of fixed rank and cardinality. Note that if a vector of parameters $(\uplambda_{r,h})_{r,h}$ comes from a certain elementary split matroid $\M$ of rank $k$ and cardinality $n$, then decreasing any individual entry of this vector yields a new vector that again comes from an elementary split matroid (which is of course obtained by relaxing a stressed subset).

\subsection{Remarks on the number of stressed subsets}

Before ending this section, let us digress on the difficulty of deciding if a given vector $(\uplambda_{r,h})_{r,h}$ counts the number of stressed subsets with non-empty cover of some elementary split matroid. As we will argue below, this problem is expected to be very difficult. Equivalently, it is hard to determine the entry-wise ``maximal'' vectors $(\uplambda_{r,h})_{r,h}$ for given values of $k$ and $n$. However, inequalities such as the one in Lemma~\ref{lemma:bound-lambdarhs} serve as a rough but useful criterion to discard vectors that certainly do not come from elementary split matroids. 

Let us be slightly more precise. We fix numbers $k$ and $n$, and restrict our attention exclusively to sparse paving matroids of rank $k$ and cardinality $n$. Assume that we are given a number $\uplambda$. Then, we want to know whether there exists a sparse paving matroid having exactly $\uplambda$ circuit-hyperplanes. To approach that question we provided two bounds:
\begin{itemize}
    \item Corollary~\ref{coro:bound-circuit-hyperplanes-sparse-paving}, which says that if $\uplambda>\min\left\{\frac{1}{k+1},\frac{1}{n-k+1}\right\}\binom{n}{k}$ then the answer is ``no''.
    \item Lemma~\ref{lem:bound-ch-sparsepaving}, which guarantees that if $\uplambda\leq \frac{1}{n}\binom{n}{k}$ then the answer is ``yes''.
\end{itemize}

We observe that there is a gap, i.e., $\frac{1}{n}\binom{n}{k} < \uplambda \leq \min\left\{\frac{1}{k+1},\frac{1}{n-k+1}\right\}\binom{n}{k}$ for which the answer is in general very difficult to obtain. Although both the upper and lower bound can be slightly improved, the computation of the maximum number of circuit-hyperplanes that a sparse paving matroid of rank $k$ and cardinality $n$ can have is a notoriously difficult problem. 
The reason is that this is equivalent to many other well-known problems; see \cite[Theorem~26]{joswig-schroter}.
One is the calculation of the size of the maximum stable subset of the Johnson graph $\mathrm{J}(n,k)$, which is the $1$-skeleton of the hypersimplex $\Delta_{k,n}$.
Computing the independence number of $\mathrm{J}(n,k)$, i.e., the size of the maximum stable subset, would have consequences on the enumeration of matroids; see \cite{bansal-pendavingh-vdp}. 
Furthermore, this number coincides with the maximum number of binary words of fixed length $n$ and constant weight $k$, under the additional assumption that the minimal Hamming distance between any two distinct words is at least $4$. This number is tabulated for small values of $n$ and $k$ (see for example \cite[Table~1]{brouwer-etzion}). As a small glimpse of how hard this is expected to be, we point out that its exact value is not known for values as small as $n=18$ and $k=9$.

\begin{problem}
    Let us fix numbers $k$ and $n$. Assume that we are given a list of non-negative integers $(\uplambda_{r,h})_{0\leq r\leq h\leq n}$. Does there exist an (elementary) split matroid of rank $k$ and cardinality $n$ with exactly $\uplambda_{r,h}$ stressed subsets with non-empty cover of size $h$ and rank $r$?
\end{problem}

As said before, the general answer to this question is expected to be exceedingly difficult as well, as this includes the case we discussed above. Nonetheless, we feel motivated to ask the following questions.

\begin{itemize}
    \item Is it possible to obtain a (significant) improvement of the upper bound of Lemma~\ref{lemma:bound-lambdarhs}?
    \item Do there exist any versions of Lemma~\ref{lem:bound-ch-sparsepaving} for elementary split matroids, i.e., lower bounds for the entry-wise maximal vectors $(\uplambda_{r,h})_{r,h}$?
\end{itemize}

\section{Valuations on split matroids}
\noindent  The main goal of this section is to prove Theorem~\ref{thm:main}. This result will be obtained as a consequence of the family of subdivisions arising from the relaxations of stressed subsets described in Theorem~\ref{thm:relaxation-induces-subdivision}. The results in this section lay the groundwork for the remainder of the paper, in which we will apply Theorem~\ref{thm:main} for some specific invariants. 

In this section we also study a subdivision using graphic matroids for certain matroids of corank $2$ that will be needed later in the paper. In addition, we include a short discussion about the problem of pursuing generalizations of Theorem~\ref{thm:main} to encompass more matroids.

\subsection{Valuative invariants for split matroids}

Our goal is to put together some of the ingredients that we developed so far. The short version of the story is that evaluating a valuative invariant on a split matroid is a task that can be achieved by knowing only certain (often tractable) data. Moreover, the next statement describes the change of a valuation when relaxing a stressed subset with non-empty cover in an arbitrary matroid. 

\begin{prop}\label{prop:change-of-invariant-with-relaxation}
    Let $\M$ be a matroid of rank $k$ and cardinality $n$ with a stressed subset $F$ whose cover is non-empty. Let $\widetilde{\M}$ be the relaxation $\Rel(\M,F)$, $|F|=h$ and $\rank_\M(F)=r$, then for every valuative invariant $f:\mathfrak{M}\to \A$ the following holds
        \[ f(\widetilde{\M}) = f(\M) + f(\LL_{r,k,h,n}) - f(\U_{k-r,n-h}\oplus \U_{k,h})\enspace .\]
\end{prop}

\begin{proof}
    Let us first assume that the matroid $\M$ is disconnected. Assume that $E=[n]$ and $F=\{n-h+1,\ldots,n\}$ for simplicity. 
    By definition of being a stressed subset, the restriction $\M|_F \cong \U_{r,h}$ and contraction $\M/F = \U_{k-r,n-h}$ must both be uniform matroids. This implies that the bases of $\U_{k-r,n-h}\oplus\U_{r,h}$ are bases of $\M$, too.
    There is a single disconnected matroid that contains all these bases, i.e., $\M = \U_{k-r,n-h}\oplus\U_{r,h}$. 
    The matroid $\M$ has a two stressed subsets corresponding to the two connected components. The statement requires that the stressed set $F$ is of size $h$ and rank $r$ therefore $\Rel(\M,F) \cong \LL_{r,k,h,n}$ and the statement becomes trivial.

    Now let us assume that $\M$ is connected, then again the statement is a consequence of the subdivision induced by the relaxation of $F$, in the lines of Theorem~\ref{thm:relaxation-induces-subdivision}. Preserving the notation of that statement, we have that the set $\mathcal{S}$ consisting of all the faces of $\mathscr{P}(\M)$ and all the faces of $\mathscr{P}(\N_1)$ induces a subdivision of $\mathscr{P}(\widetilde{\M})$. We have only three internal faces which are given by $\mathcal{S}^{\operatorname{int}} = \left\{\mathscr{P}(\M), \mathscr{P}(\N_1), \mathscr{P}(\N_2)\right\}$. In particular, the proof follows from $f$ being a valuative invariant, so that can leverage the isomorphisms $\N_1\cong \LL_{r,k,h,n}$ and $\N_2\cong \U_{k-r,n-h}\oplus \U_{r,h}$.
\end{proof}

\begin{example}\label{ex:direct-sum-of-unif-is-particular-case}
    Proposition~\ref{prop:change-of-invariant-with-relaxation} tells us that evaluating any valuative invariant $f:\mathfrak{M}\to \A$ on the matroid $\U_{k-r,n-h}\oplus\U_{r,h}$ can be done by using the identity
    \[
    f(\U_{k-r,n-h}\oplus\U_{r,h}) = 
    f(\LL_{k-r,k,n-h,n}) + f(\LL_{r,k,h,n}) - f(\U_{k,n})
    \]
    which can be seen as a consequence of Theorem~\ref{thm:relaxation-induces-subdivision} with $\M=\LL_{k-r,k,n-h,n}$. 
    See also Example~\ref{ex:matroidal_subdivision} and Figure~\ref{fig:oct_subdivision} for the case $r=1$, $k=h=2$, and $n=4$.
\end{example}

We stated a characterization of elementary split matroids in Theorem~\ref{thm:elem-split-relaxation-yields-uniform}. This, when combined with the prior statement, yields the following result.

\begin{teo}\label{thm:main}
    Let $\M$ be an elementary split matroid of rank $k$ and cardinality $n$, and let $f$ be a valuative invariant. Then,
    \begin{align*}
     f(\M) &= f(\U_{k,n}) - \sum_{r,h} \uplambda_{r,h} \left(f(\LL_{r,k,h,n}) - f(\U_{k-r,n-h}\oplus \U_{r,h})\right)\\
     &= f(\U_{k,n}) - \sum_{r,h} \uplambda_{r,h} \left(f(\U_{k,n}) - f(\LL_{k-r,k,n-h,n})\right),
     \end{align*}
    where $\uplambda_{r,h}$ denotes the number of stressed subsets with non-empty cover of size $h$ and rank $r$.
\end{teo}

\begin{proof}
    Since $\M$ is elementary split, by relaxing its stressed subsets with non-empty cover $\{F_1,\ldots, F_m\}$ one after each other, we end up with a sequence
    \[
    \M = \M_0,\, \M_1,\, \ldots,\, \M_m \cong \U_{k,n}
    \]
    where $\M_i = \Rel(\M_{i-1}, F_i)$ for each $1\leq i\leq m$.
    It follows that for each $1\leq i\leq m$, we have by Proposition~\ref{prop:change-of-invariant-with-relaxation}
        \[ f(\M_i) = f(\M_{i-1}) + f(\LL_{r_i,k,h_i,n}) - f(\U_{k-r_i,n-h_i}\oplus \U_{r_i,h_i})\]
    where $r_i$ is the rank of the cyclic flat $F_i$ and $h_i=|F_i|$ the size of it. 
    Expanding the expressions recursively, we obtain a formula for $f(\M_m)$, namely
    \[ f(\U_{k,n}) = f(\M) + \sum_{r,h} \uplambda_{r,h} \left(f(\LL_{r,k,h,n}) - f(\U_{k-r,n-h} \oplus \U_{r,h})\right),\]
    from where the first of the two formulas of the statement follows. Example~\ref{ex:direct-sum-of-unif-is-particular-case} allows us to express $f(\U_{k-r,n-h}\oplus\U_{r,h})$ in terms of valuations of Schubert elementary split (i.e., cuspidal) matroids, which one can substitute into the first formula to get an expression in terms of a linear combination of evaluations of cuspidal matroids.
\end{proof}

In the case that our matroid $\M$ is sparse paving, we can simplify the previous expression to compute a valuative invariant $f$. This is because the only stressed subsets with non-empty cover of $\M$ are precisely the circuit-hyperplanes.

\begin{coro}\label{thm:main-for-sparse-paving}
    Let $\M$ be a sparse paving matroid of rank $k$ and cardinality $n$ and having exactly $\uplambda$ circuit-hyperplanes. Let $f$ be a valuative invariant. Then
    \[ f(\M) = f(\U_{k,n}) - \uplambda \left(f(\mathsf{T}_{k,n}) - f(\U_{1,n-k}\oplus\U_{k-1,k})\right).  \]
\end{coro}

\begin{proof}
    In a sparse paving matroid, being a set with non-empty cover is equivalent to being a non-basis, because all the subsets of size $k+1$ span the matroid and all the subsets of size $k-1$ are independent. 
    In particular, since for such a matroid we additionally know that the non-bases are exactly the circuit-hyperplanes (and are therefore stressed), we can simplify the sum of Theorem~\ref{thm:main}, because the only non-zero $\uplambda_{r,h}$ correspond to $r=k-1$ and $h=k$. Therefore,
    \[ f(\M) = f(\U_{k,n}) - \uplambda \left( f(\LL_{k-1,k,k,n}) - f(\U_{1,n-k}\oplus \U_{k-1,k})\right).\]
    Recall that we have an isomorphism $\LL_{k-1,k,k,n}\cong \mathsf{T}_{k,n}$, so the proof is complete.
\end{proof}

The moral in the preceding results is that if one knows how to compute a valuative invariant $f$ at cuspidal matroids, then one ``knows'' how to compute it for all elementary split matroids. The task is somewhat simpler if one is only interested in paving or  sparse paving matroids, as the amount of data needed is significantly smaller. 

\subsection{A subdivision into graphic matroids for certain corank 2 matroids}\label{sec:subdivision-rank2}

This subsection will provide a technical subdivision for certain matroids having corank $2$. These subdivisions will play a crucial role in the proof of Theorem~\ref{thm:kl-z-corank2}, which states an explicit formula for the Kazhdan--Lusztig and $Z$-polynomials for corank $2$ matroids. Since this will not be needed elsewhere, the reader may postpone the reading of this part until reaching Section~\ref{sec:kl-polys}.

For each pair of integers $a,b\geq 2$, let us denote by $\mathsf{C}_{a,b}$ the graphic matroid obtained by gluing two cycles of length $a$ and $b$ along a common edge (see Figure \ref{fig:two-cycles}); whenever $a$ or $b$ are equal to $2$, we understand that a cycle of length $2$ consists of two parallel edges joining a pair of vertices. The matroid $\mathsf{C}_{a,b}$ has cardinality $a+b-1$ and its rank is $a+b-3$. We may assume that the edges of $\mathsf{C}_{a,b}$ are labelled such that the numbers $1,\ldots,a$ are assigned to the edges of the cycle of length $a$, and the edges of the cycle of length $b$ are assigned the numbers $a,a+1,\ldots,a+b-1$.

    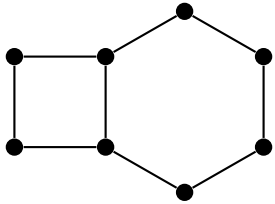
\begin{figure}
        \centering
		\begin{tikzpicture}  
		[scale=0.6,auto=center,every node/.style={circle, fill=black, inner sep=2.3pt},every edge/.append style = {thick}] 
	    \tikzstyle{edges} = [thick];
		\node (a1) at (0,0) {};  
		\node (a2) at (0,2)  {}; 
		\node (a3) at (1.7320,3)  {};  
		\node (a4) at (2*1.7320,2) {};  
		\node (a5) at (2*1.7320,0)  {};  
		\node (a6) at (2*1.7320/2,-1)  {};    
		\node (a7) at (-2,2) {};
		\node (a8) at (-2,0) {};
		\draw[edges] (a1) -- (a2); 
		\draw[edges] (a2) -- (a3);  
		\draw[edges] (a3) -- (a4);  
		\draw[edges] (a4) -- (a5);  
		\draw[edges] (a5) -- (a6);
		\draw[edges] (a6) -- (a1);
		\draw[edges] (a1) -- (a8);
		\draw[edges] (a2) -- (a7);
		\draw[edges] (a7) -- (a8);
		\end{tikzpicture} \caption{The graphic matroid $\mathsf{C}_{4,6}$.}\label{fig:two-cycles}
    \end{figure}

\begin{lemma}\label{lem:mat_poly_Cab}
    The base polytope of the matroid $\mathsf{C}_{a,b}$ is 
    \[\mathscr{P}(\mathsf{C}_{a,b}) = \left\{ x\in [0,1]^n : \sum_{i=1}^n x_i = n-2,\enspace a-2 \leq \sum_{i=1}^{a-1} x_i,\enspace 
    \sum_{i=1}^a x_i \leq a-1 \right\},\]
    where $n = a+b-1$.
\end{lemma}

\begin{proof}
    The last two inequalities are redundant whenever $a=b=2$ and the claim holds in this case as $\mathsf{C}_{2,2}=\U_{1,3}$. Otherwise, the matroid $\mathsf{C}_{a,b}$ has exactly three circuits, namely
    $A=\{1,\ldots,a\}$, $B=\{a,\ldots,n\}$ and $\{1,\ldots,n\}\smallsetminus\{a\}$. They lead to the four cyclic flats $\varnothing$, $A$, $B$ and $A\cup B$. 
    While the base polytope of every rank $n-2$ matroid satisfies the first equation and is a $0/1$-polytope, the other two inequalities are a direct consequence of the expression \eqref{eq:bases-from-cyclic-flats} where we used that
    \[
    \sum_{i=a}^n x_i \leq b-1 \Longleftrightarrow
    \sum_{i=1}^{a-1} x_i \geq a-2. \qedhere
    \]
\end{proof}

\begin{prop}\label{prop:kl-subdivision}
    Let $n\geq 3$ and $r\leq n-3$. Then the collection 
    \[\mathcal{S} = \left\{\, F \; : \; F \text{ is a face of }\, \mathscr{P}(\mathsf{C}_{a,n+1-a}) \; \text{ for some } 2\leq a \leq n-r-1\right\}\]
    is a matroid subdivision of the matroid base polytope $\mathscr{P}(\LL_{r,n-2,r+1,n})$ of the corank $2$ matroid $\LL_{r,n-2,r+1,n}$.
    Moreover, the set $\mathcal{S}^{\text{\normalfont int}}$ consists exactly of the base polytopes of $\mathsf{C}_{a,n+1-a}$ and their pairwise intersections.
\end{prop}

\begin{proof} 
    If $r=n-3$ then $a=2$ and $\mathsf{C}_{2,n-1} = \LL_{n-3,n-2,n-2,n}$ which implies the claim for this case. From now on let us assume
    $r<n-3$ and $2\leq a < a' \leq n-r-1$.
    Lemma~\ref{lem:mat_poly_Cab} shows that $\mathscr{P}(\mathsf{C}_{a,n+1-a})\cap \mathscr{P}(\mathsf{C}_{a',n+1-a'})$ is empty whenever $a+1<a'$ and otherwise $a'=a+1$ and 
    \begin{equation}\label{eq:faces_sub}
    \mathscr{P}(\mathsf{C}_{a,n+1-a})\cap \mathscr{P}(\mathsf{C}_{a',n+1-a'}) = \left\{ x\in [0,1]^n : \sum_{i=1}^n x_i = n-2,\enspace \sum_{i=1}^a x_i = a-1 \right\} .\end{equation}
    This already proves that $\mathcal{S}$ is a polyhedral complex as by construction faces of cells are in $\mathcal{S}$ and intersections of cells are faces.
    Furthermore, comparing the polytopes in Lemma~\ref{lem:mat_poly_Cab} with those of Corollary~\ref{coro:cuspidal-polytope} implies that $\mathscr{P}(\mathsf{C}_{a,n+1-a})$ is contained in $\mathscr{P}(\LL_{r,n-2,r+1,n})$ and the only faces that are not in the boundary are of the form \eqref{eq:faces_sub}.
    This implies the claim on the inner faces. Moreover, the argument shows that all inner codimension-$1$ faces of the collection $\mathcal{S}$ are contained in two maximal cells. Hence the collection $\mathcal{S}$ covers the polytope $\mathscr{P}(\LL_{r,n-2,r+1,n})$ and thus the collection $\mathcal{S}$ is a subdivision of this polytope.
\end{proof}

The following is an immediate consequence of Proposition~\ref{prop:kl-subdivision}.

\begin{coro}\label{coro:valuative-invariant-corank2}
    Let $f:\mathfrak{M}\to \A$ be a valuative invariant. Then
    \[ f\left( \LL_{r,n-2,r+1,n}\right) = \sum_{a=2}^{n-r-1} f(\mathsf{C}_{a,n+1-a}) - \sum_{a=2}^{n-r-2} f(\U_{a-1,a}\oplus \U_{n-a-1, n-a}).\]
\end{coro}

\subsection{Towards a generalization of techniques} 

Let us now discuss briefly the following two reasonable questions that one may ask:
    \begin{enumerate}[(1)]
        \item Is there a version of Theorem~\ref{thm:main} for a larger class of matroids?
        \item Is it possible to generalize the notion of relaxation to an even larger class of subsets?
    \end{enumerate}

The na\"ive answer to the first question is ``yes'', and in fact it is something already hinted by Theorem~\ref{thm:schubert-basis}. 
Essentially, in the valuative group we can write any matroid $\M$ as an integer combination of Schubert matroids of the same cardinality and rank. Following Hampe \cite{hampe}, we can give an elegant combinatorial description of the coefficients of this linear combination, as we now explain.

\begin{defi}
    Let $\M$ be a matroid and let $\mathscr{Z}(\M)$ be its lattice of cyclic flats. The \emph{cyclic chain lattice} of~$\M$ is defined as the lattice $\mathcal{C}_{\mathscr{Z}}(\M)$ whose elements are all the chains of $\mathscr{Z}(\M)$ that contain the minimal $0_{\mathscr{Z}}$  and maximal $1_{\mathscr{Z}}$ element of the lattice $\mathscr{Z}(\M)$; and an additional top element, denoted by~$\widehat{\mathbf{1}}$.
\end{defi}

It is not difficult to show that $\mathcal{C}_{\mathscr{Z}}$ is indeed a graded lattice, whose meet is given by the set-intersection and the join is the set-union when the union is indeed a chain and $\widehat{\mathbf{1}}$ otherwise. Let us consider the M\"obius function $\mu$ of this poset. To each element $\mathrm{C}\in \mathcal{C}_{\mathscr{Z}}(\M)$ we can associate the number $\uplambda_{\mathrm{C}} = -\mu(\mathrm{C},\widehat{\mathbf{1}})$.
Furthermore, since each of these elements $\mathrm{C}$ is a chain of cyclic flats, there is a unique Schubert matroid $\mathsf{S}_{\mathrm{C}}$ whose lattice of cyclic flats coincides with $\mathrm{C}$. 
A generalization of Theorem~\ref{thm:main} that includes all loopless matroids is hence the following.

\begin{teo}[\cite{hampe}]\label{thm:hampe}
    Let $\M$ be a loopless matroid and let $f:\mathfrak{M}\to \A$ be a valuative invariant. Then,
    \[ f(\M) = \sum_{\substack{\mathrm{C}\in\mathcal{C}_{\mathscr{Z}}(\M)\\\mathrm{C} \neq \widehat{\mathbf{1}}}} \uplambda_{\mathrm{C}}\, f(\mathsf{S}_{\mathrm{C}}).\]
\end{teo}

As we briefly mentioned above in the last paragraph of Section \ref{sec:polytopes}, using this result it is possible to derive an alternative proof of Theorem~\ref{thm:main}. For an elementary split matroid all the resulting Schubert matroids $\mathsf{S}_{\mathrm{C}}$ are precisely the matroids $\LL_{r,k,h,n}$, and the coefficients $\uplambda_{\mathrm{C}}$ count the number of stressed subsets with non-empty cover of each rank and size. 

The reader may wonder why not to use this more general version in most of the calculations. Let us explain the reasons for this. On the one hand, using the general version forces us to be able to evaluate $f$ on an \emph{arbitrary} Schubert matroid. This is often too complicated and might be as difficult as evaluating $f$ on an arbitrary matroid. 
For instance, it is challenging to compute the Kazhdan--Lusztig or the $g$-polynomial of an arbitrary cuspidal matroid, let alone considering a general Schubert matroid. Hence, if we really want to use Hampe's result to do computations, we are often forced to restrict ourselves to some more tractable subclass of Schubert matroids.

On the other hand, the geometry behind Theorem~\ref{thm:hampe} is more complicated; there is not a \emph{subdivision} of $\mathscr{P}(\M)$ that explains the formula. Moreover, the fact that the numbers $\uplambda_{\mathrm{C}}$ can be positive or negative introduces an additional issue. For example, it would be very hard to obtain a result such as Lemma~\ref{lemma:bound-lambdarhs}.

The class of loopless and coloopless elementary split matroids is particularly tractable because, as was stated in Theorem~\ref{thm:kristof-equivalences}, any two proper flats are incomparable, i.e., the maximal chains of cyclic flats have length at most $3$; the class of Schubert matroids that we need to consider in this case are indexed by four parameters, $r$, $k$, $h$ and $n$. If we go beyond and consider the even larger class of matroids whose lattice of cyclic flats has maximal chains of length at most $4$, then the Schubert matroids appearing in Hampe's result will depend on six parameters instead of four.

Another advantage of the restriction to elementary split matroids is that we can understand the geometry of their polytopes. This will be apparent to the reader when we address the Kazhdan--Lusztig polynomial of a corank $2$ matroid leveraging the decompositions that we proved in Proposition~\ref{prop:kl-subdivision}. In that statement, we needed to rely on the subdivisions of the base polytope of $\LL_{r,n-2,r+1,n}$ and compute the invariant for the smaller pieces. In other words, sometimes the geometric interpretation of the polytope suggests a subdivision that the cyclic chain lattice does not foresee.

\section{A toolkit for matroid valuations}\label{sec:polytopes}
\noindent We interleave our discussion of elementary split matroids to take a look at general facts regarding matroid valuations. We will present two technical tools: a general family of valuations constructed from flags of sets in the matroid, and a strengthening of the convolution theorem of Ardila and Sanchez (see Theorem~\ref{thm:ardila-sanchez-strengthening} below). Since later we will be interested in discussing how to evaluate specific invariants on the class of elementary split matroids, we will often require to show first that they are indeed valuative, so the discussion in this section will provide us with adequate machinery.

\subsection{Valuations from flags of sets}

The proof of Theorem~\ref{thm:g-invariant-universal} in \cite{derksen-fink} is intertwined with the proof of the valuativeness of other very useful functions constructed from flags of sets. Let us recapitulate this, as in fact it provides a spanning set for the space of all valuations $\mathfrak{M}_E\to \mathbb{Z}$. 

Let us fix $E$, and denote by the number $n$ the size of $E$. Let $m\leq n$ be any non-negative integer. For each flag of sets $\varnothing\subseteq X_1\subsetneq X_2\subsetneq\cdots \subsetneq X_m\subseteq E$ and each weakly increasing sequence of numbers $0\leq r_1\leq \cdots \leq r_m \leq n$, consider the map $\mathrm{S}_{X,r}:\mathfrak{M}_E\to \mathbb{Z}$ defined by
    \begin{equation} \label{eq:def-s-derksen-fink}
    \mathrm{S}_{X,r}(\M) = \begin{cases} 1 & \text{if } \rk_\M(X_i) = r_i \text{ for } i=1,\ldots,m\\
    0 & \text{otherwise}.\end{cases}
    \end{equation}

\begin{teo}[{\cite[Proposition~5.3 \& Corollary~5.6]{derksen-fink}}]\label{thm:flag-sets-rank}
    The function $\mathrm{S}_{X,r}:\mathfrak{M}_E\to \mathbb{Z}$ described above is a valuation. Moreover, every valuation $f:\mathfrak{M}_E\to \mathbb{Z}$ can be written as an integer combination of functions $\mathrm{S}_{X,r}$ for flags $X$ of length $n=|E|$.
\end{teo}

The utility of the preceding statement must not be underestimated. As the reader will see below, we use it in an instrumental way to prove the valuativeness of several invariants appearing in Hodge theory. Moreover, Theorem~\ref{thm:ardila-sanchez-strengthening} appearing below can be seen as an application of this fact. Before focusing our attention there, let us state one more consequence of Theorem~\ref{thm:flag-sets-rank}. This establishes the valuativeness of a very useful family of functions.

\begin{teo}\label{thm:flags-flats-rank-valuation}
    Let $F_0\subsetneq F_1\subsetneq \cdots \subsetneq F_m\subseteq E$ be any fixed flag of sets. The map $\widehat{\upphi}_{F_0,\ldots,F_m}:\mathfrak{M}_E\to \mathbb{Z}$ defined via
    \begin{equation} \label{eq:fixed-flag-flats-valuation-arb-length}
        \widehat{\upphi}_{F_0,\ldots,F_m}(\M) = \begin{cases}
        1 & \text{if } F_i\in\mathcal{L}(\M) \text{ for all } 0\leq i\leq m,\\
        0 & \text{otherwise,}\end{cases}
    \end{equation}
    is a valuation. Moreover, if we consider any vector $(r_0,\ldots,r_m)$, the map $\upphi_{\substack{F_0,\ldots,F_m\\r_0,\ldots, r_m}}:\mathfrak{M}_E\to \mathbb{Z}$ defined by
    \begin{equation} \label{eq:fixed-flag-flats-valuation-arb-length-and-rank}
        \upphi_{\substack{F_0,\ldots,F_m\\r_0,\ldots, r_m}}(\M) = \begin{cases}
        1 & \text{if } F_i\in\mathcal{L}(\M) \text{ and $\rk(F_i)=r_i$ for all } 0\leq i\leq m,\\
        0 & \text{otherwise,}\end{cases}
    \end{equation}
    is a valuation too.
\end{teo}

\begin{proof}
    If we prove that the map $\upphi_{\substack{F_0,\ldots,F_m\\r_0,\ldots, r_m}}(\M)$ is valuative, then one can deduce that $\widehat{\upphi}_{F_0,\ldots,F_m}$ is a valuation, as the latter is a sum over several instances of the former. For an arbitrary flag of subsets $F_0\subsetneq F_1\subsetneq \cdots F_m\subsetneq E$, and numbers $r_0,\ldots,r_m$, consider the map $\mathrm{S}_{\substack{F_0,\ldots,F_m\\r_0,\ldots, r_m}}:\mathfrak{M}_E\to \mathbb{Z}$ defined by
    \begin{equation}\label{eq:chain-fixed-rank}
    \mathrm{S}_{\substack{F_0,\ldots,F_m\\r_0,\ldots, r_m}}(\M) = \begin{cases}
        1 & \text{if } \rk(F_i)=r_i \text{ for all } 0\leq i\leq m,\\
        0 & \text{otherwise.}\end{cases}
    \end{equation}
    This is of course a more precise way of presenting the maps of equation~\eqref{eq:def-s-derksen-fink}. Now, for each subset $A\subseteq E$ and each number $r$, consider the map $\mathrm{I}_{A,r} :\mathfrak{M}_E\to \mathbb{Z}$ defined by
    \[ \mathrm{I}_{A,r}(\M) = \begin{cases}
        1 & \text{if $A\in\mathcal{L}(\M)$ and $\rk(A) = r$}\\
        0 & \text{otherwise.}
    \end{cases}\]
     Notice that $A$ being a flat means that for every subset $A'\supsetneq A$ we have $\rk(A) < \rk(A')$. In particular, this implies that we can write
        \[\mathrm{I}_{A,r}(\M) = \sum_{A'\supseteq A} (-1)^{|A'\smallsetminus A|} \mathrm{S}_{\substack{A, A'\\r,r}}(\M).\]
    It follows that
        \begin{equation}\label{eq:prod-chains-fixed-rank} \upphi_{\substack{F_0,\ldots,F_m\\r_0,\ldots, r_m}}(\M) = \prod_{i=0}^m \mathrm{I}_{F_i,r_i} = \prod_{i=0}^m \sum_{G_i \supseteq F_i} (-1)^{|G_i \smallsetminus F_i|} \mathrm{S}_{\substack{F_i,G_i\\r_i,r_i}}(\M) = \sum \prod_{i=0}^m (-1)^{|G_i \smallsetminus F_i|} \mathrm{S}_{\substack{F_i,G_i\\r_i,r_i}}(\M),
        \end{equation}
    where the last sum on the right is over all choices of subsets $G_0,\ldots,G_m$ such that $G_i\supseteq F_i$ for each~$i$. Assume that for some $j\leq m-1$ we have that $G_j\not\subseteq F_{j+1}$ and, after fixing any linear ordering on $E$, choose the smallest element $x\in F_{j+1}\smallsetminus F_j$. If $x\in G_j$ then consider the set $G_j' = G_j\smallsetminus\{x\}$; otherwise, $x\not\in G_j$ and consider $G_j' = G_j\cup\{x\}$. We claim that the terms corresponding to the choice $(G_0,\ldots,G_m)$ and the choice $(G_0, \ldots, G_{j-1}, G'_j, G_{j+1},\ldots,G_m)$ cancel out, and this happens exactly once. Indeed, this follows directly from the fact that their signs differ and that \[\mathrm{S}_{\substack{F_j,G_j\\r_j,r_j}} = \mathrm{S}_{\substack{F_j,G'_j\\r_j,r_j}}.\]
    
    It follows from the above paragraph that we can assume that $G_i\subseteq F_{i+1}$ for every $i\leq m-1$.
    Now, observe that this allows us to simplify equation \eqref{eq:prod-chains-fixed-rank} as follows
    \[ \upphi_{\substack{F_0,\ldots,F_m\\r_0,\ldots, r_m}}(\M) = \sum (-1)^{\upepsilon(G_0,\ldots,G_m)} \mathrm{S}_{\substack{F_0,G_0,\ldots,F_m,G_m\\r_0,r_0,\ldots,r_m,r_m}}(\M),\]
    where $(-1)^{\upepsilon(G_0,\ldots,G_m)}$ is a sign depending on the $G_i$ and the sum is over all the choices of subsets $G_0,\ldots,G_m$ such that $F_0\subseteq G_0 \subseteq F_1 \subseteq \cdots \subseteq F_m\subseteq G_m$. This shows that our map $\upphi_{\substack{F_0,\ldots,F_m\\r_0,\ldots, r_m}}(\M)$ is a linear combination of valuations and is therefore a valuation.
\end{proof}

\begin{remark}
    A variation of Theorem~\ref{thm:flags-flats-rank-valuation} for arbitrary flags of cyclic flats can be found in \cite[Lemma~5.3]{fink-olarte} by Fink and Olarte. In fact the proof we just gave is strongly inspired by theirs. 
\end{remark}

\subsection{Convolutions of valuations}

Throughout the past decades it has been gradually realized that ideas from Hopf theory can be used to describe structural properties of valuations in matroid theory. For a more precise landscape, we suggest the reader to take a look at the work of Joni and Rota \cite{joni-rota}, Schmitt \cite{schmitt94}, and Aguiar and Ardila \cite{aguiar-ardila}. Recently, Ardila and Sanchez \cite{ardila-sanchez} used a Hopf-theoretic framework to establish a very handy result that explains in a succinct and straightforward way the valuativeness of several matroid functions.  

Let $\mathrm{R}$ be a ring. Let $f,g:\mathfrak{M}\to \mathrm{R}$ be two maps. The \emph{convolution} of $f$ and $g$ is the map $f\star g : \mathfrak{M}\to \mathrm{R}$ such that
        \[ (f\star g)(\M) = \sum_{S\subseteq E} f(\M|_S)\, g(\M/S)\]
    for every matroid $\M$ with ground set $E$.

Of course, the convolution operation is not commutative in general. Its importance stems from the fact that some invariants of matroids are obtained as convolutions of other easier invariants. We present examples for this approach in some of the sections below. 

\begin{remark}\label{rem:convolutions-loops}
    Many maps $g:\mathfrak{M}\to \mathrm{R}$ of matroids satisfy that $g(\M) = 0$ whenever $\M$ has loops. 
    An example is the characteristic polynomial $\chi:\mathfrak{M}\to \mathbb{Z}[t]$ assigning to each matroid~$\M$ its characteristic polynomial.
    In this case, the convolution of any other invariant $f$ (with or without that property) with $g$ depends only on the flats of $\M$. More precisely 
        \[ (f\star g)(\M) = \sum_{F\in \mathcal{L}(\M)} f(\M|_F)\, g(\M/F),\]
    because $\M/S$ is loopless if and only if $S$ is a flat. 
\end{remark}

One of the main results of Ardila and Sanchez establishes that the convolution of two valuations is again a valuation. Let us record this fact.

\begin{teo}[{\cite[Theorem~C]{ardila-sanchez}}]\label{thm:convolutions1}
    Let $\mathrm{R}$ be a ring. If $f,g:\mathfrak{M}\to \mathrm{R}$ are valuations, then $f\star g$ is a valuation too.
\end{teo}

In the remainder of this subsection we present an alternative proof of this result. Furthermore, we prove a strengthening that is of some independent interest. To the best of our knowledge, this stronger version is new, and exhibits manifestly the power of the result by Derksen and Fink presented in Theorem~\ref{thm:flag-sets-rank}. Let us mention that we were privately communicated an alternative proof by Nicholas Proudfoot, via a categorical approach. Proudfoot's proof will appear in \cite{elias-miyata-proudfoot-vecchi}.

\begin{teo}\label{thm:ardila-sanchez-strengthening}
   Consider two valuations $f,g : \mathfrak{M}\to \mathrm{R}$. For each set $A$, the map $f\star_A g:\mathfrak{M}\to \mathrm{R}$ defined by
        \[ (f\star_A g)(\M) := \begin{cases}
        f(\M|_A)\, g(\M/A) & \text{ if $A\subseteq E$,}\\
        0 & \text{ otherwise}.\end{cases}\]
    is a valuation. 
\end{teo}   

\begin{proof}
    As said above, all valuations $\mathfrak{M}_E\to \mathbb{Z}$ are integer combinations of valuations of the form $\mathrm{S}_{X,r}$ where $X$ is a maximal flag. In other words, we have:
        \begin{align*}
            f = \sum_{X,r} a_{X,r}\, \mathrm{S}_{X,r},\quad\text{ and }\quad
            g = \sum_{X,r} b_{X,r}\, \mathrm{S}_{X,r}\enspace ,
        \end{align*}
    where the sums run over all maximal flags $X$ and all possible sequences $r$, and the expressions $a_{X,r}$ and $b_{X,r}$ denote integer coefficients. To prove our statement, it suffices to reduce to the case:
         \begin{align*}
            f = a_{X,r}\, \mathrm{S}_{X,r},\quad \text{ and }\quad
            g = b_{X',r'}\, \mathrm{S}_{X',r'}\enspace .
        \end{align*}
    Let us denote $|A|$ by $a$. The map $\M\mapsto f(\M|_A)g(\M/A)$ can be described by considering the concatenations of a flag of length $\rk(\M|_A)$ of subsets $X_0\subseteq\cdots\subseteq X_{\rk(\M|_A)}=A$ of ranks $r_0\leq \cdots \leq r_a=\rk(\M|_A)$ followed by $A\sqcup X'_0\subseteq\cdots \subseteq A\sqcup X'_{\rk(\M/A)} = E$ having ranks $r_a+r'_0\leq \cdots \leq r_a+r'_{\rk(\M/A)}=r$.
    In particular, Theorem~\ref{thm:flag-sets-rank} tells that $\M\mapsto f(\M|_A)\cdot g(\M/A)$ is a valuation too. 
\end{proof}

Of course, as was said above, this proves Theorem~\ref{thm:convolutions1}. Each individual summand on the right-hand-side of the convolution is a valuation itself.

\section{Applications to basic valuative invariants}\label{sec:seven}
\noindent  The purpose of this section is to explore the consequences of Theorem~\ref{thm:main} for some basic valuative invariants. In some cases, in order to be able to apply it, we first need to show their valuativeness. This will be done by leveraging the technical tools we have presented in Section~\ref{sec:polytopes}. 

The general framework developed in Section~\ref{sec:stressed-subsets-relaxations}, together with its geometric counterpart, allow us to recover and unify results appearing in various sources in a straightforward way. This includes the case of volumes and Ehrhart polynomials in work of Lam and Postnikov \cite{lam-postnikov}, Ashraf \cite{ashraf-volume}, Hanely, Martin, McGinnis, Miyata, Vindas-Mel\'endez, and Yin \cite{hanely}, and Ferroni \cite{ferroni3}. 

We also address an invariant that will be useful later in this paper: the Tutte polynomial. We prove useful formulas for the Tutte polynomials of elementary split matroids that we believe may be of independent interest. Also, by making use of Theorem~\ref{thm:main} and the bounds in Section~\ref{sec:large-classes}, we settle negatively a log-concavity question raised by Matt Larson.

\subsection{The Volume}\label{sec:volume}

Our main goal is to present connections to known results in the literature \cite{ashraf-volume,ferroni2,hanely} which we then generalize. Another motivation is to illustrate to the reader the methods and usage of Theorem~\ref{thm:main}. We leverage this easy setting to present the technique in a transparent way.

Let us consider a matroid $\M$ on $E$. In the euclidean space $\mathbb{R}^E$ we have a lattice $\mathbb{Z}^E$ consisting of the integer combinations of the vectors $\{e_i\}_{i\in E}$. The intersection of this lattice with the hyperplane $\sum_{i \in E} x_i = \rk(\M)$ is a sublattice. Using this sublattice, we can assign to each $\M\in \mathfrak{M}_E$ of fixed rank a notion of volume.

\begin{defi}
    The map $\vol : \mfrak{M}_E \to \mathbb{R}$ is defined by
        \[ \vol(\M) = \vol(\mathscr{P}(\M)),\]
    where the right-hand-side is the volume of $\mathscr{P}(\M)$ with respect to the restriction of the ambient lattice to the hyperplane $\sum_{i\in E} x_i = \rk(\M)$.
\end{defi}

We make the explicit remark that if instead of restricting the lattice $\mathbb{Z}^E$ to the hyperplane $\sum_{i\in E} x_i = \rk(\M)$, we further restrict it to the affine span of $\mathscr{P}(\M)$ we have the notion of ``relative volume'' --- using that notion, volumes of disconnected matroids no longer vanish.

Throughout this section and the rest of the paper we will denote by $\mathfrak{S}_n$ the symmetric group on $n$ elements. If we write a permutation $\sigma$ using one-line notation, namely $\sigma=\sigma_1\sigma_2\cdots\sigma_n$, then a \emph{descent} of $\sigma$ is an index $i\in \{1,\ldots,n-1\}$ such that $\sigma_i > \sigma_{i+1}$. The number of descents of $\sigma$ is denoted by $\operatorname{des}(\sigma)$. The \emph{Eulerian number} $A_{n,k}$ counts all the permutations in $\mathfrak{S}_n$ with $k$ descents. That is,
    \begin{equation}\label{eq:def-eulerian-number} 
    A_{n,k} :=  \left|\left\{\sigma\in\mathfrak{S}_{n}: \operatorname{des}(\sigma) = k\right\}\right|.
    \end{equation}

\begin{example}\label{example-volume-uniform}
    A classical result attributed in \cite{stanleyeulerian} to Laplace implies that the normalized $(n-1)$-dimensional volume of the hypersimplex $\Delta_{k,n}$ is the Eulerian number $A_{n-1,k-1}$. In other words, we have 
        \[ \vol(\U_{k,n}) = \frac{1}{(n-1)!}\, A_{n-1,k-1}.\]
\end{example}

Volumes of matroids have been studied from different points of view. For example in \cite[Theorem~3.3]{ardila-benedetti-doker} Ardila, Benedetti, and Doker derived an explicit formula for the volume of any matroid polytope in terms of the (signed) $\beta$-invariant of the matroid and its minors. One drawback is that the resulting expression is difficult to handle even for small matroids. Their formula stems from the work of Postnikov on generalized permutohedra \cite[Theorem 9.3]{postnikov}, in which he calculated the volumes of generalized permutohedra.

Following a different path and building on techniques previously developed by Hampe \cite{hampe} within the study of an intersection ring for matroids, Ashraf \cite[Theorem~1.1]{ashraf-volume} proposed an alternative method of computing volumes of matroids. We aim to generalize this method.

In order to calculate the volume of an elementary split matroid by applying Theorem~\ref{thm:main}, we need first the volumes of cuspidal matroids. Fortunately, we can express these volumes in a compact way using permutations. If $\sigma\in\mathfrak{S}_n$ is the permutation $\sigma_1\cdots\sigma_n$ (in one-line notation), we denote by $\operatorname{des}(\sigma_1\cdots\sigma_{n-h})$ the number of indices $1\leq i \leq n-h-1$ such that $\sigma_i > \sigma_{i+1}$.

\begin{prop}
    The volume of the matroid $\LL_{r,k,h,n}$ is given by:
        \[ \vol(\LL_{r,k,h,n}) = \frac{1}{(n-1)!}\; \left| \left\{\sigma\in \mathfrak{S}_{n-1} : \operatorname{des}(\sigma) = k-1 \text{ and } \operatorname{des}(\sigma_1\cdots\sigma_{n-h}) < k -r\right\}\right |.\]
\end{prop}

\begin{proof}
    Notice that by Corollary~\ref{coro:cuspidal-polytope}, we have that the base polytope of $\LL_{r,k,h,n}$ is an ``alcoved polytope'' as in \cite{lam-postnikov}, i.e., it is defined by inequalities of the form $\alpha_{ij} \leq x_i + \cdots + x_j \leq \beta_{ij}$ for integral numbers $\alpha_{ij}$ and $\beta_{ij}$. In particular, the claimed formula follows directly from \cite[Proposition 6.1]{lam-postnikov}.
\end{proof}

\begin{example}\label{example:volume-minimal}
    The minimal matroid $\mathsf{T}_{k,n}$  with $k>0$ is isomorphic to $\LL_{k-1,k,k,n}$. Notice that using the preceding formula, it follows
        \[ \vol(\mathsf{T}_{k,n}) = \frac{1}{(n-1)!} \,\left|\left\{\sigma\in \mathfrak{S}_{n-1} : \operatorname{des}(\sigma) = k -1\text{ and } \operatorname{des}(\sigma_1\cdots\sigma_{n-k}) = 0\right\}\right|.\] 
    In other words we are counting permutations $\sigma=\sigma_1\cdots \sigma_{n-1}$ such that $\sigma_1<\cdots<\sigma_{n-k}$ and  $\sigma_{n-k}>\cdots>\sigma_{n-1}$. There are $\binom{n-2}{k-1}$ such permutations. All of them satisfy $\sigma_{n-k} = n-1$, and we can of course choose $\binom{n-2}{n-k-1} = \binom{n-2}{k-1}$ elements to be on the left and put the remaining elements on the right (their order is completely determined). Hence
        \[ \vol(\mathsf{T}_{k,n}) = \frac{1}{(n-1)!}\binom{n-2}{k-1}.\]
    This formula retrieves \cite[Corollary 4.2]{ferroni2}. 
\end{example}

Based on this valuative invariant we illustrate a straightforward application of Theorem~\ref{thm:main}.

\begin{prop}\label{prop:volume-for-split}
    Let $\M$ be an elementary split matroid of rank $k$ and cardinality $n$. Then
    \[ \vol(\M) = \frac{1}{(n-1)!} \left(A_{n-1,k-1} - \sum_{r,h} \uplambda_{r,h} \vol(\LL_{r,k,h,n}) \right),  \]
    where $\uplambda_{r,h}$ denotes the number of stressed cyclic flats of $\M$ of size $h$ and rank $r$.
\end{prop}

\begin{proof}
    This follows from Theorem~\ref{thm:main} in combination with Example \ref{example-volume-uniform} where we saw that the volume of the uniform matroid $\mathsf{U}_{k,n}$ is exactly $A_{n-1,k-1}$. Recall also that the volume of a disconnected matroid vanishes.
\end{proof}

When we restrict the above formula to paving matroids we recover \cite[Theorem 7.4]{hanely} and if we restrict further to sparse paving matroids we obtain the following formula by Ashraf.

\begin{coro}[{\cite[Theorem~1.2]{ashraf-volume}}]
    Let $\M$ be a sparse paving matroid of rank $k$, cardinality $n$, and having exactly $\uplambda$ circuit hyperplanes. Then
    \[ \vol(\M) = \frac{1}{(n-1)!} \left( A_{n-1,k-1} - \uplambda \,\binom{n-2}{k-1}\right).\]
\end{coro}

In \cite[Proposition 5.3]{lam-postnikov} Lam and Postnikov derived an explicit formula for the volume of arbitrary rank $2$ matroids without loops. Since the class of elementary split matroids contains all loopless rank $2$ matroids, we can use Proposition~\ref{prop:volume-for-split} to retrieve an equivalent version of Lam and Postnikov's expression.

\begin{prop}
    Let $\M$ be matroid of rank $2$ without loops. Then
    \[ \vol(\M) = \frac{1}{(n-1)!}\left(2^{n-1} - \sum_{h} \uplambda_{h} \sum_{i=n-h}^{n-1}\binom{n-1}{i}\right),
    \]
    where $\uplambda_h = \uplambda_{1,h}$ is the number of stressed hyperplanes of size $h$.
\end{prop}

\begin{proof}
    Observe that, in analogy to what we did in Example \ref{example:volume-minimal}, we have
    \[ \vol(\LL_{1,2,h,n}) = \frac{1}{(n-1)!}\,\left|\{\sigma\in \mathfrak{S}_{n-1}: \operatorname{des}(\sigma) = 1,\; \operatorname{des}(\sigma_1\cdots \sigma_{n-h}) = 0\}\right|.\]
    We are counting permutations that have only one descent. Assume that the only descent is at position~$i$, i.e., $\sigma_i > \sigma_{i+1}$. Notice that there are $\binom{n-1}{i}$ choose the first $i$ elements $\sigma_1<\cdots<\sigma_{i}$. The remaining $n-1-i$ elements must be sorted increasingly. The only case in which $\sigma_i < \sigma_{i+1}$ is precisely when $\sigma$ is the identity permutation. Thus, there are $\binom{n-1}{i} - 1$ permutations having descent set $\{i\}$. Thus we obtain
    \[ \vol(\LL_{1,2,h,n}) =\frac{1}{(n-1)!} \sum_{i=n-h}^{n-2}\left(\binom{n-1}{i} - 1\right).\]
    Hence it follows that
    \[ \vol(\LL_{1,2,h,n}) = \frac{1}{(n-1)!}\left(\sum_{i=n-h}^{n-2} \binom{n-1}{i} - (h-1)\right) = \frac{1}{(n-1)!}\left(\sum_{i=n-h}^{n-1} \binom{n-1}{i} - h\right).\]
    Notice that for a loopless matroid of rank $2$, the stressed subsets with non-empty cover are all of rank $1$. Using Proposition~\ref{prop:volume-for-split} we derive
    \[ \vol(\M)= \frac{1}{(n-1)!}\left(A_{n-1,1} - \sum_{h} \uplambda_{h} \left( \sum_{i=n-h}^{n-1}\binom{n-1}{i} - h\right)\right).\]
    Additionally, $A_{n-1,1} = 2^{n-1}-n$. Furthermore, all the hyperplanes in a rank $2$ loopless matroid are disjoint (because the only flat of smaller rank is the empty set) and form a partition of the ground set (because every element lies in a hyperplane), we have that $\sum_{h} \uplambda_{h} h = n$. Hence, the above formula reduces to the expression in the statement.
\end{proof}

\subsection{Ehrhart polynomials}\label{subsec:ehrhart}
\noindent
In this subsection we draw our attention to Ehrhart polynomials. We begin by recalling their definition.

\begin{teo}[\cite{Ehrhart}]
    Let $\mathscr{P}\subseteq\mathbb{R}^n$ be a lattice polytope of dimension $d$. The map
        \[ t \mapsto \left|t\mathscr{P}\cap \mathbb{Z}^n\right|,\]
    defined for each non-negative integer $t$, is the evaluation of a polynomial of degree $d$. 
\end{teo}

The rational polynomial defined by the above property is the \emph{Ehrhart polynomial} of the polytope~$\mathscr{P}$; we will denote it by $\ehr(\mathscr{P},t)\in\mathbb{Q}[t]$. We refer to \cite{beck-robins} for a detailed treatment of Ehrhart polynomials.

The Ehrhart polynomial counts lattice points in dilations of polytopes, whence an inclusion-exclusion argument reveals that it is a valuative invariant for matroids. We will write $\ehr(\M,t)$ as a shorthand of $\ehr(\mathscr{P}(\M),t)$. 
This invariant of matroids has been the focus of many articles due to an intriguing conjecture of De Loera, Haws and K\"oppe \cite[Conjecture~2]{deloera}. 

\begin{conj}[\cite{deloera}]
    For every matroid $\M$, the Ehrhart polynomial of $\mathscr{P}(\M)$ has positive coefficients.
\end{conj}

Despite several supporting results  \cite{postnikov, castillo-liu, castillo-liu2, jochemko-ravichandran}, including positive results for uniform matroids \cite{ferroni1}, minimal matroids \cite{ferroni2} and matroids of rank $2$ 
\cite{ferroni-jochemko-schroter}, it has been shown in \cite{ferroni3} that this conjecture is false in general.
We demonstrate in this section how one can recover Ferroni's counterexamples as well as further results on Ehrhart polynomials of (sparse) paving matroids by applying Theorem~\ref{thm:main} or Corollary~\ref{thm:main-for-sparse-paving}.  

In \cite{katzman} Katzman studied the Hilbert functions of certain algebras of Veronese type and derived an explicit formula for the Ehrhart polynomial of the hypersimplex $\Delta_{k,n}$.

\begin{teo}[\cite{katzman}]\label{katzman}
    The Ehrhart polynomial of the uniform matroid~$\U_{k,n}$ is given by
		\begin{equation} 
			\ehr(\U_{k,n},t) = \sum_{j=0}^{k-1} (-1)^j \binom{n}{j} \binom{(k-j)t+n-1-j}{n-1}\enspace .
		\end{equation}
\end{teo}

Although this formula is compact, proving that it yields a polynomial with positive coefficients is a challenging problem. 
Notice that it is an alternating sum, and that the variable $t$ appears inside a binomial coefficient that sometimes has negative terms when expanded. The main result in \cite{ferroni1} is a combinatorial formula for the coefficients of the above polynomial that shows manifestly that they are positive. In \cite[Theorem~3.8]{knauer-martinez-ramirez} Knauer, Mart\'inez-Sandoval, and Ram\'irez-Alfons\'in found a formula for the Ehrhart polynomial of a minimal matroid, in terms of Bernoulli numbers and symmetric functions evaluated at certain integers. Later, Ferroni \cite{ferroni2} provided the following formulas.

\begin{prop}[\cite{ferroni2}]\label{prop:ehrhart-minimal}
    The Ehrhart polynomial of the base polytope of the minimal matroid $\mathsf{T}_{k,n}:=\LL_{k-1,k,k,n}$ is given by
		\begin{align*}
		    \ehr(\mathsf{T}_{k,n},t) &=\sum_{j=0}^t \binom{n-k-1+j}{j}\binom{k-1+j}{j}
		    \\
		    &= \frac{1}{\binom{n-1}{k-1}} \binom{t+n-k}{n-k} \sum_{j=0}^{k-1}\binom{n-k-1+j}{j}\binom{t+j}{j}\enspace .
		\end{align*}
\end{prop}

The second of these two formulas reveals immediately that the coefficients of $\ehr(\mathsf{T}_{k,n},t)$ are always positive. Moreover, the shifted evaluation $\ehr(\mathsf{T}_{k,n},t-1)$ has non-negative coefficients. 

Now, having these compact formulas for $\ehr(\U_{k,n},t)$ and $\ehr(\mathsf{T}_{k,n},t)$ we can combine them with Corollary~\ref{thm:main-for-sparse-paving} and obtain the following proposition.

\begin{prop}[\cite{ferroni3}]\label{prop:ehrhart-sparsepaving}
    Let $\M$ be a sparse paving matroid of rank $k$ and cardinality $n$ having exactly $\uplambda$ circuit-hyperplanes. Then
    \[ \ehr(\M, t) = \ehr(\U_{k,n}, t) - \uplambda \ehr(\mathsf{T}_{k,n}, t-1).\]
\end{prop}

\begin{proof}
    By Corollary~\ref{thm:main-for-sparse-paving} we have that
    \[ \ehr(\M, t) = \ehr(\U_{k,n}, t) - \uplambda \left(\ehr(\mathsf{T}_{k,n}, t) - \ehr(\U_{k-1,k}\oplus \U_{1,n-k}, t)\right).\]
    As we mentioned before, the Ehrhart polynomial is multiplicative and invariant under taking duals. 
    Hence, combining that with Proposition~\ref{katzman}, we obtain
    \begin{align*}
        \ehr(\U_{k-1,k}\oplus \U_{1,n-k}, t) &= \ehr(\U_{k-1,k},t) \cdot \ehr(\U_{1,n-k},t)\\ 
        &= \ehr(\U_{1,k},t)\cdot \ehr(\U_{1,n-k},t)\\
        &= \binom{t+k-1}{k-1}\binom{t+n-k-1}{n-k-1}.
    \end{align*}
    Using the first of the two formulas for $\ehr(\mathsf{T}_{k,n},t)$ in Proposition~\ref{prop:ehrhart-minimal}, it is straightforward to conclude.
\end{proof}

Since $\ehr(\mathsf{T}_{k,n},t-1)$ has non-negative coefficients, the heuristic is that having more circuit-hyperplanes causes the coefficients of $\ehr(\M,t)$ to be smaller.  We can now benefit from Lemma~\ref{lem:bound-ch-sparsepaving}. 

\begin{coro}
    There exist matroid polytopes with Ehrhart polynomials having negative coefficients.
\end{coro}

\begin{proof}
    Our Corollary~\ref{coro:bound-circuit-hyperplanes-sparse-paving} guarantees that for each $k$ and $n$ we can find a sparse paving matroid of rank $k$, cardinality $n$ having at least $\uplambda=\frac{1}{n}\binom{n}{k}$ circuit-hyperplanes. Substituting the value $\uplambda = \frac{1}{n}\binom{n}{k}$ in the formula of  Proposition~\ref{prop:ehrhart-sparsepaving} shows that there are Ehrhart polynomials of base polytopes of sparse-paving matroids with negative coefficients. For example, the quadratic coefficient of the matroid obtained taking $n = 3589$ and $k=3$ is negative.
\end{proof}

The natural step now is to compute the Ehrhart polynomials of cuspidal matroids. This allows us to describe the Ehrhart polynomials of elementary split matroids by applying Theorem~\ref{thm:main}, and thus generalizes the results of Hanely, Martin, McGinnis, Miyata, Nasr, Vindas-Mel\'endez, and Yin \cite{hanely}. 

In \cite[Theorem~1.1]{fan-li} Fan and Li derived an intricate formula for the Ehrhart polynomial of an arbitrary Schubert matroid. Their formula is difficult to handle, unless the Schubert matroid under consideration is particularly well-structured. Fortunately, it suffices to our purposes, because when restricting to the cuspidal matroids, it reduces to a more tractable expression. 

\begin{lemma}[{\cite[Equation~(4.5)]{fan-li}}]\label{lemma:ehr-for-cuspidal}
    The Ehrhart polynomial of the matroid $\LL_{r,k,h,n}$ satisfies
    \begin{multline*}
     \ehr(\LL_{r,k,h,n},t) = \\
    \sum_{i=0}^{mt} \sum_{j=0}^{n-h}\sum_{\ell=0}^h (-1)^{j+\ell} \binom{n-h}{j}\binom{h}{\ell}\binom{(a-j)t + n-h-j+i-1}{n-h-1}\binom{(b-\ell)t+h-\ell-i-1}{h-1},
    \end{multline*}
    where $a = n-h-k+r$, $b=h-r$ and $m = \min(h,k)-r$. 
    Here we use the convention that $\binom{\alpha}{\beta}=0$ whenever $\alpha<\beta$ or $\alpha < 0$.
\end{lemma}

The variable $t$ appears in the upper limit of the indices in the first sum. This means that for the computation of the coefficients of $\ehr(\LL_{r,k,h,n},t)$ one has to use Lagrange's interpolation or another method. Moreover, Fan and Li conjecture that cuspidal matroids have Ehrhart polynomials with positive coefficients \cite[Conjecture~1.6]{fan-li}. This is related to another conjecture by Ferroni, Jochemko, and Schr\"oter \cite[Conjecture~6.1]{ferroni-jochemko-schroter} on positroids.

In \cite[Corollary~5.3]{hanely}, a simplification of the above formula is obtained in the case in which $r=k-1$, thus getting rid of the variable $t$ in the upper limit. 

\subsection{Tutte polynomials}\label{sec:tutte}
\noindent
Perhaps the most popular among all matroid invariants is the Tutte polynomial. This fact is not a coincidence, since the Tutte polynomial encodes several important statistics associated to a matroid, e.g., the number of bases, the number of independent sets, the number of spanning sets, etc. Furthermore, the Tutte polynomial has the property of being universal among all Tutte--Grothendieck invariants. See \cite{brylawski-oxley} and \cite{handbook-Tutte} for a thorough account of remarkable properties that Tutte polynomials have; in particular, we point out the chapter \cite[Chapter 32]{handbook-Tutte} by Falk and Kung, in which they describe how one obtains the Tutte polynomial as linear specialization of the $\mathcal{G}$-invariant with a focus on paving matroids.

\begin{defi}
    Let $\M$ be a matroid on $E$. The \emph{Tutte polynomial} of $\M$ is defined as the bivariate polynomial $T_{\M}(x,y)\in \mathbb{Z}[x,y]$ given by
    \[ T_{\M}(x,y) = \sum_{A\subseteq E} (x-1)^{\rk(E) - \rk(A)} (y-1)^{|A|-\rk(A)}.\]
\end{defi}

\begin{example}\label{ex:tutte_uniform}
    The Tutte polynomial of the uniform matroid $\U_{k,n}$ is given by
    \[ T_{\U_{k,n}}(x,y) = \sum_{i=1}^{k} \binom{n-i-1}{n-k-1} \, x^i + \sum_{i=1}^{n-k} \binom{n-i-1}{k-1}\,  y^i,\]
    whenever $1\leq k \leq n - 1$. If $k=0$ then $T_{\U_{0,n}}(x,y) = y^n$; if $k=n$, then $T_{\U_{n,n}}(x,y) = x^n$; see for instance \cite{merino-tutte}.
\end{example}

The Tutte polynomial behaves interchanges the variables when taking duals, i.e., $T_{\M^*}(x,y) = T_{\M}(y,x)$, and it is multiplicative, i.e., $T_{\M\oplus \N}(x,y) = T_{\M}(x,y)\cdot T_{\N}(x,y)$. Another important property that follows from the definition of the Tutte and the characteristic polynomial is that 
    \begin{equation}\label{eq:char-from-tutte}
    \chi_{\M}(t) =(-1)^{\rk(\M)} T_{\M}(1-t,0),
    \end{equation}
where $k$ is the rank of $\M$. 

There exist several proofs in the literature of the fact that the map $\mathfrak{M}\to \mathbb{Z}[x,y]$ given by $\M \mapsto T_{\M}(x,y)$ is a valuative invariant. For instance, see \cite[Lemma~3.4]{speyer-conjecture} or the more general version \cite[Theorem~5.4]{ardila-fink-rincon}. 
It is possible to recover the Tutte polynomial as a linear specialization of the $\mathcal{G}$-invariant, which of course shows that it is a valuation. However, using the convolution property of Ardila and Sanchez, it is possible to obtain yet another and rather compact proof of the valuativeness. We include this for the sake of completeness.

\begin{teo}
    The map $T : \mathfrak{M} \to \mathbb{Z}[x,y]$, which associates to each matroid its Tutte polynomial, is a valuative invariant.
\end{teo}

\begin{proof}
    Consider the following two maps $f,g:\mathfrak{M} \to \mathbb{Z}[x,y]$ defined by
    \begin{align*}
        f(\M) = y^{\rk(\M^*)} \quad \text{ and } \quad
        g(\M) = x^{\rk(\M)}.
    \end{align*}
    As we mentioned in Example \ref{ex:rank-function} both $f$ and $g$ are valuative invariants. By Theorem~\ref{thm:convolutions1}, we have that their convolution,
        \[ (f\star g)(\M) = \sum_{A\subseteq E} y^{\rk((\M|_A)^*)} x^{\rk(\M/A)}\]
    is a valuative invariant. Notice that $\rk((\M|_A)^*) = |A|-\rk_{\M}(A)$, whereas $\rk(\M/A) = \rk(\M)-\rk_{\M}(A)$. Hence, we obtain
        \[ (f\star g)(\M) = T_{\M}(x+1,y+1),\]
    After a change of variables, we conclude that the Tutte polynomial is a valuative invariant.
\end{proof}

Our next goal is to provide a good way of computing Tutte polynomials for arbitrary elementary split matroids. As we did with the volume, a possible approach consists of first determining it for all cuspidal matroids. In fact, since these are particular lattice path matroids, it is possible to write down explicit formulas following the interpretation of the coefficients found in \cite[Theorem~5.4]{bonin-demier}. We follow another approach that leads to dealing with fewer subcases. We first investigate how the Tutte polynomial changes under a relaxation. As usual, our matroid $\M$ has rank $k$ and cardinality $n$, and the stressed subset $F$ with non-empty cover that we relax has rank $r$ and cardinality $h$.

\begin{prop}\label{prop:tutte-relax}
    Let $\M$ be a matroid, and let $F$ be a stressed subset with non-empty cover. Then the Tutte polynomial of the relaxation $\widetilde{\M}=\Rel(\M,F)$ satisfies
    \[ T_{\widetilde{\M}}(x,y)\enspace =\enspace T_{\M}(x,y)\, + \sum_{i=r+1}^h\sum_{j=0}^{k-r-1} \binom{h}{i}\binom{n-h}{j}\, \alpha_{r,k}^{i,j}(x,y)\enspace,\]
    where 
    \[\alpha_{r,k}^{i,j}(x,y) := 
    \begin{cases} (x-1)^{k-i-j}\left(1-((x-1)(y-1))^{i-r}\right) & \text{if $i+j\leq k$}\\
    (y-1)^{i+j-k}\left(1-((x-1)(y-1))^{k-r-j}\right) & \text{if $i+j > k$}.
    \end{cases}\]
    In particular, $T_{\widetilde{\M}}(x,y) - T_{\M}(x,y)$ is a multiple of $x+y-xy$.
\end{prop}

\begin{proof}
    By Proposition~\ref{prop:rank-relaxation}, we know that the rank functions of $\M$ and $\widetilde{\M}$ coincide on the sets $A$ with $|A\cap F|\leq \rk(F)$.
    Furthermore, Remark~\ref{rem:rank} tells us that, for the remaining subsets, $\widetilde{\rk}(A) = \min(|A|, k)$ and $\rk(A) = \min(k, r+|A\smallsetminus F|)$. Here we denote the rank function of $\widetilde{\M}$ by $\widetilde{\rk}:2^E\to \mathbb{Z}_{\geq 0}$.
    We may calculate the difference of the Tutte polynomials of the matroids by considering only sets where the rank functions potentially differ.
    \[
        T_{\widetilde{\M}}(x,y)- T_{\M}(x,y) = \sum_{\substack{A\subseteq E\\|A\cap F| \geq r + 1}} \left( (x-1)^{k - \widetilde{\rk}(A)}(y-1)^{|A|-\widetilde{\rk}(A)}-(x-1)^{k - \rk(A)}(y-1)^{|A|-\rk(A)}\right).
    \]
    Observe that all the subsets over which we are iterating are of the form $A = X\sqcup Y$ where $X \subseteq F$, $Y\subseteq E\smallsetminus F$ and $|X|\geq r+1$. Observe that both $\rk(A)$ and $\widetilde{\rk}(A)$ will depend only on the sizes of $X=A\cap F$ and $Y=A\smallsetminus F$. In particular, if we let $i = |X|$ and $j = |Y|$, we have 
    \begin{align*}
        (x-1)^{k - \widetilde{\rk}(A)}(y-1)^{|A|-\widetilde{\rk}(A)} &= (x-1)^{k - \min(i+j,k)}(y-1)^{i+j-\min(i+j,k)}
        \end{align*} and
        \begin{align*}
        (x-1)^{k - \rk(A)}(y-1)^{|A|-\rk(A)} &= (x-1)^{k-\min(k,r+j)}(y-1)^{i+j-\min(k,r+j)}.
    \end{align*}
    Further if $|Y|\geq k-r$, then $\rk(A) = k$ and obviously $\widetilde{\rk}(A) = k$, so that we are allowed to ignore the terms in the sum corresponding to $|Y|\geq k-r$. 
    Once we fix the sizes of $X$ and $Y$ as above, we have $\binom{h}{i}$ possibilities to choose $X\subseteq F$ and $\binom{n-h}{j}$ for $Y\subseteq E\smallsetminus F$. Further $i$ and $j$ range between $r+1$ and $h$ and between $0$ and $k-r-1$, respectively. 
    Whenever $j\leq k-r-1$ also $\min(k,r+j) = r+j$, hence we may simplify the second of the two equations. By putting all the pieces together, we obtain the identity of the statement.
\end{proof}

\begin{remark}
    When the stressed subset with non-empty cover $F$ is a hyperplane, the preceding formula reduces to \cite[Proposition~3.13]{ferroni-nasr-vecchi}, which therefore leads to a nice expression for the Tutte polynomial of paving matroids. Such a formula was stated by Brylawski in \cite[Proposition~7.9]{brylawski2}.
\end{remark}

Now, we can state how the Tutte polynomial of an arbitrary elementary split matroid looks like.

\begin{teo}\label{thm:tutte-elementary-split}
    Let $\M$ be an elementary split matroid of rank $k$ and cardinality $n$. The Tutte polynomial of $\M$ is given by
    \[ T_{\M}(x,y) = T_{\U_{k,n}}(x,y) - \sum_{r,h} \uplambda_{r,h}\sum_{i=r+1}^h\sum_{j=0}^{k-r-1} \binom{h}{i}\binom{n-h}{j}\, \alpha_{r,k}^{i,j}(x,y)\enspace,\]
    where $\alpha_{r,k}^{i,j}$ is defined as in Proposition~\ref{prop:tutte-relax}.
\end{teo}

As was observed by Brylawski in \cite{brylawski2}, a matroid $\M$ is paving if and only if all the monomials $x^iy^j$ for $i\geq 2$ and $j\geq 1$ of its Tutte polynomial have a vanishing coefficient. In other words, Tutte polynomials can be used in a straightforward way to decide whether a given matroid is paving. However, for elementary split matroids we have less structure. In general, for an elementary split matroid all the coefficients can be strictly positive simultaneously; for an example just consider the minimal matroid $\mathsf{T}_{k,n}$. Also, unlike for paving matroids, the Tutte polynomial cannot detect whether a matroid is or not elementary split, as the following example shows.

\begin{example}\label{ex:tutte-split}
    The two matroids $\M_1$ and $\M_2$ depicted in Figure~\ref{fig:counterexample-tutte} have rank $3$ and size $7$. The one on the left is elementary split, while the one on the right is not. Their Tutte polynomials agree, and are equal to:
        \[ T_{\M_1}(x,y) = T_{\M_2}(x,y) = y^{4} + x^{3} + x^{2} y + x y^{2} + 3 y^{3} + 3 x^{2} + 5 x y + 5 y^{2} + 3 x + 3 y.\]
    \begin{figure}[ht]
    \centering
	\begin{tikzpicture}  
	[scale=0.5,auto=center,every node/.style={circle,scale=0.8, fill=black, inner sep=2.7pt}] 
	\tikzstyle{edges} = [thick];
	\node[label=left:$1$] (a1) at (0,0) {};  
	\node[label=above:$2$] (a2) at (2/2,1/2)  {};
	\node[label=above:$3$] (a3) at (4/2,2/2)  {};  
	\node[label=above:$4$] (a4) at (8/2,4/2) {};
	\node[label=below:$5$] (a5) at (4/2,-2/2)  {};  
	\node[label=below:{$6$, $7$}] (a6) at (7.75/2,-4/2)  {}; 
	\node[] (a7) at (8.25/2,-4/2)  {}; 
	\draw[edges] (a1) -- (a2); 
	\draw[edges] (a2) -- (a3);  
	\draw[edges] (a3) -- (a4);  
	\end{tikzpicture} \qquad\qquad\qquad	\begin{tikzpicture}  
	[scale=0.5,auto=center,every node/.style={circle,scale=0.8, fill=black, inner sep=2.7pt}] 
	\tikzstyle{edges} = [thick];
	
	\node[label=left:$1$] (a1) at (0,0) {};  
	\node[label=right:$2$] (a2) at (4/2,0/2)  {};
	\node[label=above:$3$] (a3) at (4/2,2/2)  {};  
	\node[label=above:$4$] (a4) at (8/2,4/2) {};
	\node[label=below:$5$] (a5) at (4/2,-2/2)  {};  
	\node[label=below:{$6$, $7$}] (a6) at (7.75/2,-4/2)  {}; 
	\node[] (a7) at (8.25/2,-4/2)  {}; 
	
	\draw[edges] (a1) -- (a3); 
	\draw[edges] (a3) -- (a4);  
	\draw[edges] (a1) -- (a5);  
	\draw[edges] (a5) -- (a6);
        \draw[edges] (a3) -- (a5);
        
	\end{tikzpicture} 
 \caption{A split and a non-split matroid having the same Tutte polynomial.} \label{fig:counterexample-tutte}
\end{figure}
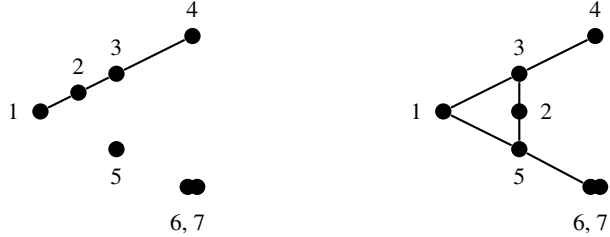    
\end{example}

Let us close this subsection by mentioning an interesting application of our machinery, and Theorem~\ref{thm:tutte-elementary-split}. The following settles negatively a conjecture on the log-concavity of certain evaluation of the Tutte polynomial, posed by Matt Larson during the BIRS workshop \emph{``Algebraic Aspects of Matroid Theory''} held in Banff, Canada in March 2023. 

\begin{coro}
    The polynomial $T_{\M}(x+1,x+1)$ does not have log-concave coefficients in general. Moreover, it is not unimodal in general.
\end{coro}

\begin{proof}
    For a sparse paving matroid the formula in Theorem~\ref{thm:tutte-elementary-split} simplifies drastically as $r=k-1$, $h=k$ implies that $i=k$ and $j=0$ in the above formula. Clearly,
        \[
            \alpha_{k-1,k}^{k,0}(x,y) = 1-(x-1)(y-1) = x+y-xy
        \]
        and thus
    \[
        T_{\M}(x,y) = T_{\U_{k,n}}(x,y) - \uplambda (x+y-xy)
    \]
    for a sparse paving matroid $\M$ of rank $k$ on $n$ elements with $\uplambda$ circuit-hyperplanes. In particular, Lemma~\ref{lem:bound-ch-sparsepaving} says that there is a matroid on $n=176$ elements, having rank $k=81$, and $\uplambda=\left\lfloor\tfrac{1}{176}\binom{176}{81}\right\rfloor$ circuit-hyperplanes. Its Tutte polynomial has the form:
    {\footnotesize
\begin{align*}
        T_{\M}(x+1,x+1)&= 
        3285987011079620145461965951500485419858828427672500\; +\\
        &\qquad 6617084784610995354941772132439688305063454012156950 \; x\; + \\
        &\qquad 6654590882901946448474136220360853789828324829830700 \; x^2\; + \\
        &\qquad 
        6654535137231641703627509311976459391294765151016000 \; x^3 \;+ \\
        &\qquad 6656646120873892602379298143301761305388688100978400 \; x^4 \;+ \;\cdots 
        \end{align*}
}
A direct inspection of the above few terms reveals that the sequence of coefficients is not unimodal, and therefore also not log-concave.
\end{proof}

\section{Invariants arising from flats and Hodge theory}\label{sec:eight}
\noindent In this section we address invariants that are related to the enumeration of flats. We show that the flag $f$-vector of the lattice of flats is valuative, which yields automatically the valuativeness of several useful specializations such as the Whitney polynomial and the chain polynomial. We also discuss the Hilbert series of Chow rings: we formulate a conjecture about their real-rootedness, and discuss several facts related to the problem of computing them.

\subsection{The flag \texorpdfstring{$f$}{f}-vector}

In \cite{bonin-kung} Bonin and Kung studied the $\mathcal{G}$-invariant of Definition~\ref{def:ginvar}. They proved that it is equivalent to knowing the \emph{catenary data} of a matroid. The catenary data of a matroid is a vector that enumerates all the maximal chains of flats of $\M$ according to the \emph{cardinality} of each of the flats appearing in it.

The goal now is to prove a related result, which will be of central use later in the paper to derive the valuativeness of other invariants, such as the chain polynomial of $\mathcal{L}(\M)$ or the Kazhdan--Lusztig polynomial of the matroid $\M$.

\begin{defi}
    The \emph{flag $f$-vector} of the lattice of flats $\mathcal{L}(\M)$ of a matroid $\M$ is the function $\alpha_{\M}:2^E\to \mathbb{Z}$ defined by
        \[ \alpha_{\M}(\{r_1,\ldots,r_s\}) = \#\{\text{flags }F_1\subseteq \cdots \subseteq F_s \text{ in }\mathcal{L}(\M): \rk(F_i) = r_i \text{ for each $i$}\}.\]
\end{defi}

Put more succinctly, the flag $f$-vector enumerates chains of flats according to the total length of the chain, and the rank of each of the flats in the chain. The central result that we will prove here is a straightforward yet important application of Theorem~\ref{thm:flags-flats-rank-valuation}. The main difference with Bonin and Kung's catenary data is that we do not make any reference to the \emph{sizes} of the flats in the chain, and our flags need not be maximal.

\begin{teo}\label{thm:flag-f-vector-valuative}
    The map associating to each matroid $\M$ the flag $f$-vector of $\mathcal{L}(\M)$ is a valuative invariant.
\end{teo}

\begin{proof}
    The flag $f$-vector is the sum over all families of subsets $F_0,\ldots,F_m$ and all choices of numbers $r_0,\ldots,r_m$ of the functions defined in equation \eqref{eq:fixed-flag-flats-valuation-arb-length-and-rank}. Since each summand is itself a valuation, the result follows.
\end{proof}

\begin{remark}
    It is possible to define an even finer invariant than the flag $f$-vector by taking into account both the ranks and sizes of the flats in each chain. With considerable effort, one could attempt to adapt the proof of Theorem~\ref{thm:flags-flats-rank-valuation}. We omit such laborious task here because we will not require a result at that level of generality.
\end{remark}

\subsection{Whitney polynomials}

An important family of statistics of the matroid $\M$ that are not captured by the Tutte polynomial are the Whitney numbers of the second kind. These numbers are denoted $W_0,\ldots, W_k$, where $W_i$ counts the number of rank $i$ flats of $\M$. Following Mason \cite{mason}, the rank generating function of the lattice of flats $\mathcal{L}(\M)$ will be referred to as the \emph{Whitney polynomial} of $\M$, and will be denoted by:
        \[ \mathrm{W}_{\M}(x) = \sum_{i=0}^k W_i\, x^i.\]
An immediate consequence of the valuativeness of the flag $f$-vector is that the Whitney polynomial is itself a valuative invariant.

\begin{coro}\label{coro:whitney-is-valuative}
    The map $\mathrm{W} : \mathfrak{M} \to \mathbb{Z}[t]$ assigning to each matroid $\M$ the polynomial $\mathrm{W}_{\M}(t)$ is a valuative invariant.
\end{coro}

A basic property of Whitney polynomials is that they behave multiplicatively under direct sums of matroids. An important motivation for studying these polynomials stems from a long standing conjecture by Rota \cite{rota-conjecture} which postulates that $\mathrm{W}_{\M}(t)$ has always unimodal coefficients. Some strengthenings by Mason \cite{mason} go beyond and conjecture the log-concavity. The most general family of inequalities for Whitney numbers of the second kind is provided by the main result of Braden, Huh, Matherne, Proudfoot, and Wang in \cite{bradenhuh}. They proved that $W_i\leq W_j$ whenever $0\leq i\leq j \leq k-i$. The proof of these inequalities relies on deep constructions related to the singular Hodge theory of matroids. 

In an attempt to construct potential counterexamples to the unimodality or log-concavity conjectures, it is plausible to approach the computation of Whitney polynomials of elementary split matroids.

Let us state now the general formula for the Whitney numbers when $\M$ is an arbitrary cuspidal matroid. We start with the following auxiliary result.
\begin{prop}
    The number of hyperplanes of the matroid $\LL_{r,k,h,n}$ is 
    \[
    \binom{h}{r-1}+\sum_{j=r}^h \binom{h}{j} \binom{n-h}{k-1-j} \enspace .
    \]
\end{prop}
\begin{proof}
    Let $E_1 = \{1,\ldots,n-h\}$ and $E_2 = [n]\smallsetminus E_1$. A set $S$ is a hyperplane of $\LL_{r,k,h,n}$ if and only if $\rank(S)=k-1$ and $S$ is closed. 
    Thus using Corollary~\ref{cor:rank_cusp} we get that the hyperplanes of $\LL_{r,k,h,n}$ are the sets $S$ of size $k-1$ intersecting $E_2$ in at least $r$ elements, and the sets $S=E_1\sqcup S'$ of size $n-h+r-1\geq k-1$ where $|S'|=r-1$. Counting the number of these sets leads to the desired formula.  
\end{proof}

\begin{prop}\label{prop:W-poly_cusp}
    The $\mathrm{W}$-polynomial of the matroid $\LL_{r,k,h,n}$ is given by
    \[ \mathrm{W}_{\LL_{r,k,h,n}}(t) = \sum_{i=0}^k W_i\, t^i,\]
    where the $i$-th Whitney number of the second kind of $\LL_{r,k,h,n}$ is
    \[ W_i = \binom{h}{r+i-k} + \sum_{j=r+i-k+1}^h \binom{h}{j}\binom{n-h}{i-j},\]
    for $0\leq i \leq k-1$, and $W_{k} = 1$. Here we use the convention that $\binom{a}{b}=0$ whenever $b<0$.
\end{prop}

\begin{proof}
    The number of flats of rank $k-i-1$ in the matroid $\LL_{r,k,h,n}$ is the number of hyperplanes of the $i$-th truncation $\mathrm{T}^i(\LL_{r,k,h,n})$. The truncation of the matroid $\LL_{r,k,h,n}$ is again a cuspidal matroid and, in fact, $\mathrm{T}^i(\LL_{r,k,h,n}) \cong \LL_{\max\{r-i,0\},k-i,h,n}$. Hence, using the preceding result, the formula of the statement follows.
\end{proof}

In particular, since the $\mathrm{W}$-polynomial is multiplicative, it is therefore easy to compute for a direct sum of uniform matroids. By using Theorem~\ref{thm:main}, we obtain an easy description of the Whitney numbers of the second kind of all elementary split matroids.

\begin{teo}\label{thm:whitney-for-split}
    Let $\M$ be an elementary split matroid of rank $k$ and cardinality $n$ and $\uplambda_{r,h}$ the number of stressed subsets with non-empty cover of rank $r$ and size $h$. Then
    \[ \mathrm{W}_{\M}(t) = \mathrm{W}_{\U_{k,n}}(t) - \sum_{r,h} \uplambda_{r,h} \left(\mathrm{W}_{\LL_{r,k,h,n}}(t)-\mathrm{W}_{\U_{k-r,n-h}}(t) \cdot \mathrm{W}_{\U_{r,h}}(t)\right).
    \]
    In other words, the Whitney numbers of the second kind are given by $W_k = 1$ and, for $0\leq i\leq k-1$,
    \[
    W_i = \binom{n}{i} - \sum_{\substack{r,h}} \uplambda_{r,h} \left(\sum_{j=r}^h \binom{h}{j}\binom{n-h}{i-j} - \binom{n-h}{i-r}\right).
    \]
\end{teo}
\begin{proof}
    The formula for $\mathrm{W}_{\M}(t)$ follows directly from Theorem~\ref{thm:main} and the fact that Whitney polynomials are multiplicative.
    To obtain the formula for the coefficients we expand the polynomial $\mathrm{W}_{\M}(t)$. We begin with the product $\mathrm{W}_{\U_{k-r,n-h}}(t) \cdot \mathrm{W}_{\U_{r,h}}(t)$ for which we have 
    \begin{align*}
    \mathrm{W}_{\U_{k-r,n-h}}(t) \cdot \mathrm{W}_{\U_{r,h}}(t) &=  \left( t^{k-r}+ \sum_{\ell=0}^{k-r-1}\binom{n-h}{\ell} t^\ell \right)
    \left(t^r + \sum_{j=0}^{r-1} \binom{h}{j} t^j\right)\\
    &= t^k 
    + \sum_{j=k-r}^{k-1} \binom{h}{r+j-k} t^j + \sum_{\ell=r}^{k-1} \binom{n-h}{\ell-r} t^\ell +
    \sum_{i=0}^{k-2} \sum_{j=r+i-k+1}^{\min\{i, r-1\}} \binom{h}{j}\binom{n-h}{i-j} t^i.
    \end{align*}
    Applying Proposition~\ref{prop:W-poly_cusp}, we see that the $i$-th term in the difference $\mathrm{W}_{\LL_{r,k,h,n}}(t)-\mathrm{W}_{\U_{k-r,n-h}}(t) \cdot \mathrm{W}_{\U_{r,h}}(t)$ is zero if $i=k$ and otherwise it is equal to
    \begin{align}\label{eq:W_diff}
    \binom{n-h}{i-r} - 
    \sum_{j=\min\{i+1, r\}}^h \binom{h}{j}\binom{n-h}{i-j} \enspace .
    \end{align}
    The formula of the statement follows from the fact that $[t^i]\mathrm{W}_{\U_{k,n}}(t) = \binom{n}{i}$ for all $i<k$, and that the expression in \eqref{eq:W_diff} vanishes whenever $r>i$. 
\end{proof}

At first glance this result might seem rather tautological, but notice that we are expressing the number of flats of each rank only in terms of the number of stressed subsets of each rank and size. 

We used Theorem~\ref{thm:whitney-for-split} to produce an extensive list of $\mathrm{W}$-polynomials of matroids. As said before, our motivation was that this idea appeared to be a plausible approach to try to construct potential counterexamples to Rota's conjecture. Despite the fact that we were not able to find any such counterexamples, we believe the above formulas are worth to be included here.

\subsection{The chain polynomial}\label{subsec:overview-on-problems}

The \emph{Bergman complex} of a matroid $\M$ is defined as the order complex of the poset of proper non-empty flats $\mathcal{L}(\M)\smallsetminus\{\widehat{0},\widehat{1}\}$. This complex will be subsequently denoted by $\Delta(\widehat{\mathcal{L}}(\M))$. For basic properties it possesses, we refer to Bj\"orner's book chapter \cite{bjorner}.

The most fundamental numerical invariant associated to any simplicial complex $\Delta$ is its \emph{$f$-vector}, $f(\Delta)=(f_0,f_1,\ldots,f_d)$, counting the faces of a fixed size $i$ 
    \[ f_i = \left|\{F\in \Delta: |F| = i\}\right|,\]
for each $i\geq 0$. The number $d$ denotes the size of the largest face or facet of $\Delta$. One can conveniently arrange these numbers into a polynomial, called the \emph{$f$-polynomial}, 
    \[ f_{\Delta}(x) = \sum_{i=0}^d f_i x^{d-i}.\]
In the particular case of Bergman complexes, the numbers $f_i$ count the number of chains of flats $\widehat{0}\subsetneq F_1 \subsetneq \cdots \subsetneq F_i \subsetneq \widehat{1}$. In particular, the $f$-polynomial of $\Delta(\widehat{\mathcal{L}}(\M))$ is commonly referred to as the \emph{chain polynomial} of $\M$.

Clearly, the chain polynomial of a matroid is a linear specialization of the flag $f$-vector of $\mathcal{L}(\M)$. Therefore, Theorem~\ref{thm:flag-f-vector-valuative} allows us to conclude that it is a valuative invariant.

\begin{coro}\label{teo:chain-poly-valuative}
    The map $f:\mathfrak{M}\to \mathbb{Z}[t]$ assigning to each non-empty loopless matroid $\M$ the chain polynomial $f_{\Delta(\widehat{\mathcal{L}}(\M))}(t)$, and to each matroid either having loops or being the zero polynomial.
\end{coro}

Our original motivation to look at this invariant paper was the work of Athanasiadis and Kalampogia-Evangelinou in which they propose the following intriguing conjecture.

\begin{conj}[\cite{athanasiadis-kalampogia}]\label{conj:chain-poly}
    For every matroid $\M$, the polynomial $f_{\Delta(\widehat{\mathcal{L}}(\M))}(x)$ is real-rooted.
\end{conj}

There is an explicit expression for $f_{\Delta(\widehat{\mathcal{L}}(\M))}(x)$ for an arbitrary uniform matroid $\M$. 
Such formula can be obtained by interpreting the proper part of $\mathcal{L}(\U_{k,n})$ as the poset of non-empty faces of the $(k-1)$-th skeleton of an $n$-dimensional simplex, and using the main results of  Brenti and Welker \cite{brenti-welker}. It follows from their results that for uniform matroids the chain polynomial is real-rooted (see also \cite[Section~2.3]{athanasiadis-kalampogia}). Before stating the formula, let us introduce some terminology. 
We will denote by $A_m(x)$ the $m$-th \emph{Eulerian polynomial}, defined via
    \begin{equation} \label{eq:eulerian-polynomial}
        A_m(x) = \sum_{i=0}^{m-1} A_{m,i}\, x^i,
    \end{equation}  
for each $m\geq 1$, where $A_{m,i}$ is the Eulerian number as defined in equation \eqref{eq:def-eulerian-number}. Also, we will define $A_0(x) = 1$.

\begin{prop}\label{prop:chain-poly-uniform}
    The chain polynomial of the uniform matroid $\U_{k,n}$ is given by
    \begin{align*}
        f_{\Delta(\widehat{\mathcal{L}}(\U_{k,n}))}(x) =  \sum_{j=0}^{k-1} \binom{n}{j} A_j(x+1)\cdot x^{k-1-j}.
    \end{align*}
\end{prop}

Observe that computing by definition the chain polynomial of a given matroid might result in a computationally expensive procedure. The number of chains in a matroid elements grows exponentially on the size of the ground set. Therefore, we feel encouraged to provide fast ways of computing chain polynomials by relying on the valuative property.

The preceding result suggests that similar formulas for cuspidal matroids might be within reach. However carrying out the computations appears to be a subtle and delicate task. Nonetheless, we can write a very concrete formula for paving matroids. The following result tells that this can be achieved provided that one preprocesses chain polynomials for two classes of matroids: uniform matroids and uniform matroids with an extra coloop.

\begin{coro}\label{coro:chain-paving}
    Let $\M$ be a paving matroid of rank $k$ and cardinality $n$. The chain polynomial of $\M$ is given by
    \[ f_{\Delta(\widehat{\mathcal{L}}(\M))}(x) = f_{\Delta(\widehat{\mathcal{L}}(\U_{k,n}))}(x) - \sum_{h=k}^n \uplambda_{h} \left(f_{\Delta(\widehat{\mathcal{L}}(\U_{k,h+1}))}(x)-f_{\Delta(\widehat{\mathcal{L}}(\U_{1,1}\oplus\U_{k-1,h}))}(x)\right) \] 
    where $\uplambda_h$ denotes the number of (stressed) hyperplanes of size $h$ in $\M$.
\end{coro}

The above formula follows from Theorem~\ref{thm:main} by noticing that $\operatorname{si}(\LL_{k-1,k,h,n}) \cong \U_{k,h+1}$ and that $\operatorname{si}(\U_{1,n-h}\oplus \U_{k-1,h})\cong \U_{1,1}\oplus \U_{k-1,h}$. Notice that both polynomials $f$ and $h$ are invariant under taking simplifications. Moreover, it is possible to use \cite[Theorem~5.5]{athanasiadis-kalampogia} to derive an explicit formula for $f_{\Delta(\widehat{\mathcal{L}}(\U_{k,n}\oplus\U_{1,1}))}(x)$, i.e., uniform matroids with an extra coloop.

We take the opportunity to mention that, by relying on a variation of the formula in Corollary~\ref{coro:chain-paving} along with properties of interlacing polynomials, Br\"and\'en and Saud-Maia-Leite \cite{branden-saud-personal-comm} proved that Conjecture~\ref{conj:chain-poly} holds for all paving matroids.

Experiments and the work of Hameister, Rao, and Simpson \cite{hameister-rao-simpson} suggest an intriguing connection between chain polynomials with the Hilbert series of the Chow ring, another invariant that we will explore in the next subsection. 

\subsection{Chow rings and Chow polynomials}

Within the resolution of the Heron--Rota--Welsh conjecture by Adiprasito, Huh, and Katz \cite{adiprasito-huh-katz}, a central role is played by the Chow ring of the matroid $\M$. This is a standard graded, Artinian, Gorenstein ring constructed from the lattice of flats $\mathcal{L}(\M)$. Following the terminology introduced in \cite{bradenhuh-semismall}, we will denote the Chow ring by $\uCH(\M)$. For the precise definition of $\uCH(\M)$ and some of its basic properties, we refer to our Appendix~\ref{appendix-chow}. 

Since the Chow ring $\uCH(\M)$ of a matroid $\M$ is a graded ring, one may ask what is the dimension of each component $\uCH^j(\M)$. This amounts to determining the \emph{Hilbert--Poincar\'e series} of $\uCH(\M)$, i.e., the series
    \begin{equation} \label{eq:hilbert-series}
        \uH_{\M}(x) := \operatorname{Hilb}(\uCH(\M), x) = \sum_{j=0}^{\infty} \dim_{\mathbb{Q}}(\uCH^j(\M))\, x^j.
    \end{equation}
Since $\uCH(\M)$ is finite-dimensional as a $\mathbb{Q}$-vector space, the Hilbert--Poincar\'e series defined above yields actually a polynomial of degree $k-1$. Let us point out that the article by Ferroni, Matherne, Stevens, and Vecchi \cite{ferroni-matherne-stevens-vecchi} was conceived and written in its entirety after the first version of the present paper was posted on the arXiv and circulated among the community. Although we will refer to some useful statements in \cite{ferroni-matherne-stevens-vecchi}, the presentation here is essentially self-contained. In consonance with \cite{ferroni-matherne-stevens-vecchi}, we will often refer to the polynomial $\uH_{\M}(x)$ as the \emph{Chow polynomial of $\M$}.

The departing point to understand this invariant is the following result, established by Feichtner and Yuzvinsky in \cite[p.~526]{feichtner-yuzvinsky}.

\begin{prop}[\cite{feichtner-yuzvinsky}]\label{prop:hilbert-chow}
    Let $\M$ be a loopless matroid. The Chow polynomial of $\M$ is given by
        \[ \uH_{\M}(x) = \sum_{\varnothing = F_0 \subsetneq F_1 \subsetneq \cdots \subsetneq F_m\subseteq E} \prod_{i=1}^m \frac{x ( 1 - x^{\rk(F_i)-\rk(F_{i-1})-1})}{1-x}.\]
    The sum is taken over all chains of flats starting at the bottom element, i.e., $\varnothing = F_0\subsetneq \cdots \subseteq F_m \subseteq E$ in $\mathcal{L}(\M)$ for every $1\leq m\leq k-1$.
\end{prop}

By the hard Lefschetz property we have the isomorphisms $\uCH^j(\M)\cong \uCH^{k-1-j}(\M)$. This allows us to conclude that the polynomial $\operatorname{Hilb}(\uCH(\M),x)$ is palindromic.
Moreover, by the hard Lefschetz property in $\uCH(\M)$ it follows that the map $\uCH^j\to\uCH^{j+1}$ sending $\upxi$ to $\ell \,\upxi$ is injective for all $j\leq \lfloor \tfrac{k-1}{2} \rfloor$. Thus, the sequence of coefficients of $\uH_{\M}(x)$ is unimodal. Let us take a look at the following concrete example.

\begin{example}
    Consider the Boolean matroid $\U_{n,n}$. As we mentioned  before, the Chow ring of $\U_{n,n}$ coincides with the Chow ring of the permutohedral variety $X_E$, for $E=[n]$.  It is a classical result that the $h$-vector of the permutohedron $\Pi(n)$ has as entries the Eulerian numbers $A_{n,0}, \ldots, A_{n,n-1}$. In particular, it follows that
        \[ \uH_{\U_{n,n}}(x) = \operatorname{Hilb}(\A(X_E), x) = A_n(x),\]
    where $A_n(x)$ is the Eulerian polynomial as defined in equation \eqref{eq:eulerian-polynomial}; see also \cite[Remark 2.11]{bradenhuh-semismall}.
    By \cite[Corollary 3.13]{hampe}, the dimension of $\uCH^j(\U_{n,n})$, i.e., the Eulerian number $A_{n,j}$, is precisely the number of loopless Schubert matroids of rank $n-j$ with ground set $[n]$.
\end{example}

\begin{teo}\label{thm:hilbert-valuative}
    The map $\mathfrak{M} \to \mathbb{Z}[x]$ assigning to each loopless matroid $\M$ the Hilbert--Poincar\'e series of $\uCH(\M)$, and to each matroid with loops the zero polynomial, is a valuative invariant.
\end{teo}

\begin{proof}
    Consider $E=[n]$. For each flag of subsets of the form $\varnothing = F_0\subsetneq F_1\subsetneq \cdots \subsetneq F_m \subseteq E$ and each list of non-negative integers $r_0,\ldots,r_m$, consider the map $\upzeta_{\substack{F_0,\ldots,F_m\\r_0,\ldots, r_m}}:\mathfrak{M}_E\to \mathbb{Z}[x]$ defined by
    \begin{equation}\label{eq:valuation-hilbert-chow}
        \upzeta_{\substack{F_0,\ldots,F_m\\r_0,\ldots, r_m}} = \begin{cases}
        \displaystyle\prod_{i=1}^m \frac{x ( 1 - x^{r_i-r_{i-1}-1})}{1-x} & \text{if } F_i\in\mathcal{L}(\M) \text{ and $\rk(F_i)=r_i$ for all } 0\leq i\leq m,\\
        0 & \text{otherwise.}\end{cases}
    \end{equation}
    Observe that $\upzeta_{\substack{F_0,\ldots,F_m\\r_0,\ldots, r_m}}:\mathfrak{M}_E\to \mathbb{Z}[x]$ is a constant multiple of the map $\upphi_{\substack{F_0,\ldots,F_m\\r_0,\ldots, r_m}}:\mathfrak{M}_E\to \mathbb{Q}$ of Theorem~\ref{thm:flags-flats-rank-valuation} and hence it is a valuation. Consider now the map $\upzeta:\mathfrak{M}_E\to \mathbb{Z}[x]$ defined as the sum of all the $\upzeta_{\substack{F_0,\ldots,F_m\\r_0,\ldots, r_m}}$ over all the flags of subsets $\varnothing= F_0\subsetneq F_1\subsetneq\cdots\subsetneq F_m\subsetneq E$ and all the list of integers $r_0,\ldots,r_m$. Observe that, by Proposition~\ref{prop:hilbert-chow}, the map $\upzeta:\mathfrak{M}_E\to \mathbb{Z}[x]$ is assigning to each matroid the Hilbert--Poincar\'e series of its Chow ring. Since it is a sum of valuations, it follows that $\upzeta$ is a valuation. 
\end{proof}

\begin{remark}
    Dotsenko conjectured \cite[Conjecture~2]{dotsenko} that Chow rings of matroids are Koszul. This was recently proved true by Mastroeni and McCullough \cite{mastroeni-mccullough}. In particular, the Koszulness of the Chow ring implies that the \emph{Poincar\'e series} of the Chow ring of a matroid, denoted by $\operatorname{Poin}(\uCH(\M), x)$, is determined by $\operatorname{Hilb}(\uCH(\M),x)$ via:
        \[ \operatorname{Poin}(\uCH(\M), x) = \frac{1}{\operatorname{Hilb}(\uCH(\M), -x)}.\]
    We refer to \cite{froberg} for the precise definitions of Koszulness, Poincar\'e series, and a proof of this equality. From this, the reader may check directly that the Poincar\'e series of the Chow ring cannot be a valuative invariant itself, although it is very close to being one.
\end{remark}

As a consequence of the valuativeness of Chow polynomials, we can use Theorem~\ref{thm:main} to get the following concrete expression for paving matroids.

\begin{teo}\label{thm:chow-polynomial-paving}
    Let $\M$ be a paving matroid of rank $k$ and cardinality $n$. Then,
    \[ \uH_{\M}(x) = \uH_{\U_{k,n}}(x) - \sum_{h\geq k}\uplambda_h \left(\uH_{\U_{k,h+1}}(x) - \uH_{\U_{k-1,h}\oplus \U_{1,1}}(x) \right),\]
    where $\uplambda_h$ denotes the number of stressed hyperplanes of size $h$ in $\M$.
\end{teo}

\begin{proof}
    Since the map $\M \mapsto \operatorname{Hilb}(\uCH(\M), x)$ is valuative by Theorem~\ref{thm:hilbert-valuative}, we can apply Theorem~\ref{thm:main}. In particular, since $\M$ is paving, all the stressed subsets with non-empty cover correspond to $r=k-1$. Hence, denoting by $\uplambda_h$ the number of these subsets having size $h$, i.e., the number of stressed hyperplanes of size $h$, we obtain
    \[\uH_{\M}(x)= \uH_{\U_{k,n}}(x) - \sum_{h}\uplambda_{h} \left(\uH_{\LL_{k-1,k,h,n}}(x) - \uH_{\U_{1,n-h}\oplus\U_{k-1,h}}(x) \right).\]
    We now use that the map $\M \mapsto \operatorname{Hilb}(\uCH(\M), x)$ is trivially invariant under simplification for loopless matroids, that is $\operatorname{si}(\LL_{k-1,k,h,n})\cong \U_{k,h+1}$ and $\operatorname{si}(\U_{1,n-h}\oplus\U_{k-1,h}) \cong \U_{1,1}\oplus \U_{k-1,h}$ which yields the statement.
\end{proof}

The reader may wonder to what extent the preceding formula can be applied in practice. To answer that question, let us start by mentioning what happens in the case of uniform matroids.

\begin{prop}[{\cite[Theorem~1.9]{ferroni-matherne-stevens-vecchi}}]\label{prop:chow-poly-uniform}
    The following identity holds:
    \[ \uH_{\U_{k,n}}(x) = \sum_{j=0}^{k-1} \binom{n}{j} d_j(x) (1+x+\cdots+x^{k-1-j}). \] 
\end{prop}

In the above statement $d_j(x)$ stand for the classical derangement polynomials.
Hameister, Rao and Simpson, \cite[Conjecture~6.2]{hameister-rao-simpson} conjectured and Ferroni, Matherne, Stevens and Vecchi proved in \cite[Corollary~3.17]{ferroni-matherne-stevens-vecchi} a different formula for the Chow polynomial of uniform matroids in terms of the $h$-polynomial of the Bergman complex of $\M$.

The case of uniform matroids with an extra coloop can be handled using a very special case of the semi-small decomposition of Braden, Huh, Matherne, Proudfoot, and Wang \cite[Theorem~1.2]{bradenhuh-semismall}. In this very special case, one should remove precisely the coloop, to obtain:
    \begin{equation}
        \uH_{\U_{k,n}\oplus\U_{1,1}}(x) = (1+x)\, \uH_{\U_{k,n}}(x) + x \sum_{j=1}^{k-1} \binom{n}{j}\, A_j(x)\, \uH_{\U_{k-j,n-j}}(x) .
    \end{equation}

We refer to \cite[Remark~3.30]{ferroni-matherne-stevens-vecchi} for more details about this formula. In particular, every term appearing in the formula of Theorem~\ref{thm:chow-polynomial-paving} can be computed effectively.

Let us also mention that in \cite[Theorem~5.1]{hameister-rao-simpson} a combinatorial description of the coefficients of $\uH_{\U_{k,n}}(x)$ was obtained in terms of certain statistics of permutations. On the other hand, Backman, Eur, and Simpson \cite[Theorem~3.3.8]{BackmanEurSimpson} proved an interpretation for the coefficients of the Chow polynomial in terms of relative nested quotients. A consequence of any of the two mentioned results is that $\uH_{\U_{n-1,n}}(x)$ coincides with the derangement polynomial $d_n(x)$. This is not obvious from Proposition~\ref{prop:chow-poly-uniform}, as putting $k=n-1$ yields a more complicated expression which also involves derangement polynomials; we refer to \cite[Corollary~4.2]{juhnke-murai-sieg} for an explanation of this surprising coincidence. In other words, for uniform matroids of corank $0$ and corank $1$, the Hilbert--Poincar\'e series of the Chow ring yields a real-rooted polynomial; for the proof of the real-rootedness of the Eulerian and the derangement polynomials and further generalizations, we point to \cite{gustafsson-solus}. We propose the following conjecture.

\begin{conj}\label{conj:hilbert-real-rooted}
    The Hilbert--Poincar\'e series of the Chow ring of any matroid is a real-rooted polynomial.
\end{conj}

Notice that this assertion is a considerable strengthening of the unimodality of the coefficients of these polynomials, which follows from the hard Lefschetz property for the Chow rings. Other two polynomials for which this phenomenon is also conjectured are the Kazhdan--Lusztig polynomial \cite{gedeonsurvey} and the $Z$-polynomial \cite{proudfootzeta}. In addition, our conjecture bears a resemblance with Conjecture~\ref{conj:chain-poly}, which asserts the real-rootedness of the chain polynomial; the connection between these invariants deserves further study. 

We have verified that indeed the Hilbert--Poincar\'e series of the Chow rings of all sparse paving matroids with at most $40$ elements are real-rooted, thus supporting our Conjecture~\ref{conj:hilbert-real-rooted}. Furthermore, in the paper by Ferroni, Matherne, Stevens, and Vecchi \cite{ferroni-matherne-stevens-vecchi} it is proved that the polynomials $\uH_{\M}(x)$ are $\gamma$-positive; this was also proved independently by Botong Wang \cite{reiner-wang-personal-comm}. This constitutes a good piece of evidence backing up our Conjecture~\ref{conj:hilbert-real-rooted}.

We end this discussion by mentioning two more invariants that the reader may ask about: the Hilbert--Poincar\'e series of both the \emph{augmented} Chow ring, introduced in \cite{bradenhuh-semismall,bradenhuh} and of the \emph{conormal} Chow ring, defined in \cite{ardila-denham-huh}. A natural problem that one can consider is computing them for uniform matroids and proving their valuativeness. For the augmented Chow ring, this is done in \cite{ferroni-matherne-stevens-vecchi}. The case of the conormal Chow ring is more subtle, and its Hilbert--Poincar\'e series fails to be valuative \cite{larson-personal-comm}.

\section{Invariants from singular Hodge theory, K-theory, and beyond}\label{sec:nine}

\noindent In this section we address three families of invariants: first, the Kazhdan--Lusztig and $Z$-polynomials; second, the $g$-polynomial defined by Speyer; and third, the spectrum polynomial and a generalization by Denham. Each subsection can be read independently of the other two.

We rely on our framework to prove new results on these invariants. For the Kazhdan--Lusztig and $Z$-polynomials, we show in Theorem~\ref{thm:kl-z-corank2} a formula for the Kazhdan--Lusztig polynomial of corank $2$ matroids; we use it to prove a conjecture by Gedeon in this case. For the $g$-polynomial, we establish a compact formula for paving matroids in Theorem~\ref{teo:speyer-for-paving}, extending a result by Speyer on rank $2$ matroids; we use it to prove the non-negativity conjecture of Speyer \cite{speyer} in this case. For the spectrum and the Denham polynomial, we show that they are valuations; in addition, by providing an explicit formula for all split matroids we are able to construct Example~\ref{ex:same_spec}, answering a question posed by Kook, Reiner, and Stanton \cite{kook-reiner-stanton}.

\subsection{The Kazhdan--Lusztig polynomials and \texorpdfstring{$Z$}{Z}-polynomials} \label{sec:kl-polys}

In analogy to the Kazhdan--Lusztig polynomials associated to Bruhat intervals in Coxeter groups, Elias, Proudfoot and Wakefield introduced in \cite{elias-proudfoot-wakefield} a variation of the Kazhdan--Lusztig (KL) polynomials for matroids. This new theory has received considerable attention in the past few years. 

We will see how Theorem~\ref{thm:main} gives a unified explanation for the results of \cite{lee-nasr-radcliffe,ferroni-vecchi,ferroni-nasr-vecchi,karn-proudfoot-nasr-vecchi}. In addition, we will establish a new explicit formula for the Kazhdan--Lusztig polynomial of all corank $2$ matroids, which constitutes an original result.

\subsubsection{A quick review}  

Let us include a brief recapitulation of the definitions and relevant properties of the Kazhdan--Lusztig polynomials of matroids.

\begin{teo}\label{PM}
    There is a unique way to assign to each loopless matroid $\M$ a polynomial $P_{\M}(t) \in \mathbb{Z}[t]$ such that the following properties hold:
    \begin{enumerate}[\normalfont(i)]
        \item If $\rk(\M) = 0$, then $P_{\M}(t) = 1$.
        \item If $\rk(\M) > 0$, then $\deg P_{\M}(t) < \frac{1}{2} \rk(\M)$.
        \item For every matroid $\M$, the polynomial
            \[ Z_{\M}(t) := \sum_{F\in \mathcal{L}(\M)} t^{\rk(F)}\, P_{\M/F}(t)\]
        is palindromic.
    \end{enumerate}
\end{teo}

The polynomials $P_{\M}(t)$ and $Z_{\M}(t)$ are called the \emph{Kazhdan--Lusztig polynomial} and $Z$-polynomial of the matroid $\M$, respectively. These polynomials are multiplicative invariants, and do not change under taking simplifications if the matroid is loopless. For these and other basic properties, we refer to the articles \cite{elias-proudfoot-wakefield,proudfootzeta}. For matroids with loops, both of these invariants will be set to be zero by convention.

Apart from the resolution of the top-heavy conjecture, another of the main contributions of Braden, Huh, Matherne, Proudfoot and Wang in \cite{bradenhuh} is the construction of a fundamental object, called the \emph{intersection cohomology module} of $\M$, denoted $\IH(\M)$. A related construction is the stalk of the intersection cohomology at the empty flat, denoted $\IH(\M)_{\varnothing}$. For details and terminology we refer to their original paper. One of the central properties of these two objects is given by the following result.

\begin{teo}[{\cite[Theorem~1.9]{bradenhuh}}]
    The Hilbert--Poincar\'e series of $\IH(\M)$ is the $Z$-polynomial of $\M$. Also, the Hilbert--Poincar\'e series of $\IH(\M)_{\varnothing}$ is the Kazhdan--Lusztig polynomial of $\M$. In particular, both $P_{\M}(t)$ and $Z_{\M}(t)$ have non-negative coefficients.
\end{teo}

\subsubsection{The valuativeness for KL polynomials}

A useful result by Ardila and Sanchez \cite[Theorem~8.9]{ardila-sanchez} is that the Kazhdan--Lusztig polynomial is a valuative invariant. Their proof relies on a subtle induction argument. We give an alternative proof of this valuativeness below, by relying on a brute-force description of the Kazhdan--Lusztig coefficients in terms of entries of the flag $f$-vector. Moreover, we also include a proof of the valuativeness of the $Z$-polynomial, for the sake of future referencing.

\begin{teo}
    The assignments $\M \mapsto P_{\M}(t)$ and $\M\mapsto Z_{\M}(t)$ are valuative invariants.
\end{teo}

\begin{proof}
    A formula of Wakefield in \cite[Theorem~11]{wakefield} shows that each coefficient of the Kazhdan--Lusztig polynomial of $\M$ is an integer combination of \emph{flag Whitney numbers}. These are by definition entries of the flag $f$-vector of $\mathcal{L}(\M)$. Hence, Theorem~\ref{thm:flag-f-vector-valuative} yields the valuativeness of the Kazhdan--Lusztig polynomial. For the $Z$-polynomial it suffices to apply the convolution theorem by Ardila and Sanchez \ref{thm:convolutions1}, because the $Z$-polynomial is the convolution of the rank function and the Kazhdan--Lusztig polynomial, both of which are valuations (cf. Example~\ref{ex:rank-function}).
\end{proof}

\begin{remark}\label{remark:gamma}
    In general, to any palindromic polynomial $p(x)$ of degree $d$ having integer coefficients, one may associate a polynomial $\gamma(x)$ of degree $\lfloor\frac{d}{2}\rfloor$ having integer coefficients, and satisfying 
        \[ p(x) = \gamma\left(\tfrac{x}{(x+1)^2}\right)\cdot (1+x)^d.\]
    If the polynomial $\gamma(x)$ has nonnegative coefficients, one usually says that $p(x)$ is \emph{$\gamma$-positive}.
    Since $Z_{\M}(t)$ is indeed palindromic and has degree $k=\rk(\M)$, one can consider a polynomial $\gamma_{\M}(t)\in\mathbb{Z}[t]$ of degree $\lfloor\frac{\rk(\M)}{2}\rfloor$ and satisfying:
        \[ Z_{\M}(t) = \gamma_{\M}\left( \tfrac{t}{(t+1)^2}\right) \cdot (1+t)^{\rk(\M)}.\]
    It is not difficult to conclude that the polynomial $\gamma_{\M}(t)$ satisfies analogous properties. The valuativeness of the $Z$-polynomial implies that the map $\M \mapsto \gamma_{\M}(t)$ is a valuative invariant as well. 
\end{remark}

We want to warn the reader that the calculation of the Kazhdan--Lusztig and $Z$-polynomials, even for uniform matroids, is a difficult task. 
Using the definitions as they were stated, the computation of these polynomials with a computer is slow for matroids of rather small cardinality. 
It is a challenging open problem to find a combinatorial interpretation of their coefficients. To have a glimpse of the type of polynomial we are facing, we present the following example.

\begin{example}
    For uniform matroids we have the following closed formulas. As mentioned above, their derivation is non trivial and, in some cases, one needs to use an ``equivariant'' version of the Kazhdan--Lusztig polynomials to get an expression that later is simplified by applying binomial identities (see \cite{gedeon-proudfoot-young-equivariant}).
    We refer to \cite[Theorem~1.3]{gao-uniform} for further details. We have:
    \begin{align*}
      P_{\U_{k,n}}(t) &= \sum_{j=0}^{\lfloor\frac{k-1}{2}\rfloor}\left( \sum_{i=0}^{n-k-1} \frac{1}{k-j} \binom{n}{j}\binom{k+i-j}{i+j+1}\binom{i+j-1}{i}\right) t^j.
     \end{align*}
     Notice that in the definition of the $Z$-polynomial, we have summands of the form $t^{\rk(F)} P_{\M/F}(t)$ where $F$ is a flat of $\M$. In particular, for $\M = \U_{k,n}$ all the contractions are uniform matroids. Hence, an elementary counting argument yields
     \[ Z_{\U_{k,n}}(t) = t^k + \sum_{j=0}^{k-1} \binom{n}{j} \cdot t^j\cdot  P_{\U_{k-j,n-j}}(t).\]
\end{example}

Although their computation is complicated in general, given that all these invariants are valuative, our Theorem~\ref{thm:main} gives the following result.

\begin{teo}\label{thm:kl-for-split}
    Let $\M$ be an elementary split matroid of rank $k$ and cardinality $n$. Then,
    \begin{align*}
        P_{\M}(t) &= P_{\U_{k,n}}(t) - \sum_{r,h} \uplambda_{r,h} \left( P_{\LL_{r,k,h,n}}(t) - P_{\U_{k-r,n-h}}(t) P_{\U_{r,h}}(t)\right),\\
        Z_{\M}(t) &= Z_{\U_{k,n}}(t) - \sum_{r,h} \uplambda_{r,h} \left( Z_{\LL_{r,k,h,n}}(t) - Z_{\U_{k-r,n-h}}(t) Z_{\U_{r,h}}(t)\right),\\
        \gamma_{\M}(t) &= \gamma_{\U_{k,n}}(t) - \sum_{r,h} \uplambda_{r,h} \left( \gamma_{\LL_{r,k,h,n}}(t) - \gamma_{\U_{k-r,n-h}}(t) \gamma_{\U_{r,h}}(t)\right),
    \end{align*}
    where $\uplambda_{r,h}$ denotes the number of stressed subsets of rank $r$ and size $h$ that have non-empty cover.
\end{teo}

If one specializes the preceding Theorem~\ref{thm:kl-for-split} to a paving matroid $\M$, then the only stressed subsets with non-empty cover satisfy $r=k-1$. Thus the index of the sums appearing above depend only on the variable $h$. 
Furthermore, the simplification of a cuspidal matroid with $r=k-1$ is given by $\operatorname{si}(\LL_{k-1,k,h,n}) \cong \U_{k-1,h}$. 
This retrieves \cite[Theorem~1.4]{ferroni-nasr-vecchi}, so our Theorem~\ref{thm:kl-for-split} can be seen as a generalization of that result\footnote{Here we omitted the case of the \emph{inverse Kazhdan--Lusztig polynomial}, but this can be included with a similar reasoning.}.
Furthermore, observe that the proof in the paper by Ferroni, Nasr and Vecchi was carried out without relying on the valuativeness of these invariants and therefore is much more involved. 

\begin{remark}
    It is of high interest to produce formulas for all cuspidal matroids. The preceding theorem would immediately yield a handy expression of these invariants for any elementary split matroids. 
    Notice that for paving matroids this issue is handled by the arguments we mentioned in the prior paragraph. 
    In such a case, the only cuspidal matroids in the formulas are of the form $\LL_{k-1,k,h,n}$ and their  Kazhdan--Lusztig polynomials coincide with those of the uniform matroid $\U_{k-1,h}$. 
\end{remark}

There are only few matroids for which explicit formulas for the Kazhdan--Lusztig invariants $P_\M(t)$ and $Z_\M(t)$ are known. In addition to uniform matroids, there are formulas for some graphic matroids, e.g., fans, wheels, and whirls \cite{wheels-whirls}, thagomizer graphs and some particular bipartite graphs \cite{gedeon-thagomizer}, parallel connection graphs \cite{braden-deletion} and complete graphs \cite{karn-wakefield,ferroni-larson}.  
Explicit formulas for the Kazhdan--Lusztig polynomial of sparse paving matroids are presented in \cite{lee-nasr-radcliffe}, which are derived via an enumeration of certain (skew) Young tableaux. Such expressions and their analogues for the $Z$-polynomials for sparse paving matroids were found independently with another approach in \cite{ferroni-vecchi}. The framework in which the two mentioned articles derive their key results are particular instances of Theorem~\ref{thm:main}, along with properties that the discussed polynomials satisfy.

\subsubsection{Formulas for corank $2$ matroids}\label{sec:kl-corank2}

Since the class of all elementary split matroids contains all copaving matroids, in particular it contains the class of all corank $2$ matroids without coloops. Using some basic properties from the Kazhdan--Lusztig polynomials, we will prove explicit formulas that allow to compute easily $P_\M(t)$ for an arbitrary corank $2$ matroid. This is our next goal.

In \cite{braden-deletion}, Braden and Vysogorets derived formulas that write the Kazhdan--Lusztig and the $Z$-polynomial of a matroid in terms of their counterparts for the matroid obtained by deleting one element of the ground set. In particular, the expression they obtained can be used directly to compute the Kazhdan--Lusztig and the $Z$-polynomial of the matroids $\mathsf{C}_{a,b}$ that we introduced in Section~\ref{sec:subdivision-rank2}.

\begin{lemma}[{\cite[Theorem~3.2 \& Example~3.4]{braden-deletion}}]\label{lemma:kl-for-cab}
    The Kazhdan--Lusztig polynomial and the $Z$-polynomial of the matroid $\mathsf{C}_{a,b}$ are given by
    \begin{align*} 
        P_{\mathsf{C}_{a,b}}(t) &= P_{\U_{n-2,n-1}}(t) - t P_{\U_{a-2,a-1}}(t)P_{\U_{b-2,b-1}}(t),\\
        Z_{\mathsf{C}_{a,b}}(t) &= Z_{\U_{n-2,n-1}}(t).
    \end{align*}
    where $n = a + b - 1$.
\end{lemma}

If we are willing to use Theorem~\ref{thm:kl-for-split} in order to calculate the Kazhdan--Lusztig and the $Z$-polynomial of an arbitrary corank $2$ matroid, the difficulty that we have to overcome is to evaluate them for all cuspidal matroids of corank $2$. 
By looking at Definition~\ref{def:cuspidal} we see that $0\leq k-r\leq n-h$ is required. The condition on the rank tells us that $k = n-2$, and hence $r\leq h \leq n-k+r = r + 2$. 
This leaves us with three possible values for $h$. If $h=r$ or $h=r+2$, then $\LL_{r,n-2,h,n}$ is either $\LL_{r,n-2,r,n} \cong \U_{n-2-r,n-r}\oplus \U_{r,r}$ by equation \eqref{eq:cuspidal-border-case} or $\LL_{r,n-2,r+2,n}\cong \U_{n-2,n}$ by Remark~\ref{remark:cases-cuspidal-is-uniform}, respectively.
The remaining case $h=r+1$ can be treated with Corollary~\ref{coro:valuative-invariant-corank2}.

Observe that the Kazhdan--Lusztig polynomial of a corank $2$ matroid with $\ell$ coloops does not change when contracting all coloops, whereas the $Z$-polynomial changes by a factor $(t+1)^{\ell}$. In particular, to simplify our computation, we may restrict our attention to coloopless matroids.

\begin{teo}\label{thm:kl-z-corank2}
    Let $\M$ be a corank $2$ matroid without coloops. Assume that for each $r$, $\uplambda_r$ denotes the number of stressed subsets of rank $r$ and size $r+1$ in $\M$. Then
    \begin{align*}
        P_{\M}(t) &= P_{\U_{n-2,n}}(t)-\\ 
        &\quad \sum_{r} \uplambda_r\left((n-r-2)P_{\U_{n-2,n-1}}(t) -  \sum_{a=2}^{n-r-1} P_{\U_{n-a-1,n-a}}(t)\left(P_{\U_{a-1,a}}(t)+t P_{\U_{a-2,a-1}}(t)\right) \right),\\
        Z_{\M}(t) &= Z_{\U_{n-2,n}}(t) - \sum_{r} \uplambda_r \left( (n-r-2)Z_{\U_{n-2,n-1}}(t) - \sum_{a=2}^{n-r-1} Z_{\U_{a-1,a}}(t)Z_{\U_{n-a-1,n-a}}(t)\right).
    \end{align*}
\end{teo}

\begin{proof}
    Let us apply Theorem~\ref{thm:kl-for-split}. Observe that in a corank $2$ matroid $\M$, all subsets with non-empty cover satisfy $r+1=h$, because $r=h$ implies being independent and $r+2=h$ implies that there are coloops in $\M$. In particular, this explains why the sum appearing in the formula of the statement depends only on $r$. Since the Kazhdan--Lusztig and the $Z$-polynomial are valuative invariants, we can use Corollary~\ref{coro:valuative-invariant-corank2} to compute them for the matroids $\LL_{r,n-2,r+1,n}$. We will indicate the proof for $P$, given that for $Z$ the statements are completely analogous. Using Lemma~\ref{lemma:kl-for-cab} we have
    \[
        P_{\LL_{r,n-2,r+1,n}}(t) = \sum_{a=2}^{n-r-1}\left( P_{\U_{n-2,n-1}}(t) - t P_{\U_{a-2,a-1}}(t)P_{\U_{n-a-1,n-a}}(t)\right) - \sum_{a=2}^{n-r-2} P_{\U_{a-1,a}}(t) P_{\U_{n-a-1,n-a}}(t).
    \]
    Notice that both sums have different upper limits. However, in the formula of Theorem~\ref{thm:main} we have to consider $f(\LL_{r,n-2,r+1,n}) - f(\U_{r,r+1}\oplus \U_{n-r-2,n-r-1})$. For the Kazhdan--Lusztig polynomial, this correspond to subtracting exactly the missing term in the second sum above. Grouping those accordingly leads to the expression of the statement. 
\end{proof}

In \cite{lee-nasr-radcliffe} a conjecture attributed to Gedeon asserts that if one fixes the rank $k$ and the cardinality $n$ of a matroid, then the uniform matroid $\U_{k,n}$ attains the coefficient-wise maximum Kazhdan--Lusztig polynomial; see also \cite[Conjecture 1.6]{karn-proudfoot-nasr-vecchi}. Analogous statements are expected to be true also for  $Z_{\M}(t)$ and $\gamma_{\M}(t)$ in general; some evidence supporting such conjectures is is that these are known facts for paving matroids \cite{ferroni-nasr-vecchi,karn-proudfoot-nasr-vecchi}. As a consequence of the preceding result, it is possible to prove by using properties of the Narayana and Motzkin numbers that this conjecture is true for the Kazhdan--Lusztig and the $Z$-polynomial for matroids of corank $2$. We will state our result and indicate the proof only for $Z_{\M}(t)$. For $P_{\M}(t)$, the proof gets slightly cumbersome as the terms that we must handle are more complicated. 

\begin{teo}
    Let $\M$ be a coloopless matroid of corank $2$ and cardinality $n$. Then $Z_{\M}(t)$ is coefficient-wise smaller than $Z_{\U_{n-2,n}}(t)$.
\end{teo}

\begin{proof}
    The claim will follow if we prove that  
    \begin{equation} \label{eq:convolution-narayana}
    Z_{\U_{a-1,a}}(t) \cdot Z_{\U_{n-a-1,n-a}}(t) \preceq Z_{\U_{n-2,n-1}}(t),
    \end{equation}
    where $\preceq$ denotes coefficient-wise inequality. This is because in the formula that we derived for $Z_{\M}$ we are subtracting from $Z_{\U_{n-2,n}}(t)$ sums that consist of $n-r-2$ terms, each of the form $Z_{\U_{n-2,n-1}}(t) - Z_{\U_{n-a-1,n-a}}(t)Z_{\U_{a-1,a}}(t)$, for $a=2,\ldots,n-r-1$.
    
    We will prove an even stronger fact: that the inequality in \eqref{eq:convolution-narayana} holds if we use the $\gamma$-polynomial instead.
    \begin{equation} \label{eq:convolution-motzkin}
    \gamma_{\U_{a-1,a}}(t) \cdot \gamma_{\U_{n-a-1,n-a}}(t) \preceq \gamma_{\U_{n-2,n-1}}(t),
    \end{equation}
    By \cite[Remark~5.10]{ferroni-nasr-vecchi}, we have the formula
    \[ [t^i] \gamma_{\U_{k-1,k}}(t) = \frac{1}{k-1-i} \binom{k-1-i}{i}\binom{k-1}{i+1}.\]
    for each $k\geq 1$. According to the OEIS \cite[A055151]{oeis}, we have
    \[ [t^i] \gamma_{\U_{k-1,k}}(t) = \left|\{\text{Motzkin paths from $(0,0)$ to $(k-1,0)$ with $i$ north-east steps}\}\right|,\]
    where a Motzkin path is a lattice path with steps of the form $+(1,0)$, $+(1,1)$ and $+(1,-1)$ that stays always above the $x$-axis. For example in \cite[Figure~3]{oste-vanderjeugt} one can see a picture of all the Motzkin paths from $(0,0)$ to $(4,0)$. There are two paths that have two north-east steps, six paths that have one north-east step, and there is only one path without north-east steps. This corresponds with $\gamma_{\U_{4,5}}(t) = 2t^2+6t+1$.
    Now, the inequality in equation \eqref{eq:convolution-motzkin} becomes immediate, because comparing the coefficients of degree $m$, we must prove that
    \[ \sum_{i=0}^m [t^i]\gamma_{\U_{a-1,a}}(t) \cdot [t^{m-i}]\gamma_{\U_{n-a-1,n-a}}(t) \leq [t^m] \gamma_{\U_{n-2,n-1}}(t).\]
    The right-hand-side is the number of Motzkin paths from $(0,0)$ to $(n-2,0)$ that have $m$ north-east steps. Each summand of the left-hand-side can be thought as enumerating a concatenation of a Motzkin path from $(0,0)$ to $(0,a-1)$ and a Motzkin path from $(a-1,n-2)$, in the first part having $i$ north-east steps, and in the second part having $m-i$ north-steps. Notice that this establishes an injection, since the number $a$ is fixed.
\end{proof}

The main open conjectures regarding the Kazhdan--Lusztig theory of matroids \cite[Conjecture~3.2]{gedeonsurvey} and \cite[Conjecture~5.1]{proudfootzeta} state that $P_{\M}(t)$ and $Z_{\M}(t)$ are real-rooted polynomials. Using the formulas of Theorem~\ref{thm:kl-z-corank2} and the assistance of a computer, we were able to show that.

\begin{teo}
    If $\M$ is a corank $2$ matroid on a ground set of size at most $80$, then $P_{\M}(t)$ and $Z_{\M}(t)$ are real-rooted.
\end{teo}

We end this section with a few remarks and suggestions of problems that may presumably be attacked using the techniques that we presented.

For matroids of rank $3$, being simple is equivalent to being paving. In particular, all of the cosimple matroids of corank $3$ are copaving and hence elementary split. We speculate that it might possible to extend the results in this section to this case. In particular, deriving a formula for all such matroids would be of particular interest, especially for the sake of a computer search of potential counterexamples to the real-rootedness conjectures.

\subsection{Covaluations and Speyer's polynomials}\label{sec:speyer}

Speyer introduced in \cite{speyer} an invariant for matroids that originates in the $K$-theoretic study of the Grassmannian. Although the definition is somewhat involved, in the end this invariant interacts nicely with matroidal subdivisions of the base polytope. This map is not a valuation, but it satisfies a similar property. Derksen and Fink called this property ``covaluative'' in \cite{derksen-fink}. In this section we first review this notion and derive an analog of our Theorem~\ref{thm:main} for covaluations. After that, we discuss the basic properties of the $g$-polynomial defined by Speyer, and we prove the following new results: a closed formula for the $g$-polynomial of arbitrary paving matroids, which we use to conclude the positivity of the coefficients of the $g$-polynomial of all sparse paving matroids.

\subsubsection{Covaluations and covaluative invariants} As we have explained in Section~\ref{sec:polytopes}, valuations are the maps that, roughly speaking, respect inclusion-exclusion under polytope subdivisions. If one disregards the signs appearing in equation~\eqref{eq:weak-valuation-inclusion-exclusion}, the resulting notion is that of \emph{weak covaluation}. Similarly, a notion of \emph{strong covaluation} can be defined by modifying Definition~\ref{def:strong-valuation} and asking for maps that factor through indicator functions of \emph{interiors} of polytopes. 

Derksen and Fink \cite{derksen-fink} proved that, for matroid polytopes, the notions of weak and strong covaluations agree. In particular, we will again drop the words `weak' and `strong' in front of covaluation. In fact, the following concise definition comprises all the preceding discussion.

\begin{defi}
    Let $\mathrm{A}$ be an abelian group. A map $f:\mathfrak{M}\to \mathrm{A}$ is said to be a \emph{covaluation} if the map $g(\M) := (-1)^{c(\M)} f(\M)$ is a valuation.
\end{defi}

In the above definition, $c(\M)$ denotes the number of connected components of the matroid $\M$. Of course, there are counterparts of Theorem~\ref{thm:main} and Corollary~\ref{thm:main-for-sparse-paving} for covaluations. A detail to take into account is that more care is needed when dealing with disconnected matroids.

\begin{teo}\label{thm:maincovaluative}
    Let $\M$ be a connected split matroid of rank $k$ and cardinality $n$, and let $f$ be a covaluative invariant then
    \[ f(\M) = f(\U_{k,n}) - \sum_{r,h} \uplambda_{r,h} \left(f(\LL_{r,k,h,n}) + f(\U_{r,h}\oplus \U_{k-r,n-h})\right),  \]
    where $\uplambda_{r,h}$ denotes the number of stressed subsets with non-empty cover of size $h$ and rank $r$.
\end{teo}

\begin{coro}\label{coro:main-for-sparse-paving-covaluative}
    Let $\M$ be a connected sparse paving matroid of rank $k$ and cardinality $n$ and having exactly $\uplambda$ circuit-hyperplanes. Let $f$ be a covaluative invariant, then
    \[ f(\M) = f(\U_{k,n}) - \uplambda \left(f(\mathsf{T}_{k,n}) + f(\U_{k-1,k}\oplus \U_{1,n-k})\right).  \]
\end{coro}

\subsubsection{Speyer's polynomial} 
As mentioned before, a famous covaluative invariant is the $g$-polynomial, introduced by Speyer in \cite{speyer} and further studied in \cite{fink-speyer}. It has an intricate definition that originates in $K$-theory. However, by taking advantage of the results in our Appendix~\ref{appendix-series-parallel}, we can give a handy definition of the $g$-polynomial by means of the following characterization.

\begin{teo}
    There exists a unique invariant $g:\mathfrak{M}\to \mathbb{Z}[t]$, defined by $\M\mapsto g_{\M}(t)$ satisfying the following properties.
    \begin{enumerate}[\normalfont (a)]
        \item $g$ is covaluative.
        \item If $\M$ has loops or coloops, then $g_{\M}(t) = 0$.
        \item If $\M$ is loopless and coloopless and is isomorphic to a direct sum of $m$ series-parallel matroids, then $g_{\M}(t)=t^m$.
    \end{enumerate}
    The polynomial $g_\M(t)$ is Speyer's $g$-polynomial. 
\end{teo}

To the best of our knowledge, this is the first time that this statement appears as a theorem in the literature. The existence of an invariant satisfying all these properties is proved by Speyer in \cite{speyer}. However, in the arXiv version of \cite{speyer}\footnote{\url{https://arxiv.org/pdf/math/0603551.pdf}} (which differs substantially from the published version), Conjecture~11.3 only \emph{postulates the uniqueness} of such an invariant. This actually follows immediately from the results in our Appendix~\ref{appendix-series-parallel}.

There are several ways of deducing the following explicit expression for the $g$-polynomial of a uniform matroid. A possible proof is following the subdivisions of Proposition~\ref{prop:subdiv-lpm}, and we encourage the readers to try that approach by themselves.

\begin{prop}[{\cite[Proposition~10.1]{speyer}}]
    The $g$-polynomial of the uniform matroid $\U_{k,n}$, where $1\leq k\leq n-1$ is given by the polynomial
        \[ g_{\U_{k,n}}(t) = \sum_{i=1}^{k} \binom{n-i-1}{k-i}\binom{n-k-1}{i-1} \,t^i\enspace .\]
\end{prop}

Observe that this implies that $g_{\U_{n-1,n}}(t) = g_{\U_{1,n}}(t) = t$ for every $n\geq 2$. On the other hand, $g_{\U_{0,0}}(t)=1$ whereas $g_{\U_{0,n}}(t) = g_{\U_{n,n}}(t) = 0$ for $n>0$. 

Now let us state a short list of properties this invariant possesses. These and further results were proved in \cite{speyer}, \cite{derksen-fink} and \cite{fink-speyer}.

\begin{prop}\label{prop:properties-of-speyer}
    The $g$-polynomial of a matroid $\M$ has the following properties.
    \begin{enumerate}[\normalfont (a)]
        \item It is multiplicative, namely $g_{\M_1\oplus\M_2}(t) = g_{\M_1}(t) \cdot g_{\M_2}(t)$.
        \item It is invariant under taking duals, i.e., $g_{\M^*}(t) = g_{\M}(t)$.
        \item If $\M$ is loopless, then $g$ is invariant under simplification, namely $g_{\operatorname{si}(\M)}(t) = g_{\M}(t)$  .
    \end{enumerate}
\end{prop}

The following result provides a way of computing the $g$-polynomial for any paving matroid. This can be seen as a generalization of \cite[Proposition~10.3]{speyer}, which is an explicit formula for the $g$-polynomial of all loopless matroids of rank $2$ (recall that such matroids are in particular paving).

\begin{teo}\label{teo:speyer-for-paving}
    Let $\M$ be a connected paving matroid of rank $k$ and cardinality $n$. Then
    \[ g_{\M}(t) = g_{\U_{k,n}}(t) - \sum_{h = k}^{n-1} \uplambda_h \left( g_{\U_{k,h+1}}(t) + t\cdot g_{\U_{k-1,h}}(t)\right),\]
    where $\uplambda_h$ denotes the number of (stressed) hyperplanes of $\M$ of cardinality $h$.
\end{teo}

\begin{proof}
    Notice that by Theorem~\ref{thm:maincovaluative}, it suffices to consider the $g$-polynomials of the matroids $\LL_{r,k,h,n}$ and $\U_{r,h}\oplus \U_{k-r,n-h}$. Since $\M$ is assumed to be paving, in fact all the stressed cyclic flats that we can relax are hyperplanes. Hence it suffices to consider the $g$-polynomial of $\LL_{k-1,k,h,n}$ and $\U_{k-1,h}\oplus \U_{1,n-h}$. 
    The simplification of $\LL_{k-1,k,h,n}$ is the uniform matroid $\operatorname{si}(\LL_{k-1,k,h,n})\cong\U_{k,h+1}$. Hence
    \begin{align*}
        g_{\M}(t) &= g_{\U_{k,n}}(t) - \sum_{h} \uplambda_h \left( g_{\LL_{k-1,k,h,n}}(t) + g_{\U_{k-1,h}\oplus \U_{1,n-h}}(t)\right)\\
        &=g_{\U_{k,n}}(t) - \sum_{h} \uplambda_h \left( g_{\U_{k,h+1}}(t) + t\cdot g_{\U_{k-1,h}}(t)\right),
    \end{align*}
    where in the last step we used the properties listed in Proposition~\ref{prop:properties-of-speyer} along with the fact that $g_{\U_{1,h+1}}(t) = t$.
\end{proof}

\begin{example}
    Consider the complete graph on $4$ vertices, and its induced matroid, $\mathsf{K}_4$. As pointed out in \cite[p.~886]{speyer}, the base polytope of this matroid cannot be decomposed into base polytopes of series-parallel matroids. 
    The matroid $\mathsf{K}_4$ has rank $3$, cardinality $6$, is paving, and has exactly $4$ stressed hyperplanes, all of size $3$. In particular, Theorem~\ref{teo:speyer-for-paving} tells us
    \begin{align*}
        g_{\mathsf{K}_4}(t) &= g_{\U_{3,6}}(t) - 4\left(g_{\U_{3,4}}(t) + t\cdot g_{\U_{2,3}}(t)\right)\\
        &= \left(t^3+6t^2+6t\right) - 4 \cdot(t + t^2)\\
        &= t^3 + 2t^2 + 2t,
    \end{align*}
    which coincides with a calculation of Speyer which can be found in \cite[Figure 1]{speyer}.
\end{example}

\subsubsection{Towards Speyer's conjectures}

A natural question when looking at a subdivision of a matroid polytope is the number of internal faces of each dimension. For example, when looking at Figure~\ref{fig:oct_subdivision} we see a subdivision of $\U_{2,4}$ that has two internal faces of dimension $3$ and one of dimension $2$. A major open problem in the theory of matroid subdivisions is the so-called \emph{$f$-vector conjecture for tropical linear spaces}, posed by Speyer in \cite{speyer-conjecture}. That conjecture asserts the validity of certain upper bounds on the number of faces of each dimension.

Let us discuss a different (but closely related) conjecture by Speyer, that is in fact stronger than the $f$-vector conjecture.

\begin{conj}[\cite{speyer}]\label{conj:speyer2}
    For every matroid $\M$, the polynomial $g_{\M}(t)$ has positive coefficients.
\end{conj}

The above conjecture is known to hold true for matroids representable over a field of characteric~$0$. Such result is the crucial ingredient in proving the $f$-vector conjecture for subdivisions whose matroids are representable over a field of characteristic $0$. The proof in that setting relies on the validity of a deep result in algebraic geometry known as the Kawamata--Viehweg vanishing theorem. Using our Theorem~\ref{teo:speyer-for-paving} it is possible to show that for the class of sparse paving matroids this non-negativity property holds as well.

\begin{teo}
    If $\M$ is a sparse paving matroid then $g_{\M}(t)$ has positive coefficients.
\end{teo}

\begin{proof}
    Since the polynomial $g_\M$ is multiplicative, we may assume that the matroid $\M$ is connected. Furthermore, we use our usual notation, $k$ for the rank, $n$ for the cardinality of the ground set, and $\uplambda$ for the number of circuit-hyperplanes of $\M$. Since the $g$-polynomial is invariant under taking duals, and the dual of a sparse paving matroid is yet another sparse paving matroid, we might assume without loss of generality that $n-k \leq k$. 
    Moreover, since all matroids of rank or corank strictly smaller than $3$ are representable over a field of characteristic $0$, for which the $g$-polynomial is known to have positive coefficients, we are left with the cases $3\leq k \leq n-3$. Throughout the proof we thus assume that $3\leq n-k \leq k \leq n - 3$. With these assumptions the formula of Theorem~\ref{teo:speyer-for-paving} reduces to
    \[ g_{\M}(t) = g_{\U_{k,n}}(t) - \uplambda \left( g_{\U_{k,k+1}}(t) + t \cdot g_{\U_{k-1,k}}(t)\right) = g_{\U_{k,n}}(t) - \uplambda (t + t^2)
    \]
    as $g_{\U_{k,k+1}}(t) = g_{\U_{k-1,k}}(t)=t$.
    Thus, the only two relevant terms are the linear and the quadratic one. By Corollary~\ref{coro:bound-circuit-hyperplanes-sparse-paving} we have that $\uplambda \leq \frac{1}{k+1} \binom{n}{k}$. Further, we have
    \begin{align*}
        [t^1]g_{\U_{k,n}}(t) = \binom{n-2}{k-1} \quad \text{ and } \quad
        [t^2]g_{\U_{k,n}}(t) = \binom{n-3}{k-2}(n-k-1).
    \end{align*}
    Therefore, it suffices to prove the following two inequalities, under the assumption that $n-k\leq k$,
    \begin{equation} \label{eq:speyer-both}
    \frac{1}{k+1}\binom{n}{k} \leq \binom{n-2}{k-1} \enspace \text{ and } \enspace \frac{1}{k+1}\binom{n}{k} \leq \binom{n-3}{k-2}(n-k-1).
    \end{equation}
    The first of these two inequalities is equivalent to showing that
    \begin{equation}\label{eq:speyer-linear}
    \frac{1}{k+1} \leq \frac{\binom{n-2}{k-1}}{\binom{n}{k}} = \frac{k(n-k)}{n(n-1)}\enspace \text{ or, more succinctly,}\enspace n(n-1) \leq k(k+1)(n-k).\end{equation}
    Let us analyze some cases 
    \begin{itemize}
        \item If $k \leq n - 4$. Since we assumed that $n-k\leq k$, we have $n\leq 2k$. In particular, we get the chain of inequalities:
        \[n(n-1) \leq 2k(2k-1)  = 4k^2 - 2k \leq 4k(k+1) \leq k(k+1)(n-k).\]
        Which proves inequality \eqref{eq:speyer-linear} in this case.
        \item In the remaining case $k=n-3$ then the inequality in \eqref{eq:speyer-linear} becomes a (quadratic) inequality that only depends on $n$. The assumption $3\leq n-k \leq k\leq n$ implies that $n = k + (n-k) \geq 6$, and the inequality to prove \eqref{eq:speyer-linear} becomes $n(n-1) \leq (n-3)(n-2)3$ which is trivially true for $n\geq 6$.
    \end{itemize} 
    Now, let us turn our attention the second of the two inequalities in \eqref{eq:speyer-both}. The inequality to prove is a consequence of the easier fact that in general the quadratic coefficient of $g_{\U_{k,n}}(t)$ is larger than the linear one (this is not true when $k=2$ or $k=n-2$, but recall that we are assuming $3\leq n-k\leq k\leq n -3$). We want to prove the inequality
        \begin{equation} \label{eq:speyer-second}
        \binom{n-2}{k-1} \leq \binom{n-3}{k-2}(n-k-1).\end{equation}
    This is equivalent to $n-2 \leq (n-k-1)(k-1)$. Observe that since we are working under the assumption $3\leq n-k\leq k \leq n-3$, we obtain
    \[ n-2 \leq 2k - 2 = 2(k-1) \leq (n-k-1)(k-1). \]
    This proves \eqref{eq:speyer-second}, and completes the proof.
\end{proof}

\subsection{Denham's polynomials and Spectrum polynomials}\label{sec:spectrum}

In \cite{kook-reiner-stanton} Kook, Reiner and Stanton studied the combinatorial Laplacian of the independence complex of a matroid, as introduced by Friedman in \cite{friedman} for simplicial complexes. Kook, Reiner, and Stanton introduced an invariant called the \emph{spectrum polynomial} of a matroid $\M$, which carries the information of the spectra of the Laplacians for the independence complex of $\M$.

\begin{defi}
    The \emph{spectrum polynomial} of a matroid $\M$ is the bivariate polynomial $\operatorname{Spec}_{\M}(t,q) \in \mathbb{Z}[t,q]$ given by
    \[ \operatorname{Spec}_{\M}(t,q) = \sum_{\substack{F_1, F_2\in \mathcal{L}(\M)\\ F_1\subseteq F_2}} |\tilde{\chi}(\M|_{F_1})|\cdot |\mu_{\mathcal{L}(\M)}(F_1,F_2)|\cdot t^{\rk(F_2)}q^{|F_1|}\enspace . \]
\end{defi}

In the prior definition, $\tilde{\chi}(\M)$ stands for the reduced Euler characteristic of the independence complex of $\M$. It is not difficult to show that $|\tilde{\chi}(\M)|=T_{\M}(0,1)$. The reader should not confuse this with the characteristic polynomial $\chi_{\M}(t)$. On the other hand $\mu_{\mathcal{L}(\M)}(F_1,F_2)$ is just the classical M\"obius function on the lattice of flats of the matroid $\M$.

An important property of the spectrum polynomial is that it determines the characteristic polynomial. A long-standing question, posed by Kook, Reiner, and Stanton \cite[Question~2]{kook-reiner-stanton} is whether the spectrum polynomial determines the Tutte polynomial. In this subsection we will answer negatively to their question by constructing a pair of matroids having the same spectrum polynomial but different Tutte polynomials. There are two key ingredients that will be necessary to this end: first, the valuativeness of the spectrum polynomial proved in Corollary~\ref{coro:spectrum-valuative}; and second, the explicit formulas for the spectrum polynomial of cuspidal matroids of Lemma~\ref{lemma:spectrum-cuspidal}, which lead to a fast formula for split matroids in Theorem~\ref{thm:spectrum-split}. These ingredients allowed us to speed up the computation of spectrum polynomials of various matroids, enabling us to capture an example with the desired properties by an exhaustive search.

\subsubsection{Spectrum polynomials and Denham's polynomials are valuations}

In \cite{denham} Denham introduced a more general polynomial that has both the Tutte polynomial and the spectrum polynomial as specializations. 

\begin{defi}
    The \emph{Denham polynomial} of a matroid $\M$ on $E$ is a multivariate polynomial $\Phi_{\M} \in \mathbb{Z}[x,y,\mathbf{b}]$ where $\mathbf{b}=(b_i)_{i\in E}$. It is defined by
        \[ \Phi_{\M}(x,y,\mathbf{b}) = \sum_{\text{$F$ cyc.\ flat}} (-1)^{\rk(\M) - |F|} \chi_{\M/F}(-x)\cdot \chi_{(\M|_F)^*}(-y) \cdot b_F,\]
    where $b_F := \prod_{i\in F} b_i$.
\end{defi}

Observe that $\Phi$ is \emph{not} an invariant because its defining formula depends on how the ground set is labelled. As mentioned earlier, Denham's polynomial can be specialized so to obtain the spectrum polynomial.

\begin{prop}[\cite{denham}]
    The spectrum polynomial can be computed by:
    \begin{equation}\label{eq:formula-spectrum}
        \operatorname{Spec}_{\M}(t,q) = t^{\rk(\M)} \Phi_{\M}(t^{-1},0,q,\ldots,q).
    \end{equation}
\end{prop}

In particular, instead of proving the valuativeness of the spectrum polynomial, let us focus on the valuativeness of Denham's polynomial. Again, we rely on the convolution theorem by Ardila and Sanchez.

\begin{teo}
    Consider the abelian group $\mathbb{Z}[x,y,\mathbf{b}]$ where $\mathbf{b}=(b_1,b_2,\ldots)$ is an infinite vector of variables. Denham's polynomial $\Phi:\mathfrak{M}\to \mathbb{Z}[x,y,\mathbf{b}]$ is a valuation.
\end{teo}

\begin{proof}
    Consider the maps $f_1, f_2:\mathfrak{M}\to \mathbb{Z}[x,y,\mathbf{b}]$ defined by
    \begin{align*}
        f_1(\M) = T_{\M}(0,1+y)\prod_{i\in E_{\M}} b_i\quad \text{ and } \quad
        f_2(\M) = T_{\M}(1+x, 0)
    \end{align*}
    where $E_\M$ is the ground set of the matroid $\M$.
    The Tutte polynomial is a strong valuation, hence the map $f_2$ is a strong valuation as well. 
    Furthermore, all matroids in a matroidal subdivision have a common ground set, thus $f_1$ is a strong valuation too. 
    By Theorem~\ref{thm:convolutions1}, it follows that the convolution $f:\mathfrak{M}\to \mathbb{Z}[x,y,\mathbf{b}]$ defined by
    \[ f(\M) = (f_1\star f_2)(\M)= \sum_{A\subseteq E} f_1(\M|_A)\cdot f_2(\M/A)\]
    is a strong valuation. We have:
    \begin{align*}
        f(\M) &= \sum_{A\subseteq E} f_1(\M|_A)\cdot f_2(\M/A)\\
        &= \sum_{A\subseteq E} T_{\M|_A}(0,1+y)\cdot T_{\M/A}(1+x,0)\cdot b_A\\
        &= \sum_{A\subseteq E} \left((-1)^{\rk((\M|_A)^*)} \chi_{(\M|_A)^*}(-y)\right)\cdot\left( 
        (-1)^{\rk(\M/A)}\chi_{\M/A}(-x)\right)\cdot b_A\\
        &=\sum_{\text{$F$ cyc. flat}} \left((-1)^{|F|-\rk(F)} \chi_{(\M|_F)^*}(-y)\right)\cdot\left( 
        (-1)^{\rk(\M)-\rk(F)}\chi_{\M/F}(-x)\right)\cdot b_F\\
        &=\Phi_{\M}(x,y,\mathbf{b}),
    \end{align*}
    where we used that $\chi_{(\M|_A)^*}(-y)=0$ if $\M|_A$ has coloops and $\chi_{\M/A}(-x)=0$ if $A$ is not a flat (cf. Remark~\ref{rem:convolutions-loops}). The proof is complete.
\end{proof}

\begin{coro}\label{coro:spectrum-valuative}
    The assignment $\mathfrak{M}\to \mathbb{Z}[t,q]$ given by $\M\longmapsto \operatorname{Spec}_{\M}(t,q)$ is a valuative invariant.
\end{coro}

\subsubsection{Spectrum polynomials for elementary split matroids}

In order to use Theorem~\ref{thm:main}, we require to compute first the spectrum polynomial of a uniform matroid. Since this computation has not been carried out explicitly in the literature, we include a short proof below.

\begin{prop} Let $0\leq k\leq n$.
    The spectrum polynomial of the uniform matroid $\U_{k,n}$ is given by
    \[ \operatorname{Spec}_{\U_{k,n}}(t,q) = \binom{n-1}{k}\, t^k (q^n-1) + \sum_{i=0}^k \binom{n}{i}\, t^i.\]
\end{prop}

\begin{proof}
    The only cyclic flats of the uniform matroid $\U_{k,n}$ are the empty set and its ground set whenever $0<k<n$, while the only cyclic flat of $\U_{0,n}$ is the ground set and of $\U_{n,n}$ is the empty set.
    Hence, combining the formula of equation \eqref{eq:formula-spectrum} with $\sum_{j=0}^k (-1)^j \binom{n}{j}=(-1)^k\binom{n-1}{k}$ and the following well-known formula for the characteristic polynomial
    \[\chi_{\U_{k,n}}(t) = \sum_{j=0}^{k-1} (-1)^j \binom{n}{j}\big(t^{k-j}-1\big)\] if $k>0$ and $\chi_{\U_{0,n}}= 1$ yields the result for all values of $k$.
\end{proof}

Kook, Reiner, and Stanton showed that the spectrum polynomial is multiplicative, i.e., 
\[\operatorname{Spec}_{\M_1\oplus \M_2}(t,q) = \operatorname{Spec}_{\M_1}(t,q)\cdot \operatorname{Spec}_{\M_2}(t,q).\]
In particular, if we want to get handy formulas for the spectrum polynomial of arbitrary elementary split matroids we should now focus on explicit formulas for all cuspidal matroids. Recall that in Section \ref{sec:tutte} we already stated explicit expressions for the Tutte polynomial of all such matroids.

\begin{lemma}\label{lemma:spectrum-cuspidal}
    The spectrum polynomial of the matroid $\LL_{r,k,h,n}$ is given by
    \[ \operatorname{Spec}_{\LL_{r,k,h,n}}(t,q) = t^k T_{\LL_{r,k,h,n}}(1+\tfrac{1}{t},0) + \binom{n-h-1}{k-r} t^kq^{n-h} T_{\U_{r,h}}(1+\tfrac{1}{t},0) + t^kq^n T_{\LL_{r,k,h,n}}(0,1).\]
\end{lemma}

\begin{proof}
    By equation \eqref{eq:formula-spectrum}, we know how to compute the spectrum polynomial in terms of the Denham polynomial, which is defined as a sum over the cyclic flats of the matroid. Recall that cuspidal matroids have only one proper cyclic flat. Hence, the following sum,
    \begin{align*}
        \operatorname{Spec}_{\LL_{r,k,h,n}}(t,q) = t^k \sum_{\text{$F$ cyc.\ flat}} (-1)^{k - |F|} \chi_{\LL_{r,k,h,n}/F}(-1/t)\cdot \chi_{(\M|_F)^*}(0) q^{|F|},
    \end{align*}
    consists of only three terms, each corresponding to the cyclic flats $\varnothing$, $\U_{k-r,n-h}$, and $E$. This explains the three terms appearing in the formula of the statement. Furthermore, observe that we can express everything in terms of the Tutte polynomial because $\chi_{\M}(t) = (-1)^{\rk(\M)} T_{\M}(1-t,0)$.
\end{proof}

Hence, Theorem~\ref{thm:main} yields the following explicit formula for the spectrum polynomial of an arbitrary elementary split matroid.

\begin{teo}\label{thm:spectrum-split}
    Let $\M$ be an elementary split matroid of rank $k$ and cardinality $n$. Then,
    \[ \operatorname{Spec}_{\M}(t,q) = \operatorname{Spec}_{\U_{k,n}}(t,q) - \sum_{r,h} \uplambda_{r,h} \left(\operatorname{Spec}_{\LL_{r,k,h,n}}(t,q) - \operatorname{Spec}_{\U_{r,h}}(t,q)\cdot \operatorname{Spec}_{\U_{k-r,n-h}}(t,q)\right),  \]
    where $\uplambda_{r,h}$ denotes the number of stressed cyclic flats of $\M$ of size $h$ and rank $r$.
\end{teo}

We used the preceding two formulas to do fast computations of spectrum polynomials for several matroids. One of the main motivations for that was the aforementioned question raised by Kook, Reiner, and Stanton in \cite[Question~2]{kook-reiner-stanton}. Although Denham \cite{denham} and Kook \cite{kook} were able to prove various recursive formulas for spectrum polynomials, they were not able to produce examples showing that the spectrum polynomial does not specialize to the Tutte polynomial. Our examples also show that it cannot be used to decide whether a matroid is or not elementary split --- observe that, similarly, Tutte polynomials cannot detect this property by Example~\ref{ex:tutte-split}.

\begin{example}\label{ex:same_spec}
    In Figure \ref{fig:counterexample-spectrum} we depict two matroids $\M_1$ and $\M_2$, of rank $3$ on $9$ elements. The matroid $\M_1$, on the left, is not elementary split. The matroid $\M_2$, on the right, is elementary split. Their spectrum polynomials $\operatorname{Spec}_{\M_1}(t,q)$ and $\operatorname{Spec}_{\M_2}(t,q)$ coincide and are equal to 
    \[
         30q^9t^3 + 6q^5t^3 + 6q^5t^2 + 10q^3t^3 + 12q^3t^2 + 2q^3t + 9t^3 + 15t^2 + 7t + 1,
    \]
    whereas their Tutte polynomials differ:
    \begin{align*}
        T_{\M_1}(x,y) &= y^{6} + 3 y^{5} + x^{2} y^{2} + 2 x y^{3} + 6 y^{4} + x^{3} + x^{2} y + 4 x y^{2} + 8 y^{3} + 4 x^{2} + 8 x y + 8 y^{2} + 4 x + 4 y,\\
        T_{\M_2}(x,y) &= y^{6} + 3 y^{5} + x^{2} y^{2} + x y^{3} + 6 y^{4} + x^{3} + x^{2} y + 6 x y^{2} + 9 y^{3} + 4 x^{2} + 7 x y + 7 y^{2} + 4 x + 4 y.
    \end{align*}
    This shows that the spectra of a matroid complex does not determine the Tutte polynomial nor detects the property of being elementary split.
    \begin{figure}[ht]
    \centering
	\begin{tikzpicture}  
	[scale=0.7,auto=center,every node/.style={circle,scale=0.8, fill=black, inner sep=2.7pt}] 
	\tikzstyle{edges} = [thick];
	
	
	\node[label=left:${2}$] (a1) at (210:2) {};
	\node[] (a1p) at (-.25,-1) {};
	\node[label=right:$5$] (a2) at (30:1)  {};
	\node[label=above:$4$] (a3) at (90:2)  {};  
	\node[label=left:$6$] (a4) at (150:1) {};
	\node[] (a5) at (0.25,-1)  {};  
	\node[label=below:${1,8,9}$] (a6) at (270:1)  {};
	\node[label=right:$3$] (a7) at (330:2) {};
	\node[label=left:$7$] (a0) at (0:0) {};
	
	\draw[edges] (a1) -- (a6);  
	\draw[edges] (a3) -- (a2);  
	\draw[edges] (a2) -- (a7);
	\draw[edges] (a7) -- (a6);
	\draw[edges] (a1) -- (a3);
	\draw[edges] (a1) -- (a2);
	\draw[edges] (a4) -- (a7);
	\draw[edges] (a3) -- (a6);
	\end{tikzpicture} \qquad\qquad\qquad	\begin{tikzpicture}  
	[scale=0.5,auto=center,every node/.style={circle,scale=0.8, fill=black, inner sep=2.7pt}] 
	\tikzstyle{edges} = [thick];
	
	\node[label=left:$1$] (a1) at (0,0) {};  
	\node[label=above:$2$] (a2) at (2/2,1/2)  {};
	\node[label=above:$3$] (a3) at (4/2,2/2)  {};
	\node[label=above:$4$] (a4) at (6/2,3/2) {};
	\node[label=above:$5$] (a5) at (8/2,4/2) {};
	\node[label=below:$6$] (a6) at (4/2,-2/2)  {};  
	\node[] (a7) at (7.35/2,-4/2)  {};    
	\node[label=below:${7,8,9}$] (a8) at (8.0/2,-4/2)  {};
	\node[] (a9) at (8.65/2,-4/2)  {};
	\draw[edges] (a1) -- (a2); 
	\draw[edges] (a2) -- (a3);  
	\draw[edges] (a3) -- (a5);  
	
	\end{tikzpicture}  \caption{The matroids $\M_1$ and $\M_2$ of Example~\ref{ex:same_spec}.}\label{fig:counterexample-spectrum}
\end{figure}
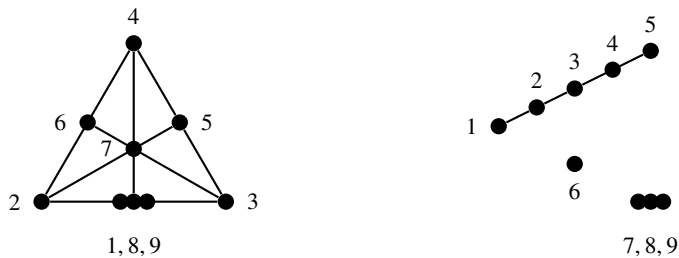
\end{example}

\subsection{Final considerations and more valuative invariants}\label{sec:last-section}
Along this paper we have addressed several invariants coming from various frameworks. We showed that the class of cuspidal matroids is often tractable enough to do computations, which leads to formulas for all elementary split matroids. Although this class is ``large'', there is still a considerable gap between understanding well an invariant for split matroid versus understanding it well for all matroids. Nonetheless, we believe that for the sake of feasible computations, testing conjectures, and producing data, this family of matroids stays at a sufficiently good and reasonable level of generality.

Besides the questions we raised above on the valuations we discussed, it is of interest to investigate many further invariants and to find formulas for them on cuspidal and general elementary split matroids.
Examples for this task are the following valuative invariants.

\begin{itemize}
    \item The quasisymmetric function defined by Billera, Jia and Reiner in \cite{billera-jia-reiner}.
    \item The volume polynomial of the Chow ring of a matroid, and the shifted-rank-volume invariant, both introduced by Eur \cite{eur-volume}.
    \item The motivic zeta functions introduced by Jensen, Kutler and Usatine \cite{jensen-kutler-usatine-motivic}. In particular, the topological zeta function studied by van der Veer in \cite{van-der-veer}.
    \item The (coarse) flag Hilbert--Poincar\'e series studied in \cite{kuhne-maglione} by K\"uhne and Maglione.
    \item The Chern--Schwarz--MacPherson cycles of a matroid, as presented in the work of L\'opez~de~Medrano, Rinc\'on and Shaw \cite{lopezdemedrano-rincon-shaw}.
    
\end{itemize}

Additionally, there are other recent invariants for which the problem of proving their valuativeness and trying to compute them for elementary split matroids would be of interest as well.

\begin{itemize}
    \item The universal cross ratio invariant, defined in \cite{baker-lorscheid}.
    \item The Gromov--Witten invariants of matroids introduced by Ranganathan and Usatine in \cite{ranganathan-usatine}.
    \item Invariants of the matroidal Schur algebras of Braden and Mautner \cite{braden-mautner}.
    \item Poset invariants of $\mathcal{L}(\M)$ such as the order polynomial or the zeta polynomial (see \cite[Section 3.12]{stanley-ec1}).
\end{itemize}

\appendix

\section{Series-parallel matroids span the valuative group of matroids}\label{appendix-series-parallel}
\noindent A particularly relevant class of matroids that behave well in terms of subdivisions are the lattice path matroids. Since the formal statements justifying the preceding assertion are better illustrated via an example, let us present the following.

\begin{example}\label{ex:split-lpm}
    Let $\M=\M[L,U]$ be the lattice path matroid of rank $4$ and cardinality $8$ defined by $L = \text{EEENNENN}$ and $U = \text{NNEENENE}$. Now consider its grid representation, depicted on the left of Figure \ref{fig:split-lpms}. We observe that there are two lattice points lying \emph{strictly between} $L$ and $U$; they are, respectively, $(1,1)$ and $(2,1)$. Let us choose one of these two points, say $p=(1,1)$. 
    There are two distinguished paths of our interest passing through $p$. On the one hand, we consider the ``lowest path'' $L_p$ between $L$ and $U$ that passes through $p$, depicted in purple on the right part of Figure \ref{fig:split-lpms}; and on the other hand we consider the ``highest path'' $U_p$  between $L$ and $U$ that passes through $p$, depicted in orange. 
    \begin{figure}[ht]
        \centering
        \begin{tikzpicture}[scale=0.55, line width=.9pt]

        \draw[line width=2.3pt,blue,line cap=round] (0,0)--(3,0) -- (3,2) -- (4,2) -- (4,4);
        \draw[line width=2.3pt,red,line cap=round] (0,0)--(0,2)--(2,2)--(2,2)--(2,3)--(3,3)--(3,4)--(4,4);
        \draw (0,0) grid (4,4);
        \end{tikzpicture}\qquad\qquad\qquad\qquad \begin{tikzpicture}[scale=0.55, line width=.9pt]

        \draw[line width=2.3pt,violet,line cap=round] (0,0)--(1,0) -- (1,1) -- (3,1) -- (3,2) -- (4,2) -- (4,4);
        \draw[line width=2.3pt,orange,line cap=round] (0,0)--(0,1)--(1,1)--(1,2)--(2,2)--(2,3)--(3,3)--(3,4)--(4,4);
        \draw (0,0) grid (4,4);
        \end{tikzpicture}  \caption{A lattice path matroid and two distinguished paths.}\label{fig:split-lpms}
    \end{figure}
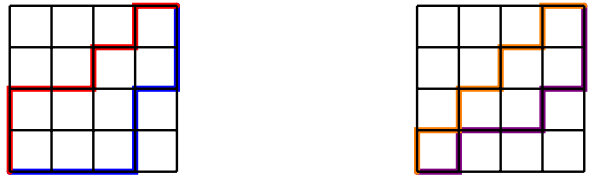
    In other words $L_p = \text{ENEENENN}$ and $U_p = \text{NENENENE}$. We construct the following three lattice path matroids $\M_1 = \M[L_p,U]$, $\M_2 = \M[L,U_p]$ and $\M_3 = \M[L_p,U_p]$. It is straightforward to verify that the polytope $\mathscr{P}(\M)$ admits a hyperplane split, which amounts to a subdivision $\mathcal{S}$ whose two maximal cells are $\mathscr{P}(\M_1)$ and $\mathscr{P}(\M_2)$. These two cells intersect along a common face, which is precisely $\mathscr{P}(\M_3)$.
\end{example}

The rigorous statement that justifies the validity of the previous example goes back to Chatelain and Ram\'irez-Alfons\'in \cite[Theorem~1 and Corollary~4]{ChatelainAlfonsin}. 

\begin{prop}\label{prop:subdiv-lpm}
    Let $\M=\M[L,U]$ be a lattice path matroid such that its representation using the rectangular grid has a lattice point $p$ lying strictly between $L$ and $U$. Consider $L_p$ and $U_p$ the minimal and maximal paths lying between $L$ and $U$ that pass through $p$. 
    Then there exists a matroid subdivision of the polytope $\mathscr{P}(\M)$ whose maximal cells are the two base polytopes of the matroids $\M[L_p, U]$ and $\M[L,U_p]$ which intersect along a common face that corresponds to the matroid $\M[L_p,U_p]$.
\end{prop}

Observe that the three lattice path matroids $\M[L_p,U]$, $\M[L,U_p]$ and $\M[L_p,U_p]$ arising from a decomposition as in the above result can be further subdivided, as long as they still possess at least one interior lattice point in their grid representations. Connected lattice path matroids with at least two elements which \emph{do not} possess such interior lattice points are named \emph{border strips} in \cite[Section 4.3]{Bidkhori} and \emph{snake matroids} in \cite{knauer-martinez-ramirez0}; we  stick to the latter terminology. 

\begin{example}
    Continuing Example \ref{ex:split-lpm}, one can see that the matroids $\M_1$ and $\M_3$ do not admit interior lattice points in their grid representations, but $\M_2$ does. By applying the same argument on $\M_2$, one deduces that the five lattice path matroids depicted in Figure \ref{fig:snake-decomposition} are the interior faces in a matroidal subdivision of $\mathscr{P}(\M)$. The three lattice path matroids of the first row yield the maximal faces of that subdivision because they are connected, i.e., they are snakes. The two matroids of the second row correspond to two interior faces, which come from matroids that are direct sums of exactly two snakes. Notice that Proposition~\ref{prop:subdiv-lpm} cannot be further applied in any of these five matroids, as their grid representations do not possess interior lattice points.
    \begin{figure}[ht]
        \centering
        \begin{tikzpicture}[scale=0.50, line width=.9pt]

        \draw[line width=2.3pt,violet,line cap=round] (0,0)--(1,0) -- (1,1)-- (2,1) -- (3,1) -- (3,2) -- (4,2) -- (4,4);
        \draw[line width=2.3pt,orange,line cap=round] (0,0)--(0,2)--(2,2)--(2,2)--(2,3)--(3,3)--(3,4)--(4,4);
        \draw (0,0) grid (4,4);
        \end{tikzpicture} \qquad\qquad\qquad\qquad \begin{tikzpicture}[scale=0.50, line width=.9pt]

        \draw[line width=2.3pt,violet,line cap=round] (0,0)--(2,0) -- (2,1) -- (3,1) -- (3,2) -- (4,2) -- (4,4);
        \draw[line width=2.3pt,orange,line cap=round] (0,0)-- (0,1) -- (1,1)--(1,2)--(2,2)--(2,3)--(3,3)--(3,4)--(4,4);
        \draw (0,0) grid (4,4);
        \end{tikzpicture} \qquad\qquad\qquad\qquad \begin{tikzpicture}[scale=0.50, line width=.9pt]

        \draw[line width=2.3pt,violet,line cap=round] (0,0)--(3,0) -- (3,2) -- (4,2) -- (4,4);
        \draw[line width=2.3pt,orange,line cap=round] (0,0)--(0,1)--(2,1)--(2,2)--(2,3)--(3,3)--(3,4)--(4,4);
        \draw (0,0) grid (4,4);
        \end{tikzpicture} \qquad\qquad\qquad\qquad \begin{tikzpicture}[scale=0.50, line width=.9pt]

        \draw[line width=2.3pt,violet,line cap=round] (0,0)--(1,0) -- (1,1) -- (3,1) -- (3,2) -- (4,2) -- (4,4);
        \draw[line width=2.3pt,orange,line cap=round] (0,0)--(0,1)--(1,1)--(1,2)--(2,2)--(2,3)--(3,3)--(3,4)--(4,4);
        \draw (0,0) grid (4,4);
        \end{tikzpicture}   \qquad\qquad\qquad\qquad \begin{tikzpicture}[scale=0.50, line width=.9pt]

        \draw[line width=2.3pt,violet,line cap=round] (0,0)--(2,0) -- (2,1) -- (3,1) -- (3,2) -- (4,2) -- (4,4);
        \draw[line width=2.3pt,orange,line cap=round] (0,0)--(0,1)--(2,1)--(2,2)--(2,3)--(3,3)--(3,4)--(4,4);
        \draw (0,0) grid (4,4);
        \end{tikzpicture}  \caption{The interior faces in a more refined subdivision of $\M$.}\label{fig:snake-decomposition}
    \end{figure}
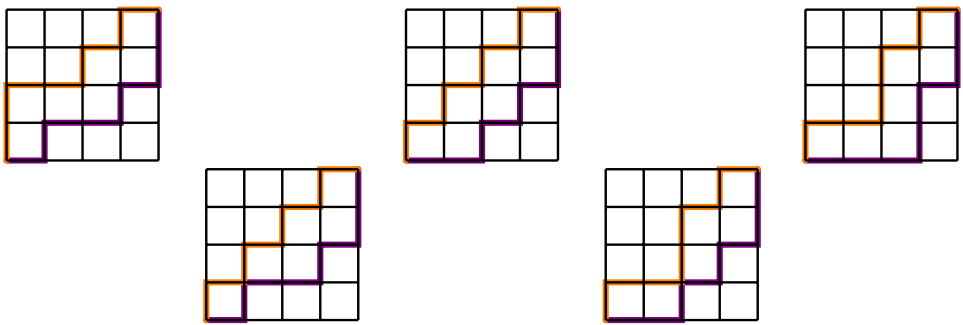
\end{example}

Generalizing the preceding construction, i.e., iterating the result in Proposition~\ref{prop:subdiv-lpm} on the smaller pieces in an arbitrary lattice path matroid, one obtains the following general result.

\begin{cor}[{\cite[Lemma~4.3.5]{Bidkhori}}]\label{cor:lpm-seriesparallel}
    Let $\M=\M[L,U]$ be a connected lattice path matroid with at least two elements. Then $\M$ is either itself a snake matroid, or there exists a subdivision $\mathcal{S}$ of $\mathscr{P}(\M)$ such that the polytopes in $\mathscr{S}^{\operatorname{int}}$ correspond to matroids that are isomorphic to direct sums of snakes.
\end{cor}

A property of snake matroids, stated for example in \cite[Theorem~2.2]{knauer-martinez-ramirez0}, is that they are graphic matroids. Moreover, the characterization of snake matroids by Knauer, Mart\'inez-Sandoval, and Ram\'irez-Alfons\'in as a graphic matroids leads to the following statement.

\begin{prop}\label{prop:snake}
    If a matroid $\M$ is isomorphic to a connected snake matroid, then it is a series-parallel matroid.
\end{prop}

Since there are several variations of the notion of series-parallel matroid, let us clarify that we are using the convention that series-parallel matroids are connected. Following Oxley \cite{oxley}, the class of series-parallel matroids are the matroids obtained by performing a sequence of series and parallel extensions to the matroid $\U_{1,2}$. By convention, we will additionally assume that $\U_{0,1}$ and $\U_{1,1}$ are series-parallel matroids.

We point out that not all series-parallel matroids are snakes, as can be revealed by the graphic matroid induced by a cycle of length $3$ in which each edge has a parallel copy. This yields a matroid of size $6$ and rank $2$ that is series-parallel but is not even isomorphic to a lattice path matroid. An immediate consequence of the preceding proposition is the following.

\begin{coro}
    Let $\M$ be a lattice path matroid. Then $\mathscr{P}(\M)$ admits a matroid subdivision in which all the interior faces are base polytopes of matroids which are direct sums of series-parallel matroids.
\end{coro}

By combining the preceding property together with Corollary~\ref{cor:lpm-seriesparallel} and Theorem~\ref{teo:schubert-is-lpm}, we obtain the following result.

\begin{teo}\label{thm:series-parallel-span}
    Let $\M$ be a matroid. The indicator function $\amathbb{1}_{\mathscr{P}(\M)}$ can be written as an integer combination of indicator functions of matroids that are direct sums of series-parallel matroids. 
\end{teo}

\begin{proof}
    By Theorem~\ref{thm:schubert-basis}, it suffices to show that if $\M$ is a Schubert matroid, then the statement holds. In such a case, Theorem~\ref{teo:schubert-is-lpm} guarantees that $\M$ is isomorphic to a certain lattice path matroid, and via relabelling the ground set, it will be enough to prove this statement for all lattice path matroids. 
    
    If the grid representation of $\M$ does not possess interior lattice points, then $\M$ is a direct sum of snake matroids, loops, and coloops. By Proposition~\ref{prop:snake} we know that a snake matroid is a series-parallel matroid. On the other hand, if there is some interior lattice point in the grid representation of $\M$, then by Proposition~\ref{prop:subdiv-lpm} we can find a subdivision of $\M$ which has three interior faces. Using the notation of that statement, it follows that:
        \[ \amathbb{1}_{\mathscr{P}(\M)} = \amathbb{1}_{\mathscr{P}(\M_1)} + \amathbb{1}_{\mathscr{P}(\M_2)} - \amathbb{1}_{\mathscr{P}(\M_3)}.\]
    Since all the $\M_i$, $i\in\{1,2,3\}$, are lattice path matroids that contain strictly less interior lattice points, by repeating this procedure inductively until we reach lattice path matroids without interior points (which are therefore direct sums of snakes) we may write $\amathbb{1}_{\mathscr{P}(\M)}$ as a signed integer combination of indicator functions of direct sums of series-parallel matroids.
\end{proof}

\begin{remark}
    As a consequence of Theorem~\ref{thm:series-parallel-span}, we have provided an answer to a problem posed by Billera, Jia and Reiner in \cite{billera-jia-reiner} (see Problem~8.6, and the paragraph thereafter). In particular, any matroid valuation is determined on all matroids only by specifying it on the direct sums of series-parallel matroids.
\end{remark}

\section{Valuations from Chow rings}\label{appendix-chow}
\noindent In this section we quickly review some of the basic properties of Chow rings of matroids, and recall the main result of Adiprasito, Huh, and Katz \cite{adiprasito-huh-katz} that these rings satisfy the K\"ahler package. We show that homomorphisms defined on the Chow ring of the permutohedral variety give rise to valuative invariants.

\subsection{Definition and K\"ahler package}

The notion of Chow ring of an atomic lattice was first introduced by Feichtner and Yuzvinsky \cite{feichtner-yuzvinsky}. Following Adiprasito, Huh, and Katz we present Chow rings of matroids by using the next definition.

\begin{defi}
    Let $\M$ be a loopless matroid. The Chow ring of $\M$ is the quotient
        \[
        \uCH(\M) = \mathbb{Q}[x_F : F\in\widehat{\mathcal{L}}(\M)]/{(I+J)},
        \]
    where $I = \left< x_{F_1} x_{F_2} \,:\, F_1,F_2 \in \widehat{\mathcal{L}}(\M) \text{ are incomparable}\right>$, $J = \left< \sum_{F\ni i} x_F - \sum_{F\ni j} x_F \,:\, i,j\in E\right>$, and 
    $\widehat{\mathcal{L}}(\M)=\mathcal{L}(\M)\smallsetminus \{\varnothing, E\}$ is the collection of proper non-empty flats of the matroid $\M$.
\end{defi}

If the loopless matroid $\M$ has rank $k$, then its Chow ring $\uCH(\M)$ is a graded ring, admitting a decomposition
    \begin{equation} \label{eq:grading-chow}
    \uCH(\M) = \bigoplus_{j=0}^{k-1} \uCH^j(\M).\end{equation}
Motivated from algebraic geometry, the elements of $\uCH^1(\M)$ are called the \emph{divisors}. Any divisor $\ell\in \uCH^1(\M)$ is a linear combination of the variables $x_F$ with $F\in\widehat{\mathcal{L}}(\M)$. In other words, 
    \[ \ell = \sum_{F\in \widehat{\mathcal{L}}(\M)} c_F\, x_F,\]
where $c_F\in \mathbb{Q}$ for each proper flat $F$. We say that $\ell$ is \emph{combinatorially ample} if the coefficients $c_F$ satisfy the \emph{strict submodularity property}. In other words, if the numbers $c_F$ are obtained as the restriction of a map $c:2^E\to \mathbb{Q}$ satisfying
    \[ c_{A_1} + c_{A_2} > c_{A_1\cap A_2} + c_{A_1\cup A_2},\]
for each pair of incomparable subsets $A_1,A_2\subseteq E$, and $c_{\varnothing} = c_E = 0$.

The top component $\uCH^{k-1}(\M)$ is a one-dimensional vector space. The \emph{degree map} is the (unique) map $\int_{\M} : \uCH^{k-1}(\M) \to \mathbb{Q}$ having the property that
\[\int_{\M} x_{F_1}\cdots x_{F_{k-1}} = 1\] 
for all maximal chains of flats $\varnothing\subsetneq F_1\subsetneq\cdots\subsetneq F_{k-1}\subsetneq E$.
For the sake of completeness, we include below the central result by Adiprasito, Huh and Katz.

\begin{teo}[{\cite[Theorem~1.4 \& Theorem~6.19]{adiprasito-huh-katz}}]\label{thm:kahler-package}
    Let $\M$ be a loopless matroid of rank $k$ and let $\ell\in \uCH^1(\M)$ be a combinatorially ample divisor. Then the following holds:
    \begin{enumerate}
        \item[{(\normalfont PD)}]\label{it:PD} For every $0\leq j\leq k-1$, the bilinear pairing $\uCH^j(\M) \times \uCH^{k-1-j}(\M)\to \mathbb{Q}$, defined by 
            \[ (\upeta,\upxi) \longmapsto \int_{\M} \upeta\,\upxi,\]
        is non-degenerate, i.e., the map $\uCH^j\to \operatorname{Hom}(\uCH^{k-1-j},\mathbb{Q})$ defined by $\upeta \mapsto \left(\upxi \mapsto \int_{\M} \upeta\,\upxi\right)$ is an isomorphism.
        \item[{(\normalfont HL)}]\label{it:HL} For every $0\leq j \leq \left\lfloor\frac{k-1}{2}\right\rfloor$, the map $\uCH^j(\M) \to \uCH^{k-1-j}(\M)$, defined by
            \[ \upxi \longmapsto \ell^{k-1-2j}\, \upxi,\]
        is an isomorphism.
        \item[{(\normalfont HR)}]\label{it:HR} For every $0\leq j \leq \left\lfloor\frac{k-1}{2}\right\rfloor$, the bilinear symmetric form $\uCH^j(\M)\times \uCH^j(\M) \to \mathbb{Q}$, defined by
            \[ (\upeta, \upxi) \longmapsto (-1)^j \int_{\M} \upeta\, \ell^{k-1-2j} \,\upxi, \]
        is positive definite when restricted to $\{\alpha \in \uCH^j(\M) : \ell^{k-2j} \alpha = 0\}$. 
    \end{enumerate}
\end{teo}

The first property is referred to as the \emph{Poincar\'e duality}, the second as the \emph{hard Lefschetz property} and the third one as the \emph{Hodge--Riemann relations}. The three of them constitute what is known as the \emph{K\"ahler package} for the Chow ring of $\M$.

\subsection{Valuations from Chow rings}

The Chow ring of a Boolean matroid, $\uCH(\U_{n,n})$, shows up as an example in toric geometry as it is the Chow ring of a permutohedral variety. More precisely, take $E=[n]$ and consider the \emph{permutohedron} $\Pi(E)$, defined as the $(n-1)$-dimensional polytope in $\mathbb{R}^E$ whose vertices are the permutations of the elements in $E$. 
The projective toric variety associated to $\Pi(E)$ is the \emph{permutohedral variety} that we denote by $X_E$. There is a canonical isomorphism $\uCH(\U_{n,n})\cong\A(X_E)$, where $\A(X_E)$ denotes the classical Chow ring defined in toric geometry, see Hampe \cite[Prop.~5.13]{hampe} (in that paper the ring $\A(X_E)$ appears as the ``intersection ring of matroids'', whereas he studies also the components $\uCH^{n-k}(\U_{n,n})$ of degree $n-k$ of $\uCH(\U_{n,n})$ which are generated by all loopless rank $k$ matroids on $E$).

Notice that all of the subsets of $E=[n]$ are flats of the Boolean matroid $\U_{n,n}$, and hence for each $F\subsetneq E$ we have a variable $x_F\in \uCH^1(\U_{n,n})$. 
In particular, when we deal with a second matroid $\M$ on the same ground set $E$, for every flat $F\in \widehat{\mathcal{L}}(\M)$ we often consider the variable $x_F$ in both of $\uCH(\M)$ and $\uCH(\U_{n,n})$; Lemma~\ref{lemma:ann_chow} below explains the precise connection between those two elements.

Let $\M$ and $\N$ be matroids on the same ground set $E$. We denote by $\M\vee\N$ the \emph{union} and by $\M\wedge\N$ the \emph{intersection} of these two matroids. That is 
$\M\vee\N$ is the matroid with independent sets 
    \[
    \mathscr{I}(\M\vee\N) = \left\{ I\cup J : I\in\mathscr{I}(\M) \text{ and } J\in\mathscr{I}(\N) \right\}
    \]
and the matroid $\M\wedge\N$ is equal to $(\M^*\vee\N^*)^*$.
Notice that the family of common independent sets
    \[
    \M\cap\N = \left\{ I\cap J : I\in\mathscr{I}(\M) \text{ and } J\in\mathscr{I}(\N) \right\}
    \]
is of interest in combinatorial optimization where the name ``intersection'' is often used for this family as well. However, this family is in general not the collection of independent sets of a matroid; see also \cite[Chapter 41]{SchrijverB}. 
Thus we will focus on the matroid $\M\wedge\N$ which can be seen as tropical intersection product of $\M$ and $\N$ (at the level of their Bergman fans) if the intersection has the correct rank.
Its independent sets are contained in $\M\cap\N$.

We begin by looking at the Chow ring $\uCH(\U_{n,n})$ and define
    \[
    x_E = -\sum_{\substack{F\ni i\\F\neq E}} x_F.
    \]
It is not difficult to see that this is in fact independent of the choice of $i\in E$ in the Chow ring. The following makes results of Hampe \cite{hampe} explicit and reformulates them in terms of the Chow ring of $\U_{n,n}$.

\begin{prop}\label{prop:can_chow_valuation}
    Let $E=[n]$. There is a unique valuative invariant $\Theta:\mfrak{M}_E\to \uCH(\U_{n,n})$ satisfying the following four properties:
    \begin{enumerate}[{\normalfont (i)}]
        \item\label{it:1} $\Theta(\U_{n,n})=1$.
        \item\label{it:0} $\Theta(\M) = 0$ whenever $\M$ has a loop.
        \item\label{it:h} $\Theta(\M) = -\sum_{F\supseteq C} x_F$ whenever $\M$ is of rank $n-1$, loopless, and its unique circuit is $C$.
        \item\label{it:prod} $\Theta(\M\wedge\N) = \Theta(\M)\cdot \Theta(\N)$ whenever $\M\wedge\N$ is loopless.
    \end{enumerate}
    Moreover, $\Theta(\M) \in \uCH^{n-k}(\U_{n,n})$ whenever $\M$ is loopless and of rank $k$.
\end{prop}


\begin{proof}
    Properties \eqref{it:1}, \eqref{it:0} and \eqref{it:h} define the map $\Theta$ for all matroids of corank $0$ and $1$.
    Property \eqref{it:prod} extends this map uniquely to loopless Schubert matroids as described in \cite[Section 5.1]{hampe}.
    Now \cite[Theorem 6.3]{derksen-fink} allows to extend the map $\Theta$ to all matroids. Moreover this extension is unique, and hence the map $\Theta$ is unique and well-defined, and by construction satisfies all the properties of the statement.
\end{proof}

\begin{remark}\label{rem:chow}
    The images of the corank one matroids in property \eqref{it:h} are exactly  the simplicial generators of Backman, Eur and Simpson \cite{BackmanEurSimpson}, and 
    the image $\Theta(\M)$ is the Bergman class of $\M$ under the canonical identification of the space of Minkowski weights, the Chow ring of the permutohedral variety and the Chow ring $\uCH(\U_{n,n})$; see \cite[Theorem~2.3.1]{BackmanEurSimpson}.
\end{remark}

In light of the preceding result, every group homomorphism $f:\uCH(\U_{n,n})\to \mathrm{G}$ to an abelian group $\mathrm{G}$ gives rise to a valuation $f \circ \Theta:\mathfrak{M}_E\to \mathrm{G}$ for $E=[n]$. The following is a useful technical result of Backman, Eur and Simpson.

\begin{lemma}[{\cite[Proposition~4.1.3 \& Theorem~4.2.1]{BackmanEurSimpson}}]\label{lemma:ann_chow}
    Let $\M$ be a loopless matroid on $E=[n]$ of rank $k$, and let $\Theta$ be the map from Proposition~\ref{prop:can_chow_valuation}. Then
    \[
    \uCH(\M)\cong \uCH(\U_{n,n})/\ann(\Theta(\M))
    \]
    where $\ann(\Delta) = \{\upxi\in\uCH(\U_{n,n}) : \upxi\,\Delta=0\}$.
    Also, the degree map $\int_\M:\uCH^{k-1}(\M)\to\mathbb{Q}$ agrees with the induced degree map $\upxi+\ann(\Theta(\M))\longmapsto\int_{\U_{n,n}}\Theta(\M)\,\upxi$. In other words, for every $\upxi\in \uCH^{k-1}(\M)$,
    \[ \int_{\M} \upxi = \int_{\U_{n,n}} \Theta(\M) \, \upxi.\]
\end{lemma}

For the interested reader, let us mention that it is possible to prove Theorem~\ref{thm:flags-flats-rank-valuation} in full by relying on the preceding lemma. We make this point to emphasize the usefulness of relying on properties of the map $\Theta$ described above. It was pointed out to us by Chris Eur that a proof along these lines would follow by \cite[Corollary~5.11]{eur-huh-larson}.


Our next result shows that many valuative invariants factor through the map $\Theta$. 

\begin{teo}
    Let $E=[n]$ and consider any valuative invariant $f:\mathfrak{M}_E\to \mathrm{G}$ having the property that $f(\M) = 0$ whenever $\M$ has loops. Then, there exists a (necessarily unique) homomorphism $\widehat{f}:\uCH(\U_{n,n})\to \mathrm{G}$ such that
        \[ f = \widehat{f}\circ \Theta,\]
    where $\Theta$ is the map defined in Proposition~\ref{prop:can_chow_valuation}. 
\end{teo}

\begin{proof}
    Fix the invariant $f:\mathfrak{M}_E\to \mathrm{G}$ as in the statement. By Theorem~\ref{thm:g-invariant-universal}, we know that there exists a homomorphism $\widetilde{f}:\mathbf{U}\to \mathrm{G}$ such that $f = \widetilde{f}\circ \mathcal{G}$. On the other hand, by \cite[Theorem~3.17]{hampe} the restriction of the $\mathcal{G}$-invariant to loopless matroids factors through the restriction of the map $\Theta:\mathfrak{M}_E\to \uCH(\U_{n,n})$ to loopless matroids, in other words, we have a map $\widehat{\mathcal{G}}:\uCH(\U_{n,n})\to \mathbf{U}$ such that $\mathcal{G} = \widetilde{\mathcal{G}}\circ \Theta$ for all matroids without loops. Putting all the pieces together, we obtain 
        \[ f = \widetilde{f} \circ \widetilde{\mathcal{G}} \circ \Theta,\]
    for loopless matroids, and hence for all matroids (notice that a matroid with loops is in the kernel of the map $\Theta$, and also in the kernel of $f$ by assumption). The conclusion of the statement follows by taking $\widehat{f} = \widetilde{f}\circ\widetilde{\mathcal{G}}$.
\end{proof}

\tocless{\section*{Acknowledgements}}
\noindent
We are grateful for the feedback and comments we received from our colleagues that helped to improve various aspects of this manuscript. We benefited from stimulating discussions and suggestions of 
Christos Athanasiadis,
Joseph Bonin,
Chris Eur,
Alex Fink,
Michael Joswig,
Matt Larson,
Vic Reiner,
Jos\'e Samper,
Mario Sanchez
and Lorenzo Vecchi. We also thank two anonymous referees for their thorough comments and suggestions on earlier versions of our manuscript.

\bibliographystyle{amsalpha0}
\bibliography{bibliography}

\end{document}